\documentclass[11pt]{article}

\RequirePackage[OT1]{fontenc}
\RequirePackage{amsthm,amsmath,amssymb,amsfonts,bbm}
\RequirePackage[numbers]{natbib}
\RequirePackage[colorlinks,citecolor=blue,urlcolor=blue]{hyperref}
\usepackage{bm,bbm}
\usepackage{multirow}
\usepackage{graphicx,caption,subcaption} 

\usepackage{anysize}

\marginsize{1.1in}{1.1in}{0.55in}{0.8in}

\newtheorem{theorem}{Theorem}[section]
\newtheorem{proposition}[theorem]{Proposition}
\newtheorem{lemma}[theorem]{Lemma}

\newtheorem{remark}{Remark}[section]

\newtheorem{definition}{Definition}[section]

\DeclareMathOperator*{\esssup}{ess\,sup}

\newcommand{\RR}{\mathbb{R}}
\newcommand{\LL}{\mathbb{L}}
\newcommand{\bX}{\bm{X}}
\newcommand{\cC}{\mathcal{C}}
\newcommand{\cD}{\mathcal{D}}
\newcommand{\cN}{\mathcal{N}}
\newcommand{\cE}{\mathcal{E}}
\newcommand{\cS}{\mathcal{S}}
\newcommand{\cT}{\mathcal{T}}
\newcommand{\cF}{\mathcal{F}}
\newcommand{\cM}{\mathcal{M}}
\newcommand{\cL}{\mathcal{L}}
\newcommand{\Ind}{\mathbbm{1}}
\newcommand{\E}{\mathrm{E}}
\renewcommand{\P}{\mathrm{P}}
\newcommand{\norm}[1]{ \left\Vert #1 \right\Vert}

\begin{document}

\title{Bayesian Inference for \textit{k}-Monotone Densities with Applications to Multiple Testing
\thanks{
The research was supported in part by NSF Grant number DMS-1916419.
}}

\author{Kang Wang \thanks{Department of Statistics, North Carolina State Universtiy. Email: kwang22@ncsu.edu}
\and
Subhashis Ghosal\thanks{Department of Statistics, North Carolina State Universtiy. Email: sghosal@stat.ncsu.edu}}

\date{}
\maketitle

\begin{abstract}
\noindent
Shape restriction, like monotonicity or convexity, imposed on a function of interest, such as a regression or density function, allows for its estimation without smoothness assumptions. The concept of $k$-monotonicity encompasses a family of shape restrictions, including decreasing and convex decreasing as special cases corresponding to $k=1$ and $k=2$. We consider Bayesian approaches to estimate a $k$-monotone density. By utilizing a kernel mixture representation and putting a Dirichlet process or a finite mixture prior on the mixing distribution, we show that the posterior contraction rate in the Hellinger distance is $(n/\log n)^{- k/(2k + 1)}$ for a $k$-monotone density, which is minimax optimal up to a polylogarithmic factor. When the true $k$-monotone density is a finite $J_0$-component mixture of the kernel, the contraction rate improves to the nearly parametric rate $\sqrt{(J_0 \log n)/n}$. Moreover, by putting a prior on $k$, we show that the same rates hold even when the best value of $k$ is unknown. 
A specific application in modeling the density of $p$-values in a large-scale multiple testing problem is considered. Simulation studies are conducted to evaluate the performance of the proposed method.

\vskip 0.3cm

\end{abstract}

\section{Introduction}
\label{sec:introduction}

A regression or density function is typically estimated under the assumption of smoothness. However, in certain cases, it is natural to impose specific shape restrictions like monotonicity or convexity. For instance, in survival analysis, The density of the failure time can often be assumed to be nonincreasing. In such scenarios, a natural and sensible estimator should comply with the given shape restriction. This allows for estimation without relying on smoothness assumptions. One advantage of estimation methods based on shape restrictions instead of smoothness is that there is no need to select bandwidths of kernels or degrees of splines. Nonincreasing or convex nonincreasing probability densities naturally arise in several contexts, particularly in inverse problems where data is indirectly observed. For example, consider Hampel's bird migration problem (see, for instance, Section 4.3 of \cite{Groeneboom2014}). This problem involves estimating the distribution of the time a population of birds spends in an oasis, based on the time elapsed between two exact captures. By employing a Poisson process model for the number of captures, the distribution function of the unobserved sojourn time can be expressed in terms of the density function of the observed time interval between the captures. As a result, the density of the time interval is proven to be convex and non-increasing on the interval $(0, \infty)$.

The study of monotone nonincreasing densities was pioneered by \cite{Grenander1956}, who obtained the nonparametric maximum likelihood estimator, commonly known as the Grenander estimator. The pointwise asymptotic distribution of the Grenander estimator was subsequently derived by \cite{Rao1969} and was revisited by \cite{groeneboom1985} using a switch relation technique. The nonparametric maximum likelihood estimator of a nonincreasing density under random right-censoring was investigated by \cite{Huang1994, Huang1995}. Convergence rates and asymptotic distributions of the nonparametric maximum likelihood estimator under convexity have also been extensively studied, as shown in works such as \cite{Mammen1991, Groeneboom2001}, among others. One notable advantage of function estimation under shape restrictions is that it eliminates the need for selecting tuning parameters. The global shape restriction itself serves as a regularization mechanism for the estimator and helps denoise the observed data.

Both nonincreasing and convex nonincreasing density classes can be viewed as special cases of $k$-monotone density classes. Roughly speaking, the class consists of (a dense set of) functions whose odd-order derivatives are non-positive and even-order derivatives are non-negative, up to order $k$. A formal definition will be given in the next section. From \cite{Williamson1956}, a $k$-monotone density on $(0,\infty)$ can be represented as a scale mixture of scaled Beta($1,k$) kernels. Moreover, a density on $(0,\infty)$, which is $k$-monotone for every $k$, is infinitely smooth and can be expressed as a scale mixture of exponential densities (see \cite{Feller1971}, page 439). In \cite{Jewell1982}, Jewell used a mixture of exponential distributions to model the lifetime distribution and studied the maximum likelihood estimator. The classes of $k$-monotone densities for intermediate values of $k$ gradually transition from the class of monotone density (the largest class) to that of completely monotone densities (the smallest class). The class of 
$k$-monotone densities for a given integer $k$, as a potentially useful model in nonparametric shape-restricted inference, has received some attention.
The study of the nonparametric maximum likelihood estimator encompasses various aspects, including estimator characterization, convergence rates, and asymptotic distributions; see \cite{Balabdaoui2007, Balabdaoui2010, Gao2009}. It is worth noting that the notion of $k$-monotonicity need not be restricted only to the positive half-line. In some applications, such as modeling $p$-values in a multiple-testing framework, as discussed below, the unit interval is the domain where the density is defined. In such cases, the mixture representation changes, though. 

This article focuses on Bayesian approaches to $k$-monotone density estimation on the unit interval. By modifying the mixture representation introduced in \cite{Williamson1956}, we can adopt the well-established technique of Bayesian nonparametrics in mixture models. As no restrictions are imposed on the mixing distribution, the Dirichlet process prior, introduced by \cite{Ferguson1973}, can be considered, leading to a Dirichlet process scale-mixture of scaled beta kernels prior for the resulting $k$-monotone density. Dirichlet process mixtures with various kernels have been explored for Bayesian density estimation; see \cite{Ferguson1983, Lo1984, Brunner1989, escobar1995, Petrone1999}, among others. Markov chain Monte Carlo computational techniques have been invented to compute posterior characteristics; see \cite{escobar1995, maceachern1998estimating, Walker2007} and others. Recently developed faster computational methods such as variational algorithms (cf. \cite{Blei2006}) make the methodology appealing for practical implementation with large datasets and even allows computation in real time.
The posterior distribution based on a suitable Dirichlet process mixture prior concentrates near the true value of the density with high true probability under appropriate conditions as shown by \cite{Ghosal1999} for the normal kernel using the Schwartz theory of posterior consistency; see \cite{ghosal2017fundamentals} for details. An extension and a multivariate generalization were obtained respectively by \cite{tokdar} and \cite{Wu2010}. Rates of contraction of the posterior distribution of the Dirichlet mixture of normal prior have been established under various settings by \cite{Ghosal2001a, Ghosal2007, shen2013} using the general theory of posterior contraction rate \cite{Ghosal2000, ghosal2017fundamentals}. Other kernels were treated in the works of  \cite{Ghosal2001b, Petrone2002, Kruijer2008, Wu2008, Rousseau2010, salomond2014concentration}. 

Among shape-restricted inference problems, for the special case of monotone nonincreasing densities on the half-line, the density can be expressed as a mixture of the uniform distributions on $[0,\theta]$. This representation has been used for Bayesian estimation of a symmetric unimodal density by \cite{Brunner1989} and for a unimodal density by \cite{Brunner1992}, using Dirichlet process mixture priors, but no convergence results were obtained. The posterior contraction rate for  monotone nonincreasing densities on the half-line was obtained by \cite{salomond2014concentration}, who  considered the Dirichlet process mixture and finite mixture priors and obtained the optimal rate $n^{-1/3}$ up to some logarithmic factors. A challenge in applying the general theory of posterior contraction in this context is that the prior concentration condition on the Kullback-Leibler neighborhood around the true density in \cite{Ghosal2000} may not be satisfied since the support of the mixture kernel is only on a finite interval whereas the true density can spread out over the positive real half-line. To overcome these, the prior concentration condition was modified in \cite{salomond2014concentration} using a suitably 
truncated density class by dropping a negligible part. Another issue is that the metric entropy increases with the monotone density class's upper bound and support length. Salomond \cite{salomond2014concentration} also modified the metric entropy condition on a sieve with a growing upper bound to address the complication. 
A closely related work is by \cite{Martin2018}, who used empirical priors in a finite mixture model setting for the same problem. The optimal posterior contraction rate $n^{-1/3}$ in Hellinger distance up to a logarithmic factor was derived based on the theory of the empirical Bayesian approach in \cite{Martin2016}. Unlike \cite{Ghosal2000,ghosal2017fundamentals}, the condition requires a sufficient prior mass of a data-dependent Kullback-Leibler neighborhood centered at the (sieve) maximum likelihood estimator instead of the true density. Mariucci et al. \cite{Mariucci2020} obtained the near minimax-optimal posterior contraction rate for a log-concave density using an exponentiated Dirichlet process mixture prior. 

Shape-restricted densities arise naturally in certain multiple testing applications, 
such as large-scale simultaneous hypothesis testing of DNA microarray data. To assess and control the error rate in simultaneous testing, inference on the proportion of the true null hypotheses is instrumental.  Estimation procedures were developed based on the observed $p$-values for these tests in \cite{Storey2002, Storey2003, Langaas2005, TangGhosalRoy}. These methods use the facts that the $p$-values from true null hypotheses are distributed as uniform over the unit interval, at least approximately, while those from the alternative have density highly concentrated near zero and decaying sharply from there. In \cite{Langaas2005}, the density of the $p$-values is modeled as a monotone density which is $k$-monotone  with $k=1$ and $2$. Tang et al. \cite{TangGhosalRoy} used a mixture of beta densities with a singularity at zero, analogous to a completely monotone density (i.e. $k$-monotone with $k=\infty$), to model the density of $p$-values and put a Dirichlet process prior on the mixing distribution. This application motivates the study of $k$-monotone densities on the unit interval with an additional uniform component. In this framework, using the mixture representation of a $k$-monotone density and putting either a finite mixture or a Dirichlet process mixture prior, we shall obtain the posterior contraction rates using the general theory of posterior contraction in \cite{Ghosal2000} with respect to the Hellinger metric. The presence of the additional uniform component makes the densities bounded away from zero, which is instrumental in controlling the Hellinger distance and the Kullback-Leibler divergence. We obtain the minimax-optimal rate $n^{-k/(2k+1)}$ up to a polylogarithmic factor for any given value of $k$, thus providing a spectrum of rates corresponding to different regularity guided by the shape-hierarchy, analogous to the smoothness hierarchy. If the true density is a $J_0$-mixture of the kernel in the mixture representation of a $k$-monotone density, we also show that the same procedure achieves the nearly parametric rate $\sqrt{(J_0 \log n)/n}$, which is difficult to achieve using a non-Bayesian method. By putting a suitable prior on $k$, the Bayesian approach can adapt to the optimal posterior contraction rate as if the best $k$ were known. This is a significant merit of the proposed Bayesian procedure, as to the best of our knowledge, the existing methods are all for a given value of $k$. A mathematical relationship in a statistical inverse problem may imply the $k$-monotone shape (\cite{Groeneboom2014, Balabdaoui2007}), but 
in applications such as \cite{Langaas2005}, the best $k$ fitting the data may not be clear to the researcher. The proposed method can automatically select the best underlying $k$ while maintaining the optimal convergence rate.

The organization of this paper is as follows. In the next section, we introduce the notion of a $k$-monotone density, characterize it through a mixture representation, and present an important approximation result using finite mixtures.
In Section~\ref{sec:rate}, we introduce the prior and present results on the posterior contraction rates. In Section \ref{sec:adaptive}, we consider Bayesian estimation with unknown $k$. The application to multiple testing is discussed in Section \ref{sec:Multi-testing}. We present a simulation study in Section \ref{sec:simulation}. 
The main conclusions are summarized in Section \ref{sec:Disc}.
Proofs are postponed to Section \ref{sec:proof} and the appendix.

\section{Preliminaries}
\label{sec:defkmono}


We shall use the following notations throughout the paper. 
Let $\RR$ be the set of real numbers and  
$\Delta_J$ be the unit $J$-simplex, $J=1,2,\ldots$.
For a set $A$, let $\Ind_A$ stand for the indicator of $A$. 
Let $f_+ = \max(f, 0)$ denote the positive part of the function $f$, and $f(x-)$ (respectively, $f(x+)$)  denote the left (respectively, right) limit of $f$ at $x$ when it exists.
If $\int |f|^p < \infty$, we use $\norm{f}_p$ to denote the $\LL_p$-norm, and $\norm{f}_{\infty}$ to denote the essential supremum $\esssup|f|$ if $f$ is measurable and bounded almost everywhere.
For $1\le p \le \infty$, the space of $p$-integrable functions on the domain $A$ is denoted by $\LL_p (A)$. Additionally, we define $d_p(f, \cS) = \inf\{\norm{f-s}_p:s \in \cS\}$ for $f\in \LL_p (A)$ and $\cS\subseteq \LL_p (A)$. The Dirac delta measure at $\theta$ will be denoted by $\delta_\theta$. The notation $\mathrm{Dir}(J; \omega_1,\ldots,\omega_J)$ stands for the Dirichlet distribution on the probabilistic $J$-simplex with parameters $\omega_1,\ldots,\omega_J$. 
For a probability measure $P$, absolutely continuous with respect to the Lebesgue measure, we denote its density by the corresponding lowercase letter $p$.  
The Hellinger distance between two densities $p_1$ and $p_2$ is defined by $d_H(p_1, p_2) = \norm{\sqrt{p_1} - \sqrt{p_2}}_2$. The Kullback-Leibler divergence and Kullback-Leibler variation are respectively given by $K(p_1, p_2) = \int p_1 \log(p_1/p_2)$ and $V(p_1, p_2) = \int p_1 [\log(p_1/p_2)]^2$.
For a semimetric space $(\cT, d)$, the metric entropy refers to the logarithm of the covering number $\cN(\epsilon, \cT, d)$, while the bracketing entropy refers to the logarithm of the bracketing number $\cN_{[\,]}(\epsilon, \cT, d)$; 
see Section 2.1 of \cite{vanderVaart1996} for details. 
For two real positive sequences $\{a_n: n\ge 1\}$ and $\{b_n: n\ge 1\}$,
the notation $a_n\lesssim b_n$ (equivalently, $b_n\gtrsim a_n$) means that $a_n \le C b_n$ for some constant $C>0$. We say $a_n \asymp b_n$ if $a_n \gtrsim b_n$ and $a_n \lesssim b_n$.

\begin{definition}[$k$-monotonicity]
\label{def:k-monotone}
Let $I$ be subinterval of $(0,\infty)$. A function $f$ on $I$ is said to be $1$-monotone on $I$ if $f$ is nonnegative and nonincreasing. For $k \ge 2$, $f$ is said to be $k$-monotone on $I$ if $(-1)^j f^{(j)}$ is nonnegative, nonincreasing and convex on $I$, for every $j = 0, \ldots, k- 2$. 
\end{definition}

Let the class of $k$-monotone functions on $I$ be denoted by $\cF^k_I$. The class of $k$-monotone probability densities on $I$ will be denoted by $\cD^k_I = \{g \in \cF^k_I: \int g = 1\}$. We shall be concerned about $k$-monotone functions on a bounded interval, which can be taken to be the unit interval $(0,1)$ without loss of generality. Since the domain is fixed at $(0,1)$ throughout, we shall drop $(0,1)$ from the notations  $\cF^k_{(0,1)}$ and $\cD^k_{(0,1)}$, and simply write $\cF^k$ and $\cD^k$ respectively. 

A closely related concept is $k$-convexity, which is sometimes referred to as $k$-monotonicity by some authors in the approximation theory literature. There are multiple ways to characterize $k$-convex functions, and we present an equivalent definition in the following.

\begin{definition}[$k$-convexity]
\label{def:k-convex}
A function $f:(0,1)\to \RR$ is said to be $1$-convex on $(0,1)$ if $f$ is nondecreasing, while for $k\ge 2$, $f$ is said to be $k$-convex on $(0,1)$ if $f^{(k-2)}$ exists and is convex on $(0,1)$. We shall write $\cC^k$ for the space of $k$-convex functions on $(0,1)$.  
\end{definition}

Introduce a probability density function 
\begin{align}\label{psi_k}
     \psi_k(x,\theta)= \frac{k}{\theta}\left(1-\frac{x}{\theta}\right)_+^{k-1}, \text{ for }  x>0, \; \theta > 0.
\end{align}
Note that $\psi_{k}(\cdot, 1)$ is the probability density function of Beta$(1,k)$. The following result shows that $k$-monotone functions and densities on $(0,1)$ admit a useful mixture representation using the kernel $\psi_k$.  

\begin{lemma}[Characterization of $k$-monotone functions and densities on $(0,1)$]
\label{lemma:charac}
A function $f\in \cF^k$ if and only if there exist a  nondecreasing function $\gamma(t)$ on $(0,1)$ and $\alpha_j\ge 0$ for $j=0,1,\ldots, k-2$, such that, for $x\in(0,1)$,
\begin{align}
\label{repre_kmofunc}
    f(x) = \sum_{j=0}^{k-1}  \alpha_j(1-x)^j + \int_0^1 \psi_k(x, t) d \gamma(t).
\end{align}
A density $g\in \cD^k$ if and only if there exists a probability measure $Q$ and $(\beta_j: 0\le j \le k)\in \Delta_{k+1}$ such that, for every $x\in (0,1)$,
\begin{align}
\label{kmonodens}
    g(x) = \sum_{j=0}^{k-1} \beta_j\psi_{j+1}(x,1) + \beta_k \int_0^1 \psi_k(x, \theta) d Q(\theta).
\end{align}
\end{lemma}
The proof of this lemma is based on Taylor expansion and further integration by parts of the integral form remainder term. Similar results can be found in \cite{Gao2008} for a $k$-monotone distribution function on a compact interval and \cite{Williamson1956} for $k$-monotone function on the positive half real line. We defer the proof of the lemma to the appendix.

A crucial property of $k$-monotone functions for deriving posterior contraction rates is that they can be approximated effectively by $k$-monotone free-knot spline functions in the $\LL_p$-metric, $1\le p<\infty$. This property is derived from Theorem 1.1 of \cite{Kopotun2003} on a shape-preserving approximation of $k$-convex functions.

Let $\cS_{N,k}$ denote the space of free knot splines of degree $k-1$ with $N$ interior knots in $[0,1]$. 
To align with the $k$-convex functions, we introduce a reflection transformation to the argument. 
Let $\tau(x) = 1-x$ for $x\in(0,1)$ and denote $\check{\cF}^k = \{f\circ \tau: f\in \cF^k\}$. 
Then shape preserving approximation to $\cF^k$ is essentially the same problem of shape-preserving approximation to $\check{\cF}^k$.
By Definition \ref{def:k-monotone}, for $k\ge 2$, $f\in \check{\cF}^k$ if and only if $f^{(j)}$ is nonnegative, nondecreasing, and convex, for every $j=0,1,\ldots, k-2$. 
It is then clear that $\check{\cF}^k$ is a subclass of $\cC^k$. Moreover, for $h \in \cC^k$ and $k\ge 2$, let $h^{(k-1)}$ denote the right derivative of $h^{(k-2)}$, which is well defined since $h$ is a convex function on $(0,1)$. We also know that $h^{(k-1)}$ is right continuous. It is not hard to see that $h\in \check{\cF}^k$ as well if $h^{(j)}(0+) \ge 0$, for $j=0,\ldots, k-1$. Indeed, $h^{(k-1)}(0+)\ge 0$ implies that $h^{(k-2)}$ is nondecreasing, and furthermore, it is known that $h^{(k-2)}$ is nonnegative, nondecreasing, and convex. Continuing in the same way, we know that $h^{(j)}$ is nonnegative, nondecreasing, and convex for all $j = 0,\ldots, k-2$, that is, $h\in \check{\cF}^k$ by definition. In view of this point, for $f\in \check{\cF}^k \subset \cC^k$, the shape preserving approximation by a free knot spline function $s\in \cS_{N,k}\cap \cC^k$ considered in \cite{Kopotun2003} is also a shape preserving approximation in $\check{\cF}^k$ (i.e. $s\in \cS_{N,k}\cap \check{\cF}^k$) provided that $s^{(j)}(0)\ge 0$ for all $j=0,\ldots, k-1$. By close inspection of the construction of approximating function in \cite{Kopotun2003}, this set of conditions is naturally satisfied. We leave the details of the argument in the appendix. 

The main result in \cite{Kopotun2003} states that the shape-preserving approximation by free knot splines can be as good as the free knot spline approximation regarding the number of splines used to construct the approximation function and the approximation error. In fact, 
Theorem 1.1 of \cite{Kopotun2003} presents a more general result.
In what follows, we only use their result with the order of free knot splines fixed at $k$ (i.e. approximation by piecewise polynomials of degree $k-1$) as this is the only case of interest in the current work.

\begin{proposition}[Theorem 1.1 of \cite{Kopotun2003}]
\label{kconvapprox}
 For any $1\le p \le \infty$ and any $f\in \cC^k\cap \LL_p(0,1)$, there exist constants $C_k>0$ and $C_{k,p}>0$ such that 
 \begin{align*}
     d_p(f, \cS_{C_k N,k}\cap \cC^k)\leq C_{k,p} d_p(f, \cS_{N, k}).
 \end{align*}
\end{proposition}

On the other hand, the approximation error of free-knot splines is well-studied in approximation theory, as can be found in Chapter 12 of \cite{DeVore1993}. If the $(k-1)$-th derivative of $f$ is bounded, the right-hand side of the last display is bounded by $N^{-k}$ up to some positive constant. Moreover, the shape-preserving approximation to a $k$-monotone function is not hard to be adapted to the shape-preserving approximation to the $k$-monotone density. With the help of Lemma \ref{lemma:charac}, the free knot spline approximation of order $k$ with $N$ interior knots admits a representation as in \eqref{kmonodens}, indicating that the mixing distribution $Q$ is supported on a set of at most $N$ points.

To summarize, we obtain the following approximation result, whose proof is deferred to the appendix.

\begin{lemma}
\label{lemma:approx}
Let $g\in \cD^k$ be given by \eqref{kmonodens} such that $|g^{(k-1)}(0+)| < \infty$. Then there exists a discrete probability measure $Q_N$ with $N$ support points in $(0,1)$ such that with 
\begin{align}
\label{approximate density}
g_N(x) = \sum_{j=0}^{k-1} \beta_j\psi_{j+1}(x,1)  + \beta_k\int_0^1 \psi_k(x, \theta) d Q_N(\theta) \in \cD^k,
\end{align}
we have that $\|g-g_N\|_{\infty}\leq CN^{-k}$ for some constant $C>0$.
\end{lemma}


\section{Posterior Contraction Rates}
\label{sec:rate}

Let $\bX_n=(X_1,\ldots, X_n)$ be independent and identically distributed (i.i.d.) samples from a $k$-monotone density $g$ given by the representation \eqref{kmonodens} for a known $k=1,2,\ldots$. To place a prior on $g$, it is natural to consider independent priors for the coefficient vector $\bm{\beta}=(\beta_0,\ldots,\beta_k)$ and the mixing distribution $Q$. 

We put a Dirichlet distribution prior on $\bm{\beta}$ with parameters $0< a_j <\infty$, for all $j=0,1,\ldots, k$.
Independently of $\bm{\beta}$, we assign either a Dirichlet process (DP) prior or a Finite Mixtures (FM) prior on $Q$: 
\begin{itemize}
\item[{DP}:]  $Q\sim \mathrm{DP}_{a H}$, where $a>0$ is the precision parameter and $H$ is the center measure supported on $(0,1)$; see \cite{ghosal2017fundamentals} for definitions; 
\item[{FM}:] $Q=\sum_{j=1}^J w_j \delta_{\theta_j}$, with  
$\theta_1,\ldots,\theta_{J}|J \overset{i.i.d.}{\sim} H$, and $(w_1,\ldots,w_J)|J\sim \mathrm{Dir}
(J;\omega_{1J},\ldots,\omega_{JJ}),$ independently. $J$ is given a prior $\Pi(J) = (n^c-1) n^{-cJ}$ on the set of positive integers. 
\end{itemize}
In the above priors, $a$, $H$, $(\omega_{jJ}: 1\le j\le J<\infty)$ and $c$ are hyperparameters. We assume that 
\begin{itemize}
    \item[(C1)] $H$ admits a Lebesgue density $p_H$ on $(0,1)$ such that in a small neighborhood of zero, $p_H(\theta) \lesssim \theta^{t_1}$ for some $t_1> 0$;
    \item[(C2)] for any interval $(u,v)\subset (0,1)$ and some $t_2>0$, $H((u,v))\gtrsim (v-u)^{t_2}$. 
\end{itemize}

If $g\in \cD^k$ for $k\ge 2$, it is assumed that $g$ is differentiable only up to order $k-2$. However, $(-1)^{k-2}g^{(k-2)}$ is convex and non-increasing on $(0,1)$. Hence, we can define $g^{(k-1)}$ uniquely almost everywhere as either the left or right derivative of $g^{(k-2)}$, which are equal except possibly on an at most countable set.

\begin{theorem}[Contraction rate for Dirichlet process mixture prior]
\label{contraction}
Let the data $\bX_n$ be generated from a $k$-monotone density $g_0$ on $(0,1)$ given by 
\begin{align}
\label{true density}
    g_0(x) = \sum_{j=0}^{k-1} \beta_{0,j}\psi_{j+1}(x,1) + \beta_{0,k}\int \psi_k(x, \theta) d Q_0(\theta),
\end{align}
where $k$ is known. We assume $g_0^{(k-1)}(0+)<\infty$ and $\beta_{0,0} > 0$. Let $\bm{\beta}$ be given a Dirichlet prior with positive constant parameters, $a_0,\ldots,a_k$, and independently, 
put a Dirichlet process prior on $Q$ satisfying Conditions (C1) and (C2). Then the posterior distribution of $g$ contracts at the rate $\epsilon_n=(n/\log n)^{-k/(2k+1)}$ at $g_0$ with respect to the Hellinger distance, i.e., $\E_0 [\Pi(d_H(g,g_0) \ge M_n \epsilon_n |\bX_n)] \to 0$ for any $M_n\to \infty$.
\end{theorem}

The same posterior contraction rate can be obtained by using a finite mixture prior on the mixing distribution, as presented in the following theorem. 

\begin{theorem}[Contraction rate for finite mixture prior]
\label{thm:fmix}
Let the data $\bX_n$ be generated from a $k$-monotone density $g_0$ on $(0,1)$  given by \eqref{true density} with a known $k$, satisfying $g_0^{(k-1)}(0+)<\infty$ and $\beta_{0,0} > 0$. Let $\bm{\beta}$ be given the Dirichlet prior with positive constant parameters $a_0,\ldots,a_k$, and independently, put a finite mixture prior for $Q$ satisfying Conditions (C1) and (C2), with $c>0$ chosen sufficiently large. Then the posterior distribution of $g$ contracts at $g_0$ at the rate $\epsilon_n=(n/\log n)^{-k/(2k+1)}$ with respect to the Hellinger distance.
\end{theorem}

\begin{remark}\rm 
\label{rem:fixed prior}
In Theorem \ref{thm:fmix}, the same posterior contraction rate can be derived if the prior on $J$ is replaced by a fixed prior that satisfies the condition,
$e^{-b_1 j \log j}\le \Pi(J=j)\le e^{-b_2 j \log j}$. For instance, a Poisson prior truncated at $0$ satisfies the required tail condition. 
\end{remark}

The posterior contraction rate is substantially improved to a nearly parametric rate using the same prior if the mixing distribution $Q_0$ is finitely supported on $J_0$ points, i.e., 
\begin{align}
    g_0(x) = \sum_{j=0}^{k-1}\beta_{0,j}\psi_{j+1}(x,1) + \beta_{0,k}\sum_{l=1}^{J_0} w_l^0 \psi_k(x,\theta_j^0). \label{dens.fmix}
\end{align}

In the result below, both $k$ and $J_0$ are allowed to depend on $n$ (and hence the resulting rate involves $k$ and $J_0$), provided that $\max(\log k,\log J_0)\lesssim \log n$. 

\begin{theorem}[Finitely supported mixing]
\label{thm:finite}
    Let the true density $g_0$ as given in \eqref{dens.fmix} with $\beta_{0,0}>0$. Let $\bm{\beta}$ be given the Dirichlet prior with positive constant parameters $a_0,\ldots,a_k$, and independently, let 
    $Q$ be given a finite mixture prior with $c>0$ chosen sufficiently large and $H$ satisfying the conditions that $H((u,v)) \gtrsim (v-u)^{t_2}$ and $p_{H} \lesssim \exp\{-t_3/\theta\}$ 
    for any interval $(u,v)\subset(n^{-2}, 1)$, and $\theta\in (0, n^{-2})$, where $t_2,t_3>0$ are constants. 
    Then the posterior of $g$ contracts at $g_0$ at the rate $\epsilon_n = \sqrt{\max(k, J_0) (\log n )/n}$ with respect to the Hellinger metric.
\end{theorem}

It can be seen from the proof that the fixed prior in Remark~\ref{rem:fixed prior} also obtains the same rate if $\log J_0 \asymp \log n$.  

\section{Adaptation to \textit{k}}
\label{sec:adaptive}

In the last section, we studied posterior contraction rates assuming that the order of monotonicity $k$ is known. Here, the parameter $k$ serves as a regularity index controlling the complexity of the model, much like a smoothness index. Adapting the rates to different values of $k$ is therefore a highly desirable objective. In the Bayesian framework, a natural approach is to treat $k$ as a model index parameter and put a prior distribution on it. Consequently, the resulting prior becomes a mixture of the priors used for a fixed index value. Under similar situations in smoothness or sparsity settings, the corresponding posterior distribution often adapts to the optimal rate under fairly mild conditions; see \cite{ghosal2008nonparametric,ghosal2017fundamentals,castillo2015bayesian}, among others. In this section, we show that such an automatic adaptation strategy works in the $k$-monotone setting as well. This feature is particularly attractive by employing the Bayesian approach, while no parallel result is known in the non-Bayesian literature for the family of $k$-monotone densities indexed by $k$.
Noting the models $\cD^k$ are nested in the following way: $\cD^{k+1}\subset \cD^k$, we define the true value $k_0$ as the largest value of $k$ such that $g_0 \in \cD^k$. We assume $k_0$ is finite, which is the case of interest. It is not hard to see that the finiteness of $k_0$ implies $\beta_{k_0} > 0$ in the characterization \eqref{kmonodens}. Otherwise, this $k_0$-monotone density would be a polynomial of degree at most $k_0-1$, which would correspond to the case $k_0 = \infty$, contradicting the finiteness of $k_0$.

Let $k$ be given a prior $\Pi$ that is one of the two types: 
\begin{itemize}
    \item [(K1)]
    $e^{ - d_1 k\log k} \le \Pi(k) \le e^{ - d_2 k\log k}$ for some $d_1\ge d_2>0$; 
    \item [(K2)]  $\Pi(k) = (n^r-1) n^{-rk}$, for some $r>0$.
\end{itemize}

\begin{theorem}[Adaptive contraction rate]
\label{thm:adp}
    Let the monotonicity index $k$ be unknown and endowed with a prior satisfying the conditions (K1)  or (K2). Given $k$, let the prior for $\bm{\beta}$ be  $\mathrm{Dir}(a_0,\ldots,a_k)$ for some $a_0,\ldots,a_k$ lying between two fixed positive numbers, and independently, let $Q$ be given either the Dirichlet process prior or the finite mixture prior with a sufficiently large $c>0$, satisfying Conditions (C1) and (C2).  Let the true density be $g_0$ be given by \eqref{true density} with $k=k_0$ satisfying $|g_0^{(k_0-1)}(0+)| < \infty$ and $\beta_{0,0}>0$. 
    Then the posterior distribution contracts at $g_0$ at the rate 
    $\epsilon_n = (n/\log n)^{-k_0/(2k_0 + 1)}$ with respect to the Hellinger metric. 
\end{theorem}

If the true $k$-monotone density has a finite representation as in \eqref{dens.fmix} with $k=k_0$, then it is still possible to obtain the nearly parametric posterior contraction rate stated in Theorem~\ref{thm:finite} without knowing the true value $k_0$ of $k$. As shown in Theorem~\ref{thm:finite}, this holds when both $k_0$ and $J_0$ satisfy that $\max(\log J_0,\log k_0)\lesssim \log n$. It may be noted that, even though  $g\in \cD^{\bar{k}}$ for any $\bar{k}<k_0$ as well, the corresponding mixture representation with the kernel $\psi_{\bar{k}}$ will not be supported on finitely many points in general.

\begin{theorem}[Adaptive contraction for finite mixture]
     Let the monotonicity index $k$ be unknown and endowed with a prior of the type (K2). Given $k$, let the prior for $\bm{\beta}$ be  $\mathrm{Dir}(a_0,\ldots,a_k)$ for some $a_0,\ldots,a_k$ lying between two fixed positive numbers. Independently, let $Q$ be given the finite mixture prior with a sufficiently large $c>0$. The prior distribution $H$ of a support point $\theta$ satisfies, for some $t_2,t_3>0$, that $H((u,v)) \gtrsim (v-u)^{t_2}$ for any interval $(u,v)\subset(n^{-2}, 1)$, and the corresponding density $p_{H}(\theta)  \lesssim \exp\{-t_3/\theta\}$ for all $\theta\in (0, n^{-2})$. 
    If $g_0$ is given by \eqref{dens.fmix} for some finite $J_0$ and $k=k_0$ with $\beta_{0,0}>0$, then the posterior contracts at $g_0$ at the rate $\sqrt{\max(k_0, J_0) (\log n) / n}$ with respect to the Hellinger metric.
\end{theorem}

It can be seen from the proof that a prior of the type (K1) for $k$ and a fixed prior for $J$ as in Remark~\ref{rem:fixed prior} may also be used to derive the same rate provided that $\log k_0 \asymp \log n$ and $\log J_0\asymp \log n$.

\section{Applications to Multiple Testing}
\label{sec:Multi-testing}

In large-scale hypothesis testing, it is essential to assess the proportion of true null hypotheses when reporting scientific findings. The proportion of null hypotheses, denoted as $\alpha$, plays a crucial role in the calculation of the positive false discovery rate \cite{Storey2002}.
Consider a problem of simultaneously testing $n$ hypotheses. For each individual test, the data are summarized using a test statistic, and a $p$-value is computed based on an exact, approximate, or asymptotic null distribution of the test statistic and the scope of the alternative hypothesis. Furthermore, we assume that the test statistics corresponding to different hypotheses are (nearly) independent, resulting in (nearly) independent $p$-values.
Under a simple null hypothesis, the $p$-value is calibrated; that is, it has a uniform distribution on $[0,1]$, provided that the test statistic follows a continuous null distribution. Even when the null hypothesis is composite, certain Bayesian $p$-values (e.g., the partial posterior predictive $p$-value of \cite{bayarri2000p}) asymptotically follow a uniform distribution (cf. \cite{robins2000asymptotic}) when the data are sampled using an i.i.d. scheme. 
The $p$-values from the alternative hypotheses usually concentrate near the origin and have decreasing density on $[0,1]$. This feature, along with true null hypotheses outnumbering true alternative hypotheses in practice, is used to estimate the proportion of null hypotheses in Storey's procedure \cite{Storey2002} for controlling the positive false discovery rate (pFDR). It is easy to see that the proportion of the null hypothesis is identifiable if the $p$-value density under the alternative approaches $0$ at $1$ (Proposition~4 of \cite{Ghosal2008}). 
This assumption is not always true; however, see the discussion in Section 2.2 of \cite{Ghosal2008}. For example, in the two-sided t-test, the density of $p$-values does not vanish at $1$, in which case we can only identify an upper bound for the proportion of null hypotheses. However, if the sample size is reasonably large, the height of the density under the alternative is very small, so the condition holds approximately. 

The $p$-value density under the alternative is explicitly modeled as a monotone decreasing density ($k$-monotone for $k=1$) in \cite{Langaas2005}. This assumption is extremely mild as it can be seen to hold under the Monotone Likelihood Ratio (MLR) property of the distribution of the test statistic for both one- and two-sided alternatives (Propositions~1 and 2 of \cite{Ghosal2008}). However, simulation results demonstrate that the Grenander estimator exhibits unstable performance near $1$, which significantly affects the quality of the estimator for the positive false discovery rate (pFDR). To enhance the performance, \cite{Langaas2005} recommends using a convex nonincreasing density to fit the density of the $p$-values.
A model-based Bayesian approach to the estimation of the pFDR was adopted in \cite{TangGhosalRoy} using certain mixtures of beta densities. The corresponding distribution function under a logarithmic transformation of the argument is completely monotone (Proposition~7 of \cite{Ghosal2008}), which corresponds to $k$-monotonicity for all $k$. Results in Section~3 of \cite{Ghosal2008} show that the Bayesian procedure under a Dirichlet process prior on the mixing distribution gives consistent posterior for the proportion of null hypotheses and the pFDR. Other model-based and Bayesian approaches to the estimation of pFDR have been proposed based on modeling probit-transformed mixtures of skew-normal densities in \cite{bean2013finite} and \cite{ghosal2011predicting} and sufficient conditions for the identification of the proportion of null hypotheses are discussed in \cite{ghosal2011identifiability}. A review of Bayesian nonparametric methods for multiple testing is available in \cite{ghosal2009bayesian}. 

A very appealing condition on the $p$-value density under the alternative compromising between the generality of the class of monotone densities and the smoothness of the class of completely monotone functions is that the density of the $p$-value distribution under the alternative belongs to the class of $k$-monotone density for some $k$. For instance, the case $k=2$ corresponding to decreasing convex densities already gives a much more stable estimator of the density \cite{Langaas2005}, but it may be harder to ensure under what condition the density of $p$-values under the alternative would be decreasing and convex. The approach to modeling the density of the $p$-values under the alternative as a $k$-monotone density is irresistibly appealing if $k$ can be left unspecified and be adaptively chosen from the data using the technique developed in Section~\ref{sec:adaptive}. The following result quantifies the accuracy of the procedure.

\begin{theorem}
\label{cor:alp}
Let $U_1,\ldots,U_n$ be independent $p$-values arising from the simultaneous testing of $n$ hypotheses. We assume that the $p$-value density $g$ is modeled as $k$-monotone, where $k\ge 2$. The value of $k$ can be either known or unknown. In the latter case, a prior on $g$ is specified as described in Section~\ref{sec:rate} or \ref{sec:adaptive}. In both scenarios, $\alpha$ represents the corresponding proportion of null hypotheses. Let $g_0$ stand for the true density and let the true proportion of null hypotheses be denoted by $\alpha_{0}$. Then under the conditions of Theorems \ref{contraction}, \ref{thm:fmix} or \ref{thm:adp}, the posterior distribution of $\alpha$ is consistent at $\alpha_0$ and contracts at the rate $\epsilon_n=(n/\log n)^{-k/(2(2k+1))}$, that is, for any $M_n\to \infty$, $\Pi(|\alpha - \alpha_{0}|>M_n \epsilon_n|U_1,\ldots,U_n)\to 0$ in probability.
\end{theorem}

\section{Simulation Study}	
\label{sec:simulation}

We implement the proposed Bayesian approach for a $k$-monotone density estimation. Specifically, we employ the Dirichlet process prior for the mixing distribution. To simplify, we only retain the additional uniform component and the mixture component of k-monotone kernels in computation. We consider both scenarios: when the value of $k$ is known and when it is unknown. Simulation results demonstrate the superiority of our method compared to nonparametric maximum likelihood estimation for monotone density, as well as for convex and monotonically nonincreasing density. We present the specifics of our simulation in the following sections.

\subsection{Estimation accuracy}

To perform posterior sampling under the Dirichlet process mixture prior, we utilize the sliced Gibbs sampling algorithm as described in \cite{Kalli2011}. We use a uniform base measure 
on $[n^{-1},1]$ for the Dirichlet process prior and set the precision parameter to a fixed value $1$. For the simplified model, we set $\beta_{j} = 0$ for $j = 1, \ldots, k-1$, and give the proportion of the uniform component $\beta_0$ a uniform prior on $[0,1]$. 
We let $k$ be fixed or assign an appropriate prior for $k$. In particular, when using the adaptive Bayesian approach, the prior on $k$ is uniformly distributed over the set ${1, \ldots, 10}$. We select the largest $k$ at $10$, which is sufficiently large to approximate common smoothly decreasing densities of interest.
In the following, we generate $1000$ posterior samples, based on which we make inferences on the unknown density function after dropping the first $2000$ burn-in ones in every Bayesian application.

Let the sequence $\theta_{j,J} = j/J$. 
We consider the following density functions:
\begin{itemize}
    \item $g_1(x) = \psi_2(x, 1) = 2(1-x)$,
    \item $g_2(x) = 0.5f_1(x) + 0.5 = 1.5 - x$,
    \item $g_3(x) = \sum_{j=1}^3 3^{-1}\psi_2(x,\theta_{j,3})$,
    \item $g_4(x) = 0.5 f_3(x) + 0.5$,
    \item $g_5(x) = \sum_{j=1}^3 3^{-1}\psi_4(x,\theta_{j,3})$,
    \item $g_6(x) = \int_0^1 \psi_4(x,\theta) 2\theta d\theta$. 
\end{itemize}
For $g_6$, the mixing distribution for $\theta$ is Beta(2, 1), and sampling according to $g_6$ is straightforward.

We take sample sizes of $n = 100, 200$, and $500$. For every sample size, we generate independent and identically distributed samples with all six aforementioned models. The proposed Bayesian procedure is applied to each dataset, accounting for both known and unknown values of $k$. 
To compare with non-Bayesian methods, we apply the posterior mean density function as the Bayesian estimator. We consider the classical Grenander estimator, as well as the nonparametric maximum likelihood estimator, for convex and nonincreasing densities on the interval $[0,1]$. To measure the deviation from the true density functions, for each estimate $\hat{g}$, we compute the mean squared error (MSE) over a grid in $[0, 1]$. This grid is defined as $x_{j,K} = j/K$ for $j=1, \ldots, K$. The MSE is then calculated as follows:
\begin{align*}
    \text{MSE}(\hat{g}) = \frac{1}{K}\sum_{j=1}^K(\hat{g}(x_{j,K}) - g_0(x_{j,K}))^2.
\end{align*}
Here we choose $K=100$ and $f_0$ stands for the corresponding $f_i$, $i=1,\ldots,6$.

We independently conduct $R=500$ iterations for each setup and present the average mean squared error (MSE) calculated in Table \ref{tb:mse}. Each row in Table \ref{tb:mse} corresponds to a specific method applied to the dataset with the corresponding sample size. ``Bay'' denotes the Dirichlet mixture model with a known $k$. ``Ada'' represents the Bayesian methods where $k$ is unknown. ``Con'' and ``Gre'' stand for the nonparametric maximum likelihood estimation for the convex and nonincreasing density class and for the nonincreasing density class, respectively.

\begin{table}[ht]
\caption{Average of MSE.} 
\label{tb:mse}
\centering
\begin{tabular}{crrrrrrrr}
  \hline
 & & $g_1$ & $g_2$ & $g_3$ & $g_4$ & $g_5$ & $g_6$\\ 
  \hline
\multirow{4}{*}{$n=100$} & Bay & 0.018 & 0.018 & 0.027 & 0.018 & 0.029 & 0.028 \\ 
& Ada & 0.024 & 0.023 & 0.027 & 0.026 & 0.030 & 0.031 \\
& Con & 0.019 & 0.022 & 0.041 & 0.032 & 0.068 & 0.076 \\
& Gre & 0.058 & 0.047 & 0.097 & 0.068 & 0.158 & 0.162 \\
\hline
\multirow{4}{*}{$n=200$} & Bay & 0.009 & 0.011 & 0.017 & 0.011 & 0.021 & 0.017 \\
& Ada & 0.014 & 0.013 & 0.016 & 0.014 & 0.019 & 0.017 \\ 
& Con & 0.010 & 0.011 & 0.024 & 0.017 & 0.040 & 0.041 \\  
& Gre & 0.036 & 0.029 & 0.058 & 0.041 & 0.102 & 0.102 \\ 
\hline
\multirow{4}{*}{$n=500$} & Bay & 0.003 & 0.005 & 0.008 & 0.006 & 0.010 & 0.010 \\ 
& Ada & 0.003 & 0.006 & 0.008 & 0.007 & 0.014 & 0.015 \\
& Con & 0.004 & 0.005 & 0.010 & 0.008 & 0.018 & 0.020 \\ 
& Gre & 0.018 & 0.015 & 0.029 & 0.022 & 0.052 & 0.053 \\ 
   \hline
\end{tabular}
\end{table}

In summary, the proposed Bayesian methods for both known and unknown values of $k$ demonstrate superiority over nonparametric likelihood estimations. The adaptive Bayesian approach performs nearly as well as the Bayesian method that employs the optimal choice of $k$.

\subsection{Estimation of the proportion of null hypotheses}

The simulation setup in this part closely follows that of \cite{Langaas2005}.
Here, we simulate DNA microarray data that involves multiple hypothesis testing problems. For each of the $m$ individuals, we collect a dataset of sample size $n$, denoted by $\bm{X}_j = (X_{1,j}, \ldots, X_{n, j})$ for $j=1,\ldots, m$, independently drawn from a multi-normal distribution, i.e., $\bm{X}_j\overset{i.i.d.}{\sim}\mathrm{N}(\bm{\mu}, \Sigma)$ where $\bm{\mu} = (\mu_1, \ldots, \mu_n)$. We test the hypotheses 
\begin{align*}
    H_{0,i}: \mu_i = 0, \text{ versus }H_{0,i}: \mu_i \neq 0,
\end{align*}
based on the t-test statistics. For comparison, we also consider the set of one-sided t-tests,
\begin{align*}
    H_{0,i}: \mu_i = 0, \text{ versus }H_{0,i}: \mu_i > 0.
\end{align*}
The mean vector $\bm{\mu}$ and the covariance matrix $\bm{\Sigma}$ are generated as follows. For $\alpha_0 \in \{0.5, 0.8, 0.9, 0.95\}$, we generate a binomial variable $n_0$ with parameter $n$ and $\alpha_0$. Randomly selected $n_0$ positions among $n$ and set the corresponding $\mu_i = 0$. For the two-sided t-tests, we generate the alternative means by independently sampling from the symmetric bitriangular distribution with parameter $a=\log_2 1.2$ and $b=2$; see \cite{Langaas2005} for details. For one-sided t-tests, the remaining $\mu_i$ are independently generated from the symmetric triangular distribution with the same parameters as in the previous case. To consider the effect of correlation between tests, we consider a specific block diagonal structure for $\bm{\Sigma}$. We choose a block size $G \in \{50, 100\}$. The within-block correlation $\rho$ takes values in $\{0, 0.25, 0.5, 0.75\}$. Between blocks, the coordinates are independent. Note that $\rho=0$ means all the tests are pairwise independent.

We choose $n=2000$ and $m=10$ throughout. For the Bayesian approach, we continue to use the Dirichlet process mixture prior with parameters defined as in the previous section while considering $k$ as an unknown parameter with a prior distribution. The proposed estimator is the posterior mean of $\beta_0$. For comparison, we also estimate the null proportion by the maximum likelihood estimation for convex and decreasing density class, as proposed in \cite{Langaas2005}. This can be easily computed using the \texttt{convest} function in the \texttt{R} package \texttt{limma}. Each setting is replicated $1000$ times, and the densities of these two estimators are plotted in Figure \ref{fig:g50} -- Figure \ref{fig:osg100}.

For both two-sided and one-sided tests, the simulation results exhibit a similar pattern. The presence of correlation between tests has a detrimental effect on the performance of both methods. However, our Bayesian procedure demonstrates more stable performance in the presence of within-group correlation and for cases with larger blocks of correlation. 
Furthermore, our method shows more accurate estimation performance when the proportion of null hypotheses is relatively large. However, when $\alpha_0$ is moderate, such as $0.5$, the convex maximum likelihood estimator appears to be less biased.
These findings suggest that our Bayesian approach offers advantages in handling correlated tests and estimating the null proportion accurately, particularly when there is a higher proportion of null hypotheses, which is the very common case. 
These observations highlight the strengths and limitations of both methods in different scenarios, providing valuable insights into their respective performance characteristics.

\begin{figure}[ht]
\captionsetup[subfigure]{labelformat=empty}
\centering
\begin{subfigure}{.24\textwidth}
  \centering
  \includegraphics[width=\textwidth]{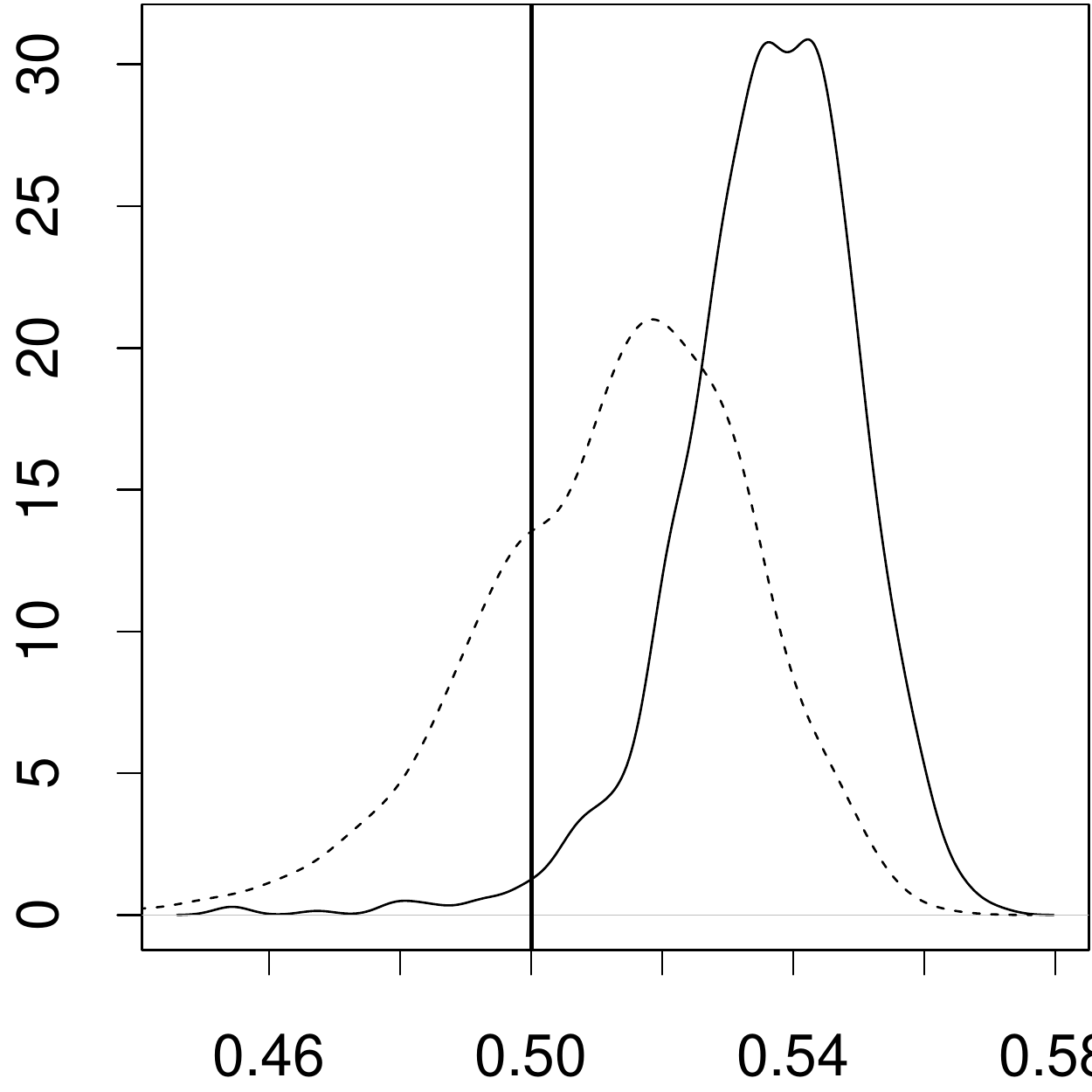}
  \caption{$\alpha_0=0.5$; $\rho = 0$}
\end{subfigure}
\begin{subfigure}{.24\textwidth}
  \centering
  \includegraphics[width=\textwidth]{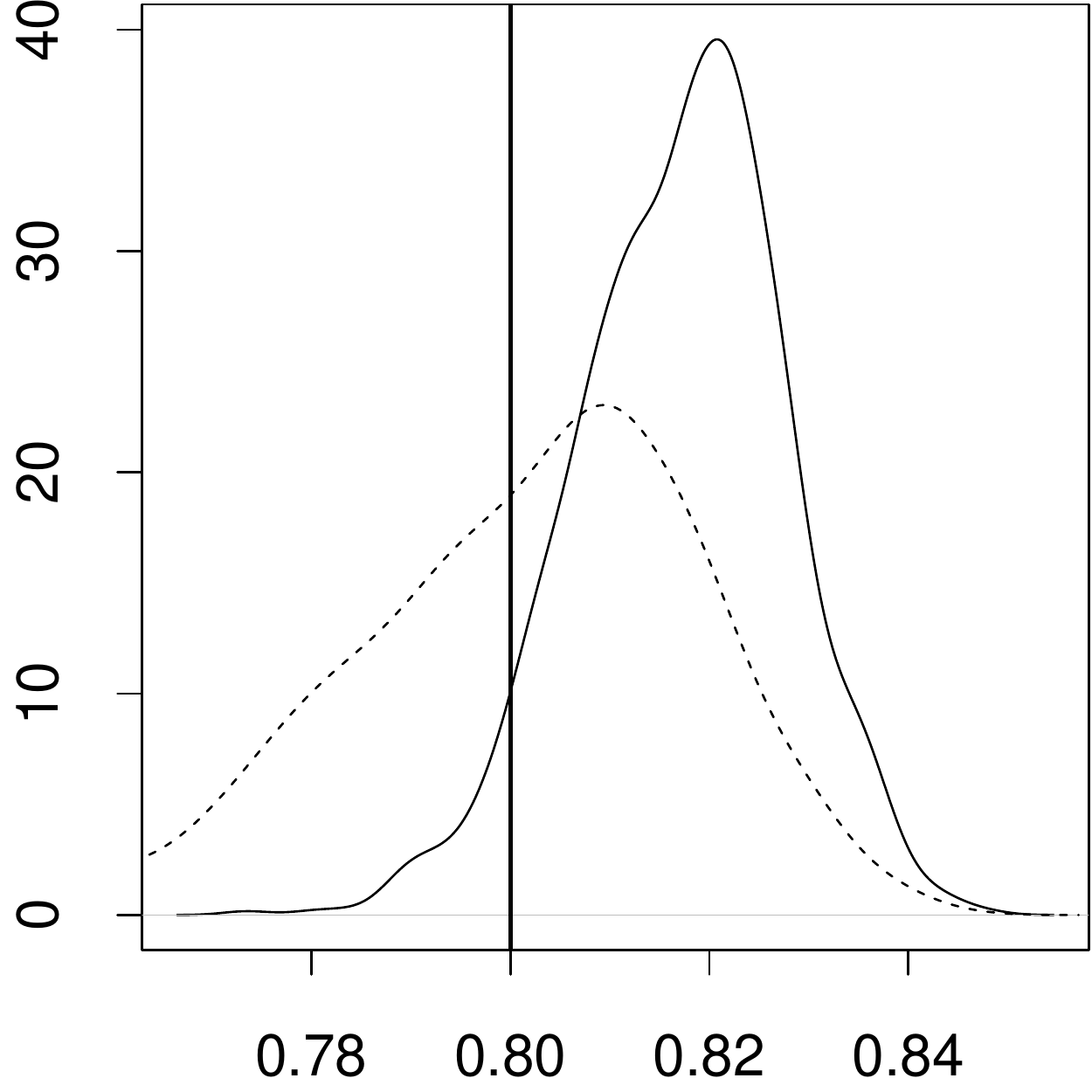}
  \caption{$\alpha_0=0.8$; $\rho = 0$}
\end{subfigure}
\begin{subfigure}{.24\textwidth}
  \centering
  \includegraphics[width=\textwidth]{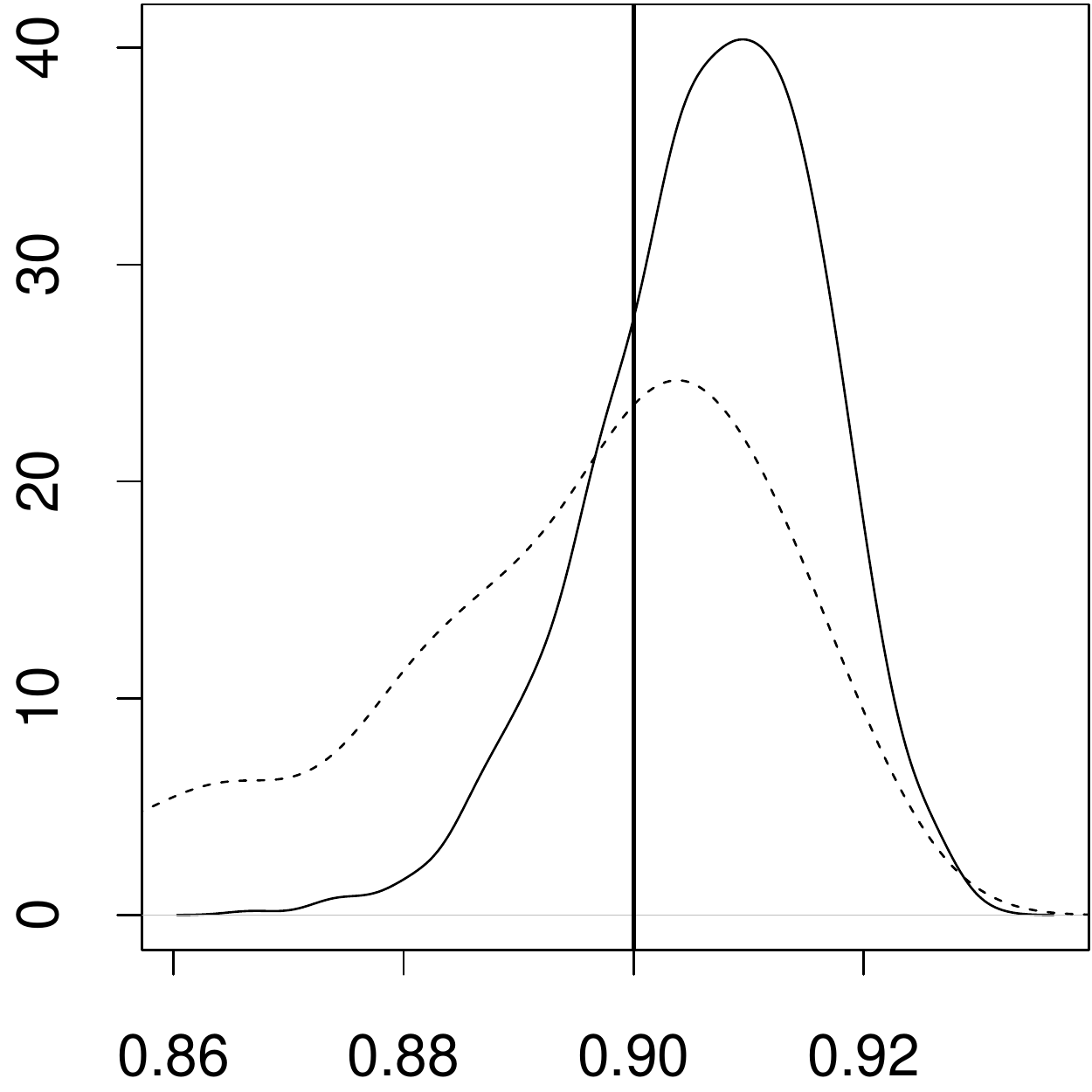}
  \caption{$\alpha_0=0.9$; $\rho = 0$}
\end{subfigure}
\begin{subfigure}{.24\textwidth}
  \centering
  \includegraphics[width=\textwidth]{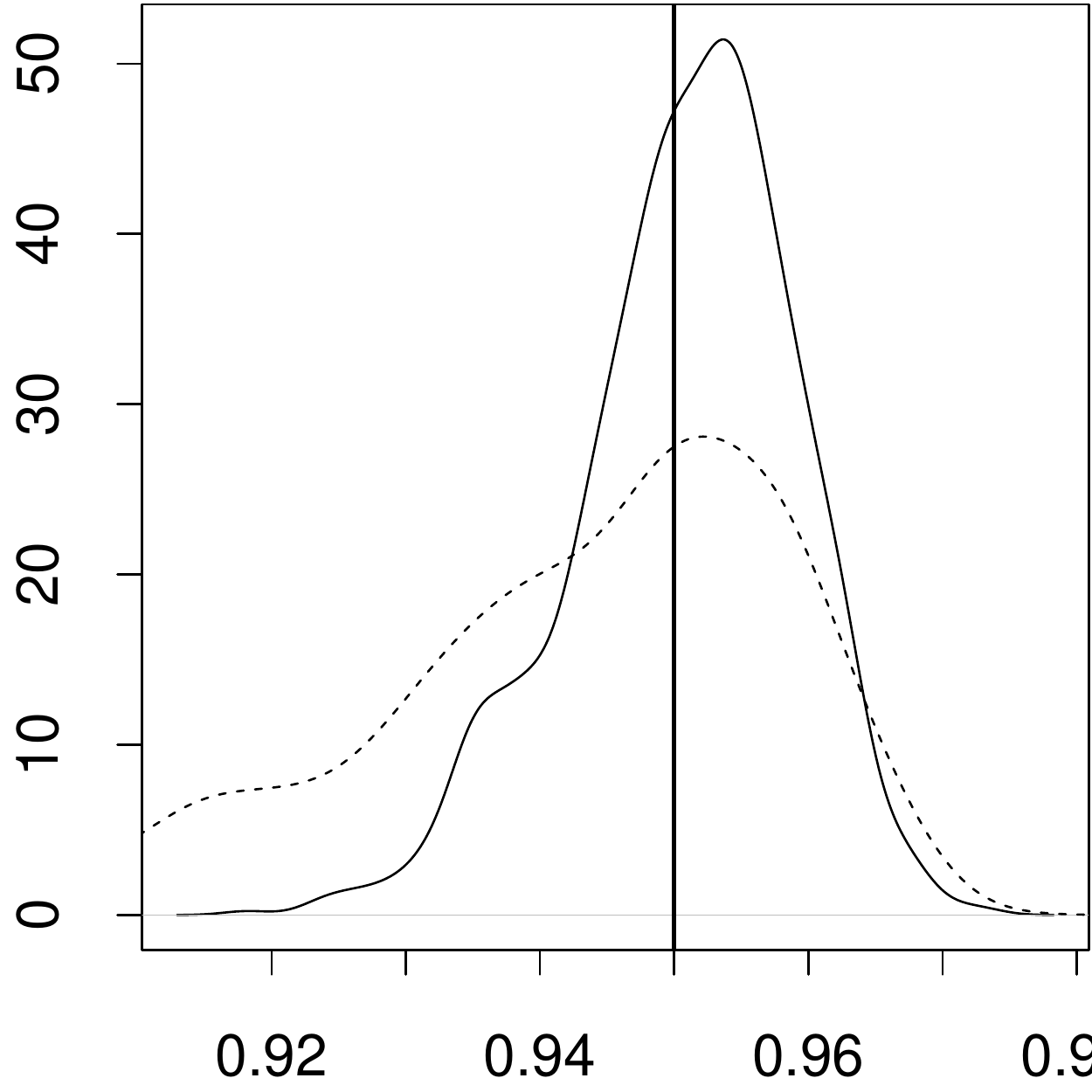}
  \caption{$\alpha_0=0.95$; $\rho = 0$}
\end{subfigure}

\begin{subfigure}{.24\textwidth}
  \centering
  \includegraphics[width=\textwidth]{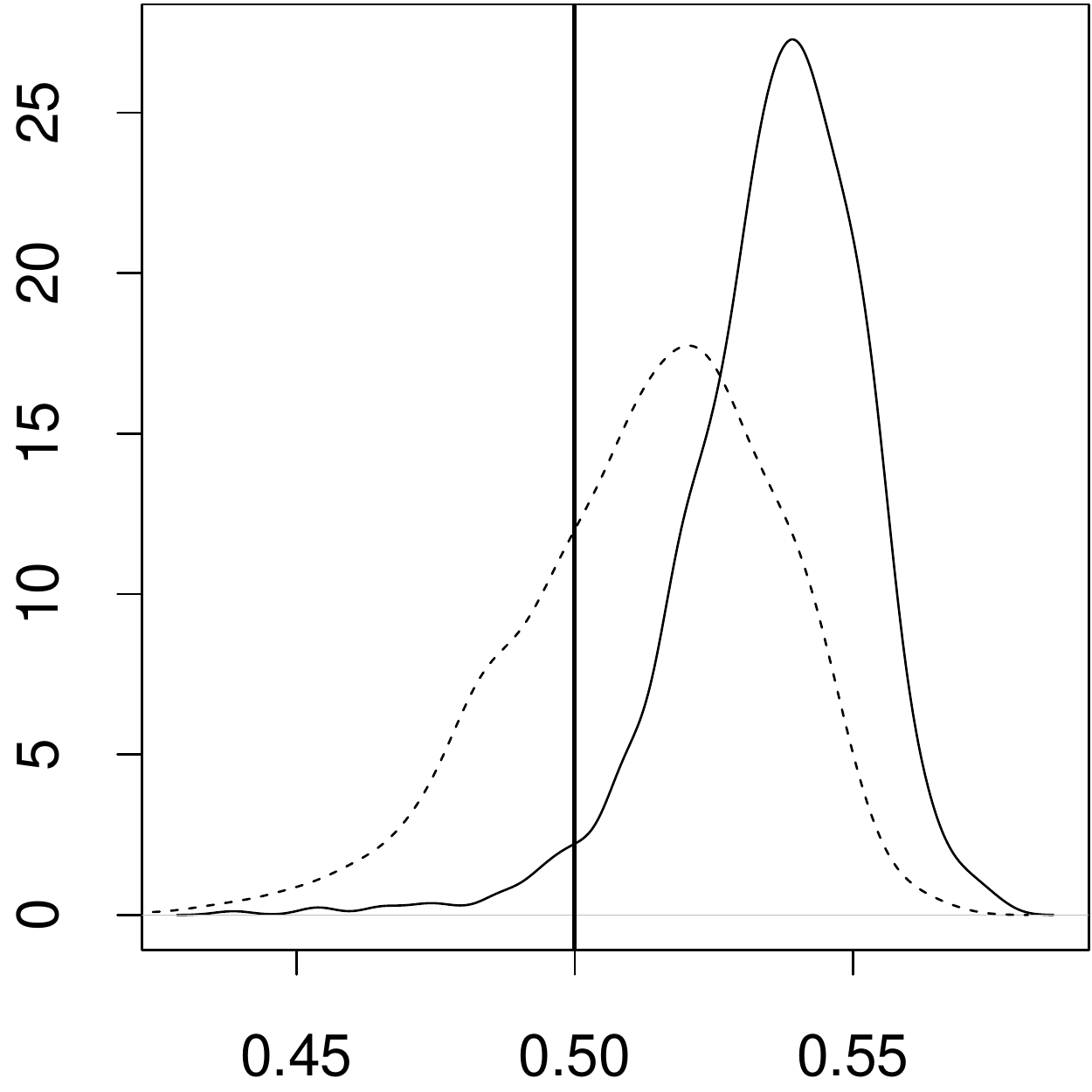}
  \caption{$\alpha_0=0.5$; $\rho = 0.25$}
\end{subfigure}
\begin{subfigure}{.24\textwidth}
  \centering
  \includegraphics[width=\textwidth]{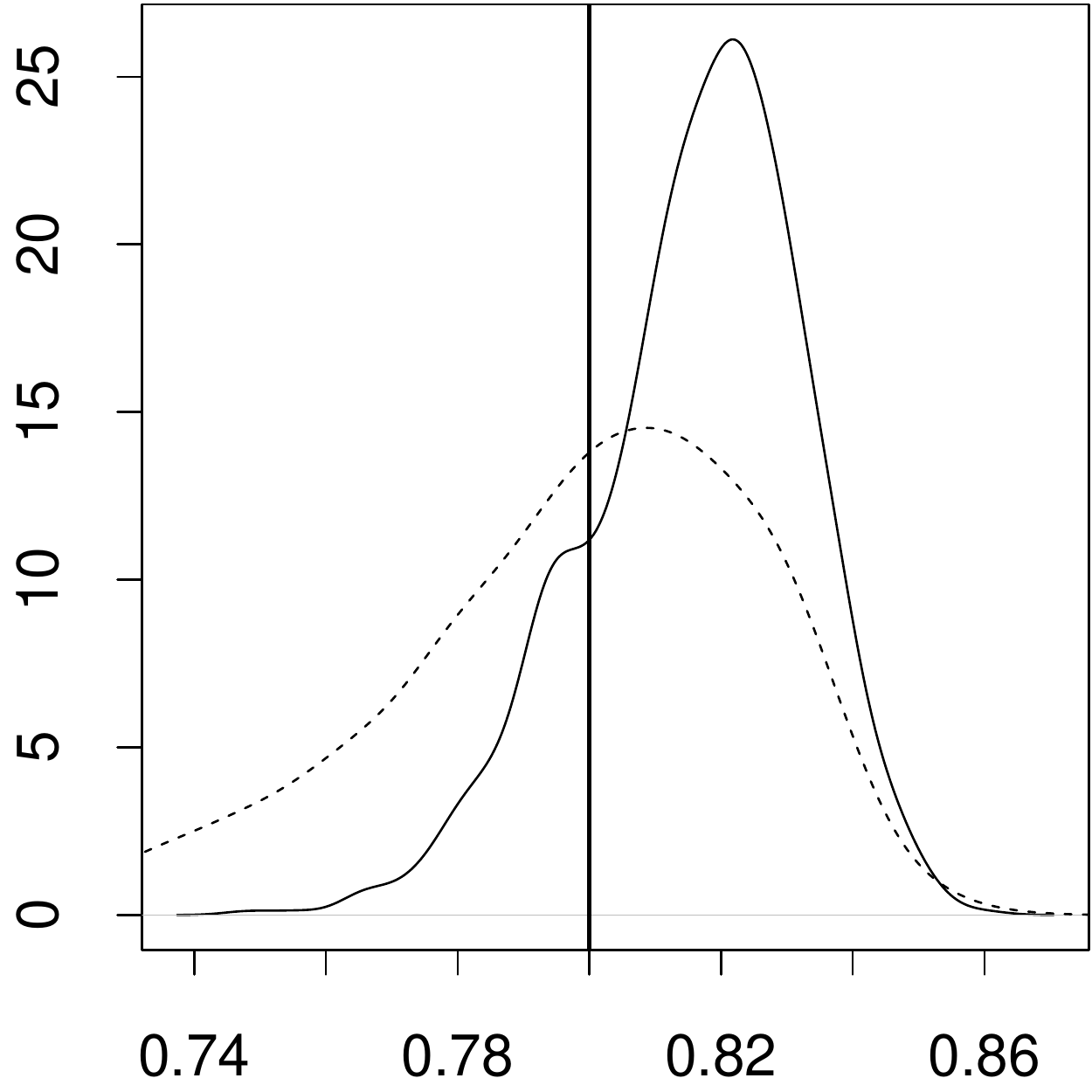}
  \caption{$\alpha_0=0.8$; $\rho = 0.25$}
\end{subfigure}
\begin{subfigure}{.24\textwidth}
  \centering
  \includegraphics[width=\textwidth]{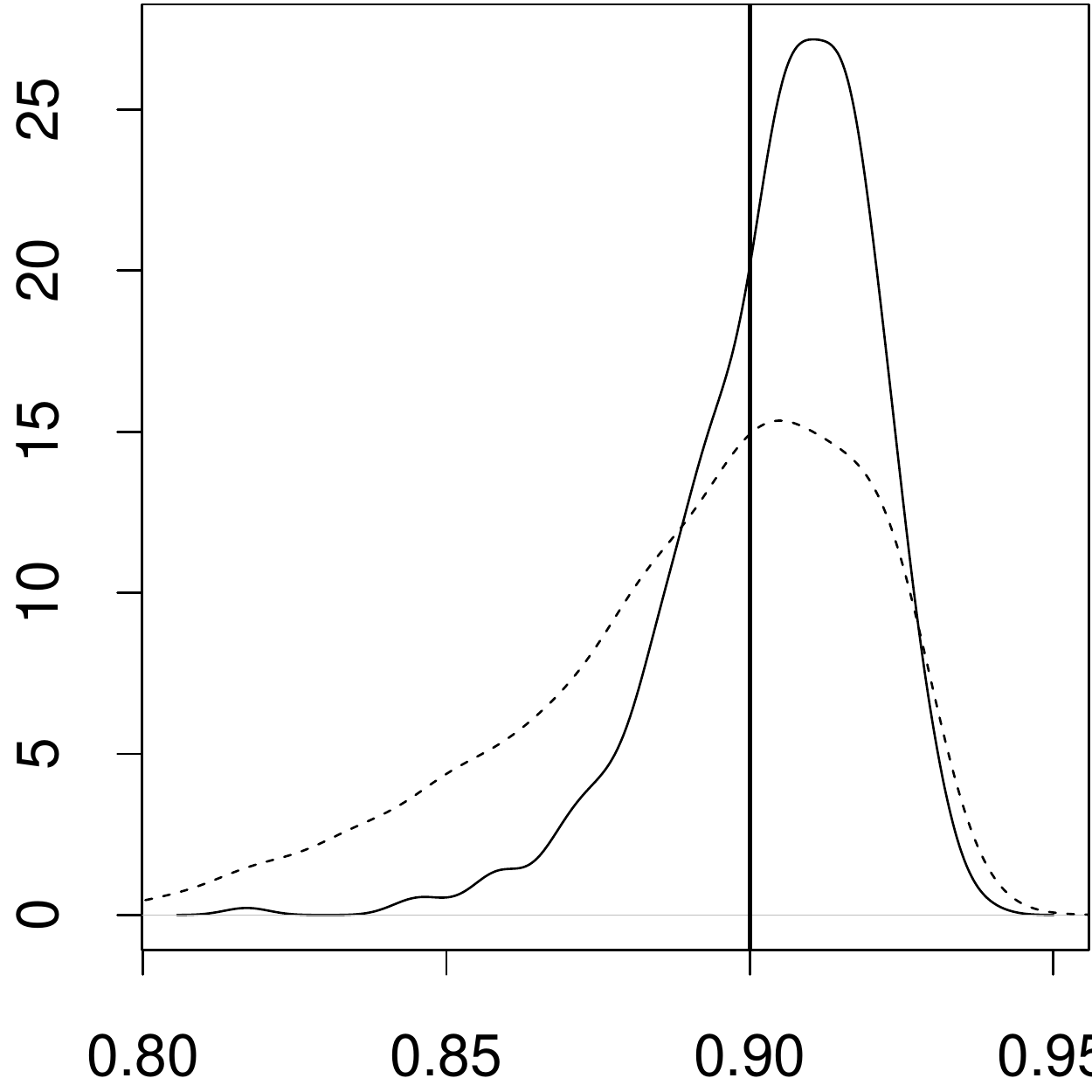}
  \caption{$\alpha_0=0.9$; $\rho = 0.25$}
\end{subfigure}
\begin{subfigure}{.24\textwidth}
  \centering
  \includegraphics[width=\textwidth]{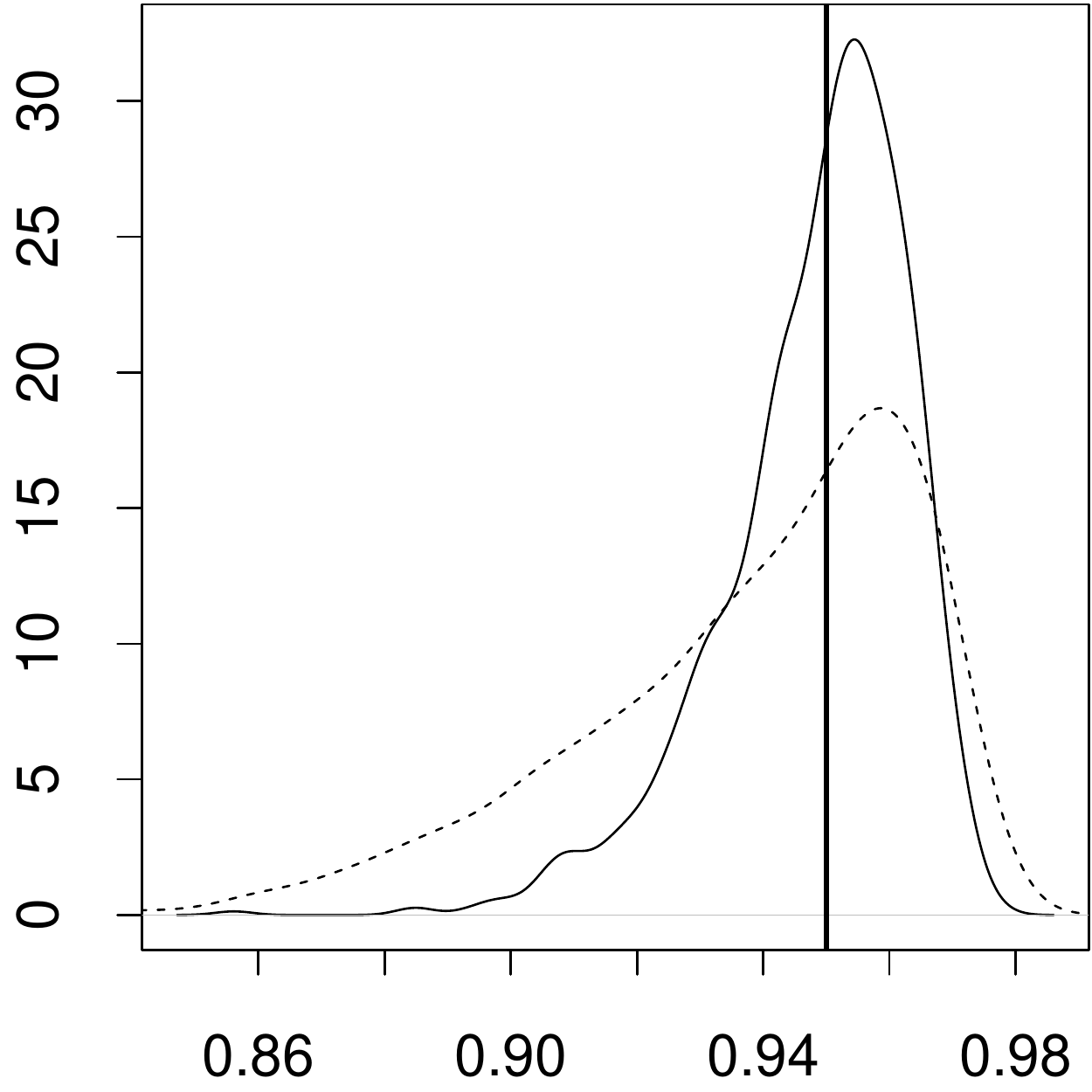}
  \caption{$\alpha_0=0.95$; $\rho = 0.25$}
\end{subfigure}

\begin{subfigure}{.24\textwidth}
  \centering
  \includegraphics[width=\textwidth]{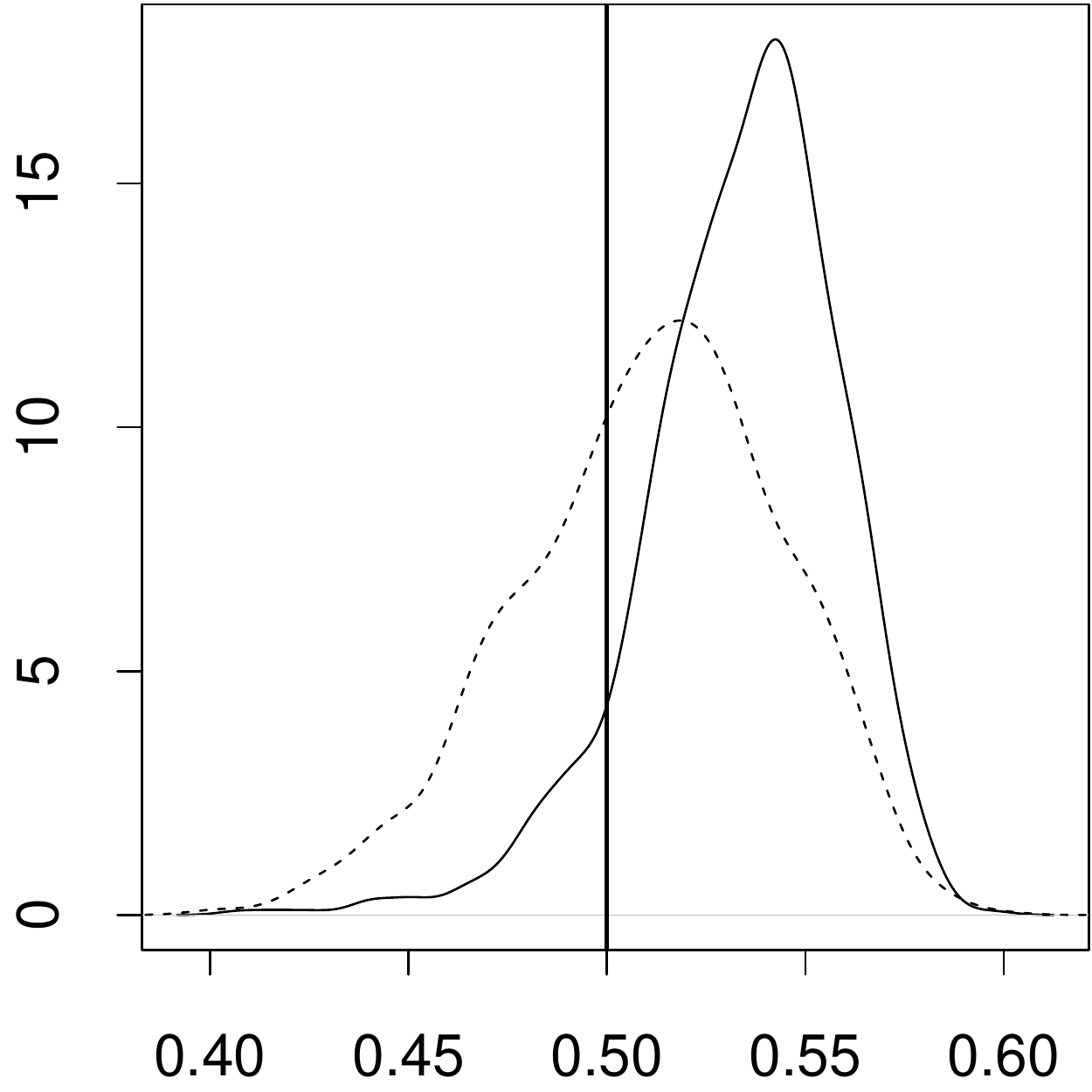}
  \caption{$\alpha_0=0.5$; $\rho = 0.5$}
\end{subfigure}
\begin{subfigure}{.24\textwidth}
  \centering
  \includegraphics[width=\textwidth]{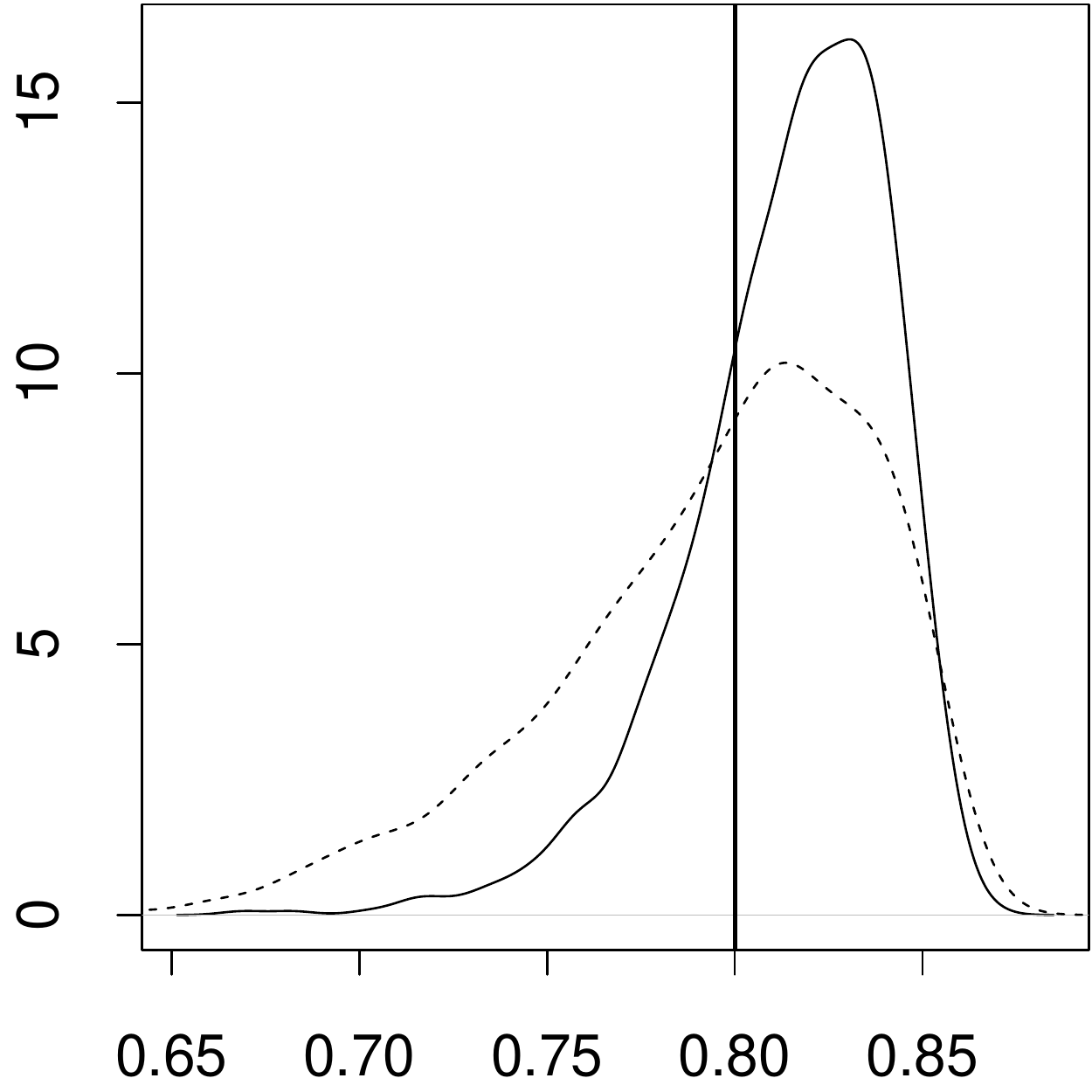}
  \caption{$\alpha_0=0.8$; $\rho = 0.5$}
\end{subfigure}
\begin{subfigure}{.24\textwidth}
  \centering
  \includegraphics[width=\textwidth]{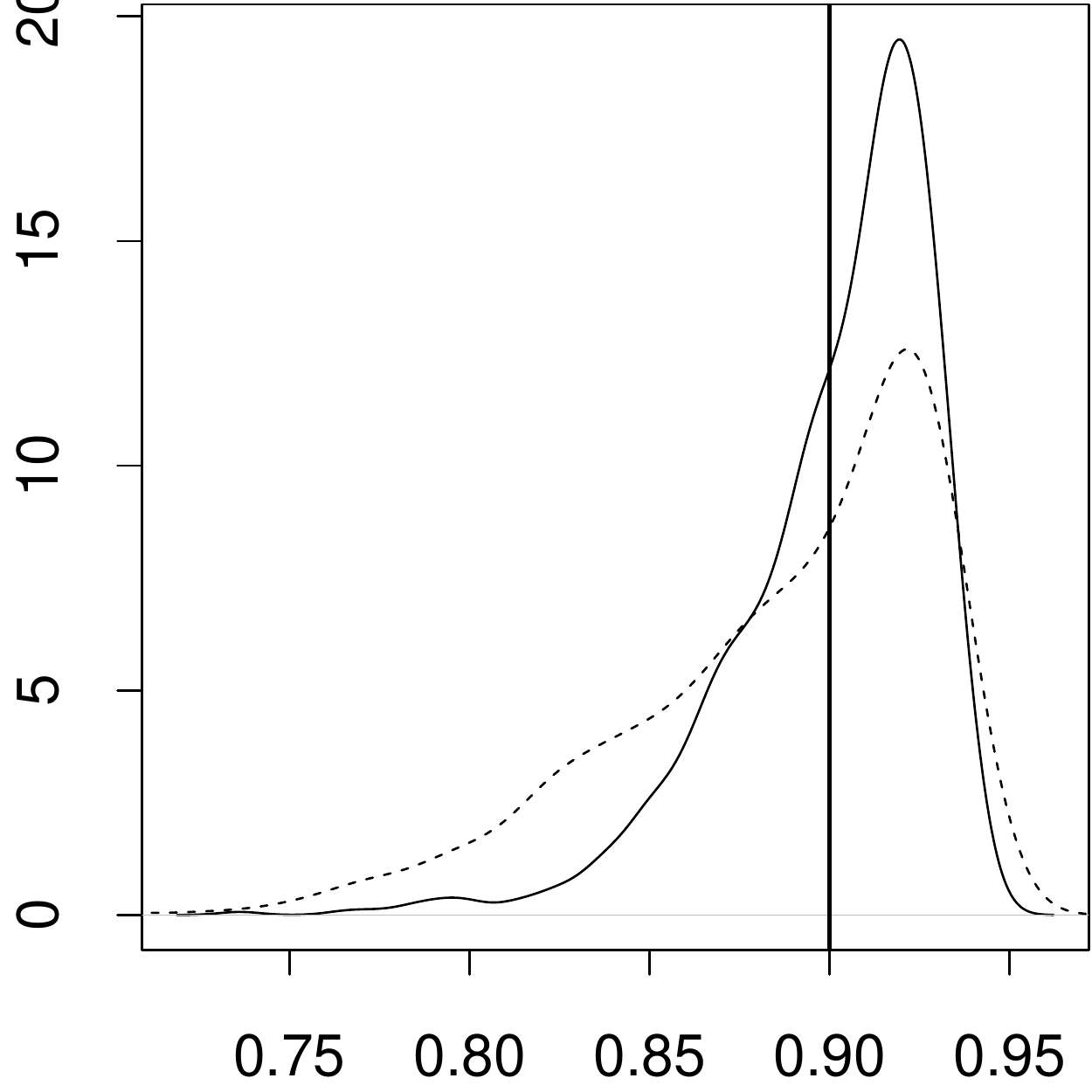}
  \caption{$\alpha_0=0.9$; $\rho = 0.5$}
\end{subfigure}
\begin{subfigure}{.24\textwidth}
  \centering
  \includegraphics[width=\textwidth]{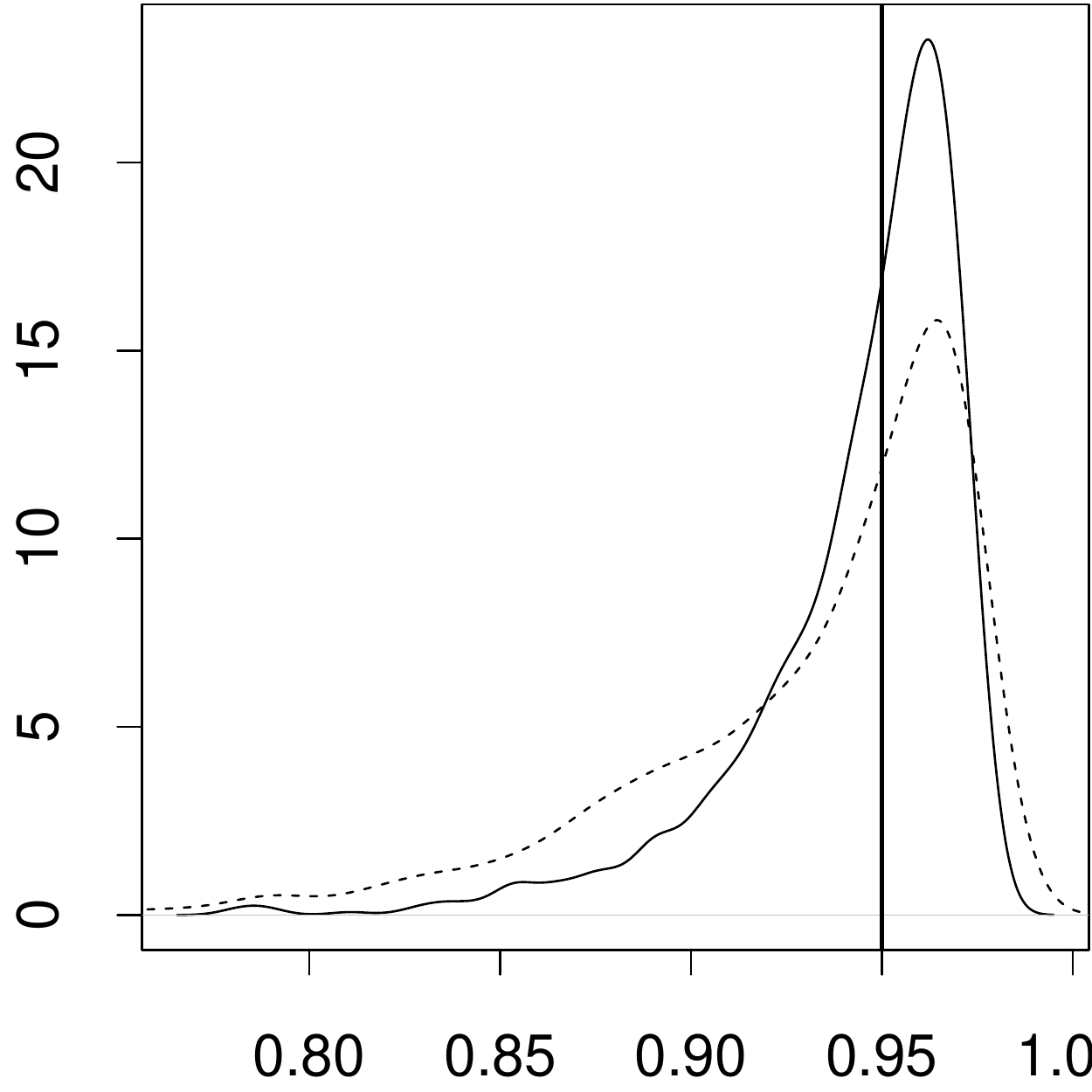}
  \caption{$\alpha_0=0.95$; $\rho = 0.5$}
\end{subfigure}

\begin{subfigure}{.24\textwidth}
  \centering
  \includegraphics[width=\textwidth]{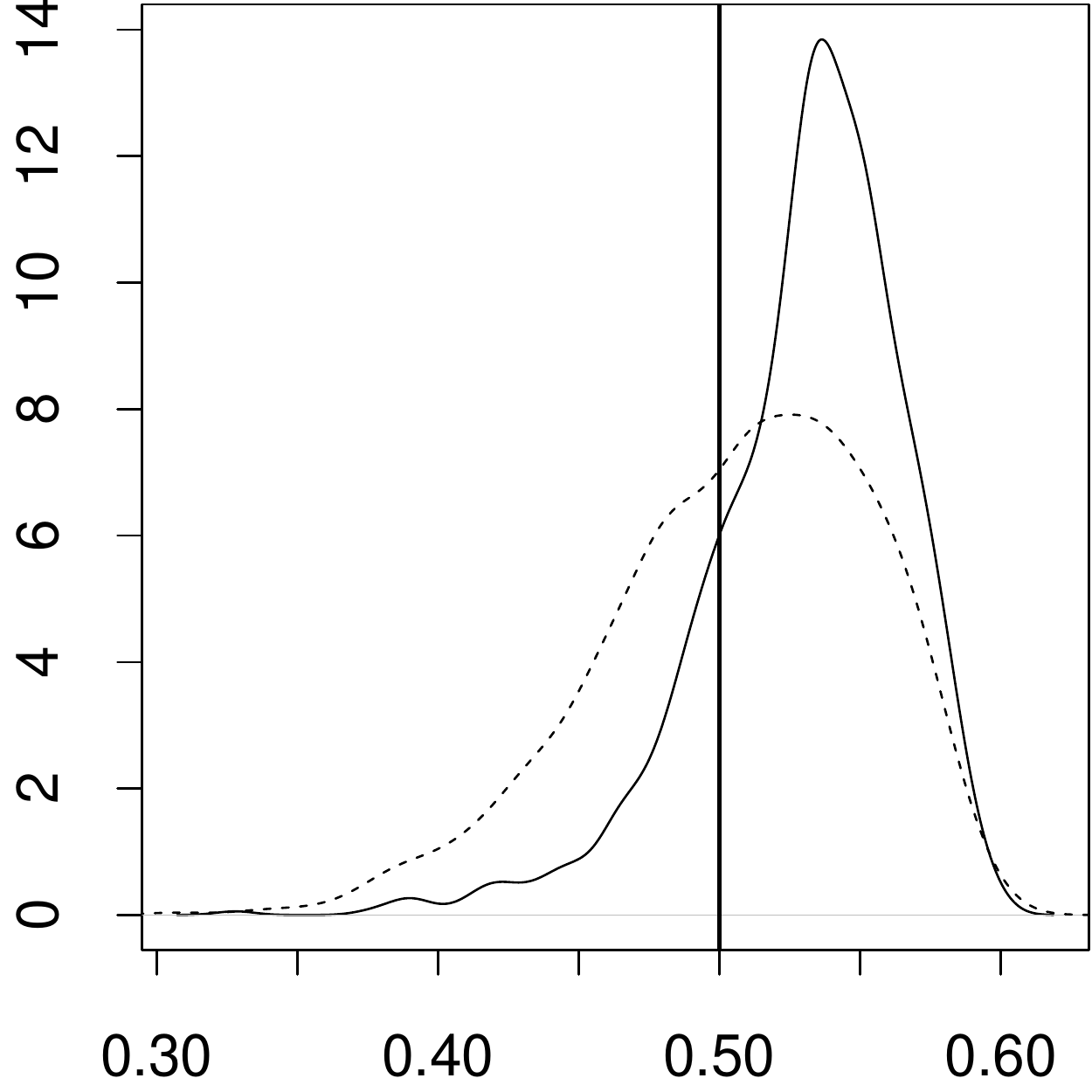}
  \caption{$\alpha_0=0.5$; $\rho = 0.75$}
\end{subfigure}
\begin{subfigure}{.24\textwidth}
  \centering
  \includegraphics[width=\textwidth]{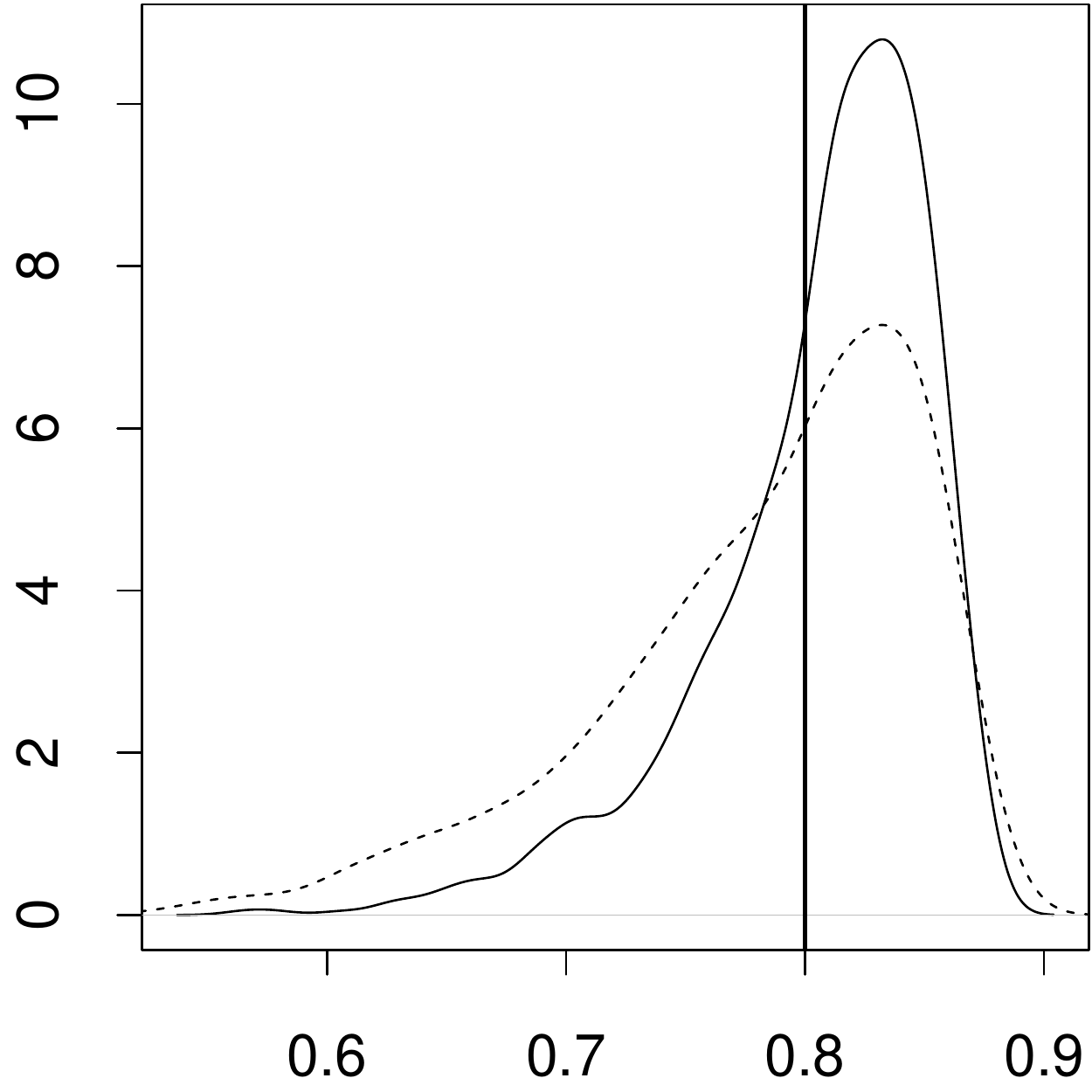}
  \caption{$\alpha_0=0.8$; $\rho = 0.75$}
\end{subfigure}
\begin{subfigure}{.24\textwidth}
  \centering
  \includegraphics[width=\textwidth]{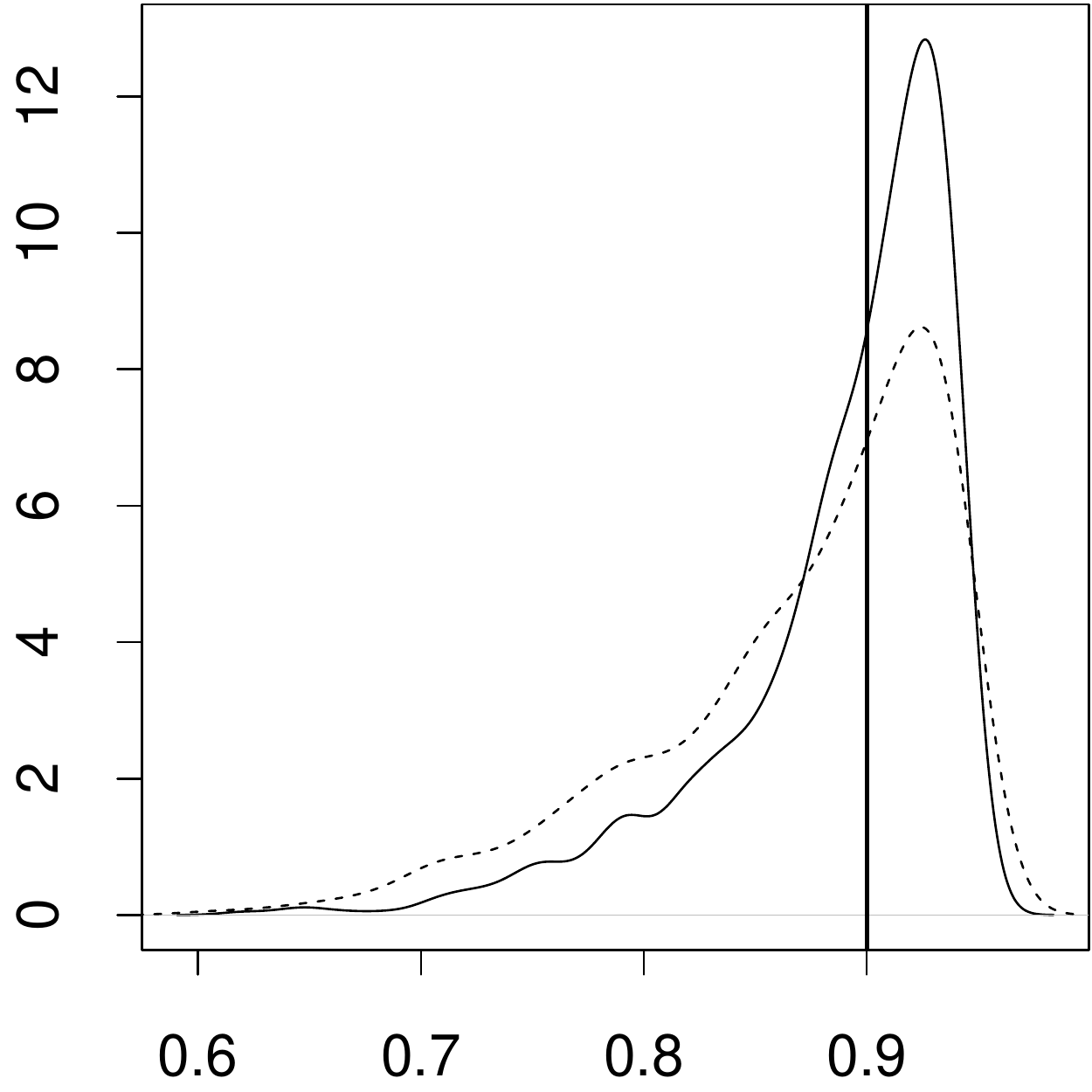}
  \caption{$\alpha_0=0.9$; $\rho = 0.75$}
\end{subfigure}
\begin{subfigure}{.24\textwidth}
  \centering
  \includegraphics[width=\textwidth]{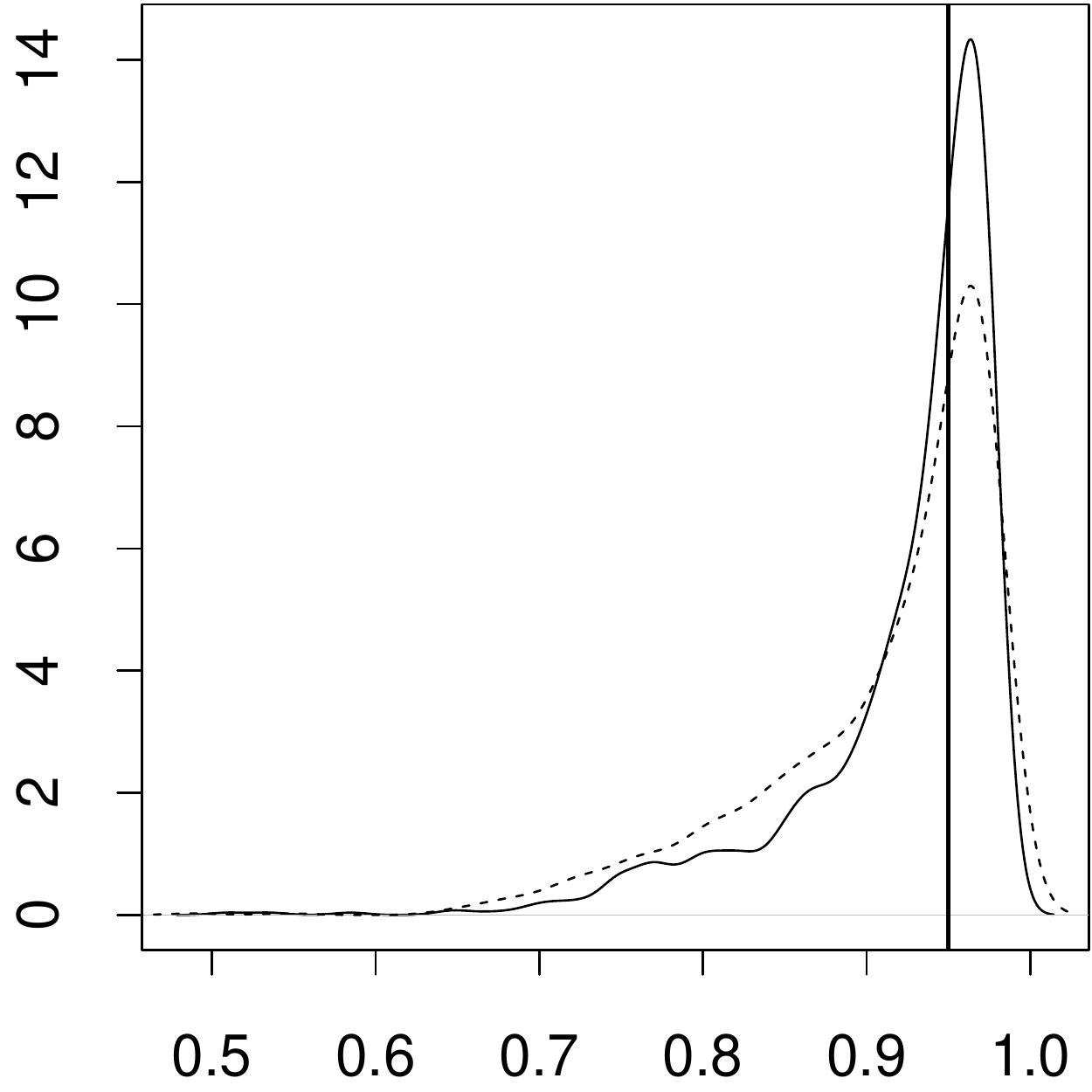}
  \caption{$\alpha_0=0.95$; $\rho = 0.75$}
\end{subfigure}
\caption{Density plots of the estimated $\alpha_0$ for $p$-values from two-sided t-tests ($G = 50$). The vertical line indicates the true $\alpha_0$, which is also marked below each figure, along with the within-group correlation coefficient. The solid lines are the densities of the posterior mean of $\alpha_0$, and the dashed lines are the densities of estimated $\alpha_0$ by fitting a convex decreasing density.
dash line}
\label{fig:g50}
\end{figure}

\begin{figure}[ht]
\captionsetup[subfigure]{labelformat=empty}
\centering
\begin{subfigure}{.24\textwidth}
  \centering
  \includegraphics[width=\textwidth]{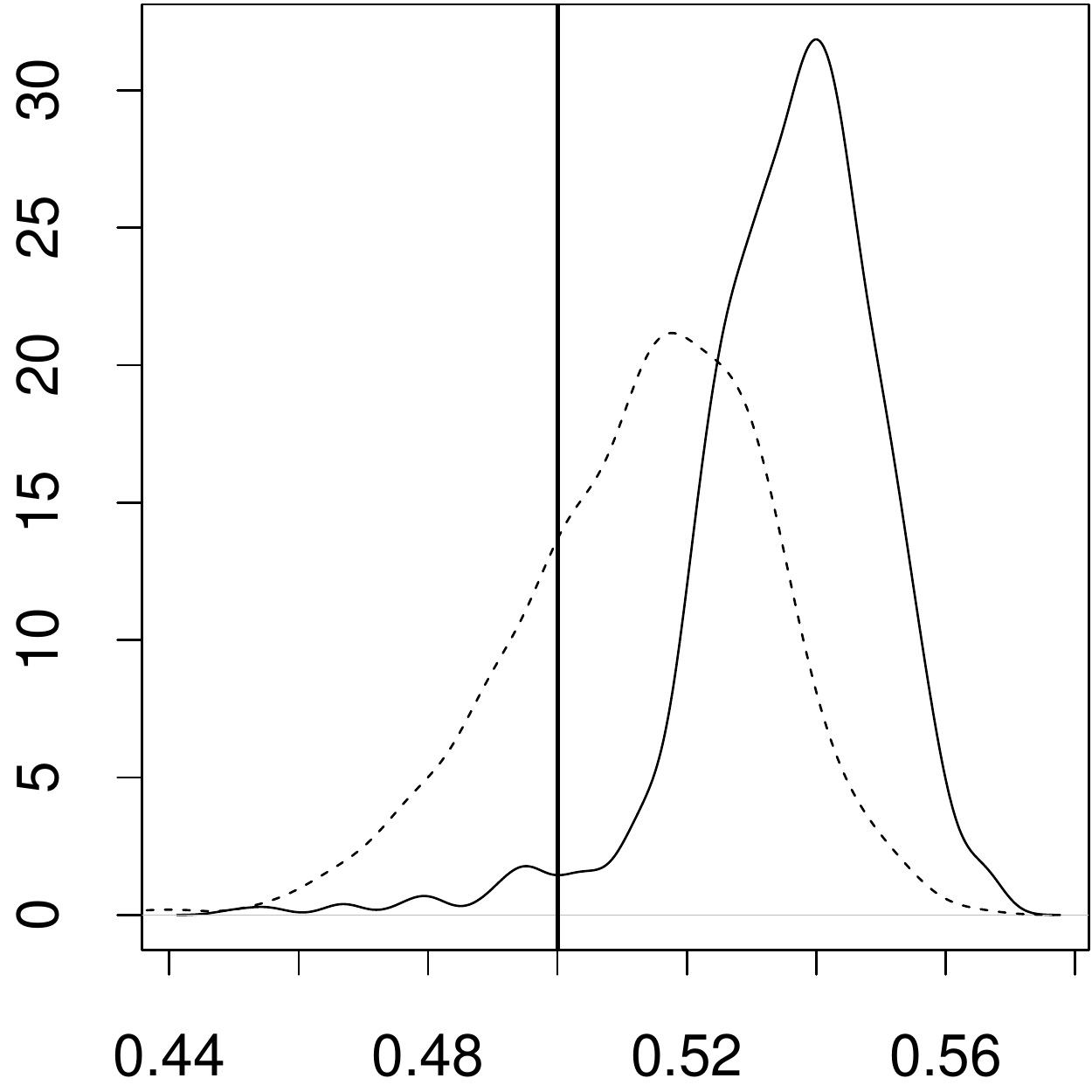}
  \caption{$\alpha_0=0.5$; $\rho = 0$}
\end{subfigure}
\begin{subfigure}{.24\textwidth}
  \centering
  \includegraphics[width=\textwidth]{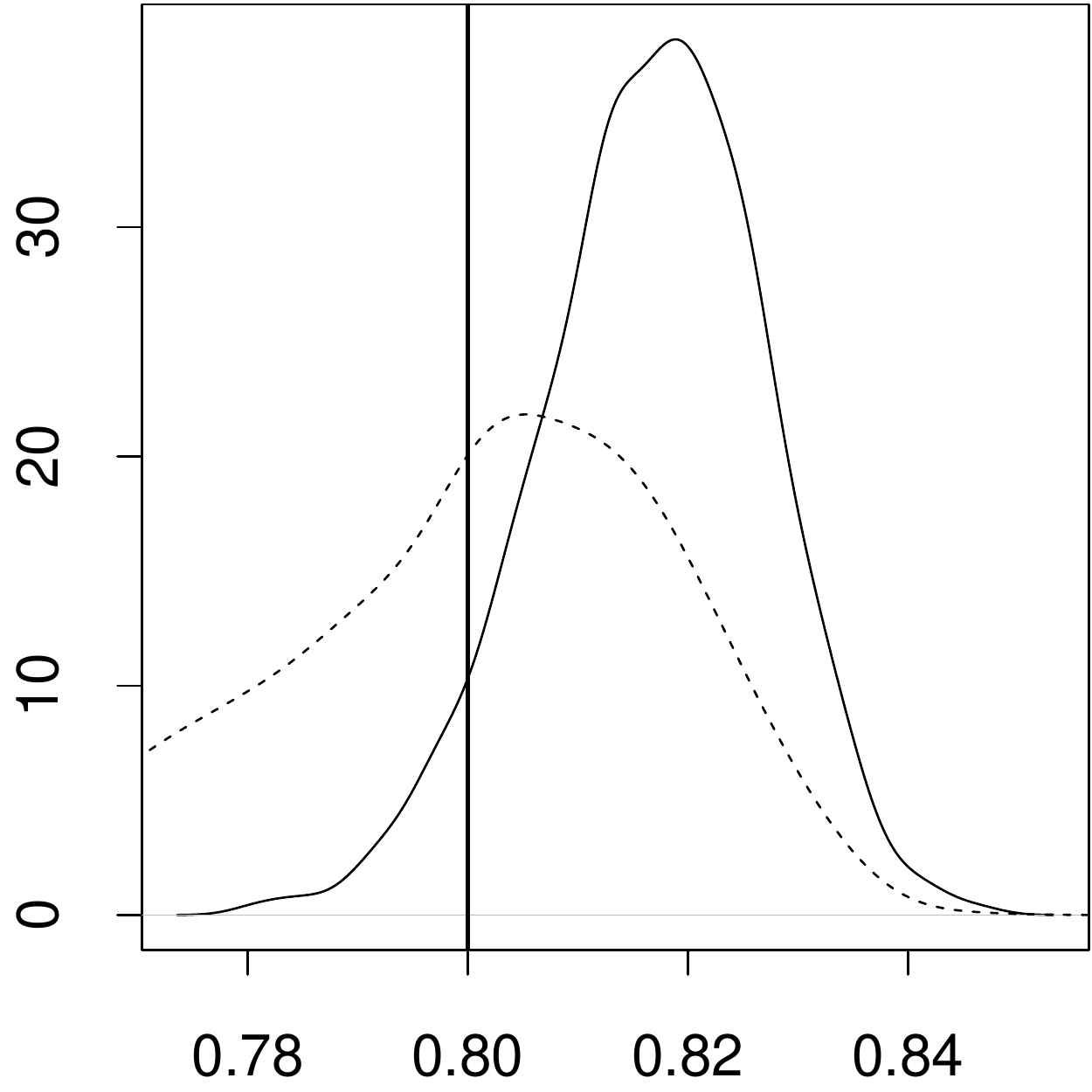}
  \caption{$\alpha_0=0.8$; $\rho = 0$}
\end{subfigure}
\begin{subfigure}{.24\textwidth}
  \centering
  \includegraphics[width=\textwidth]{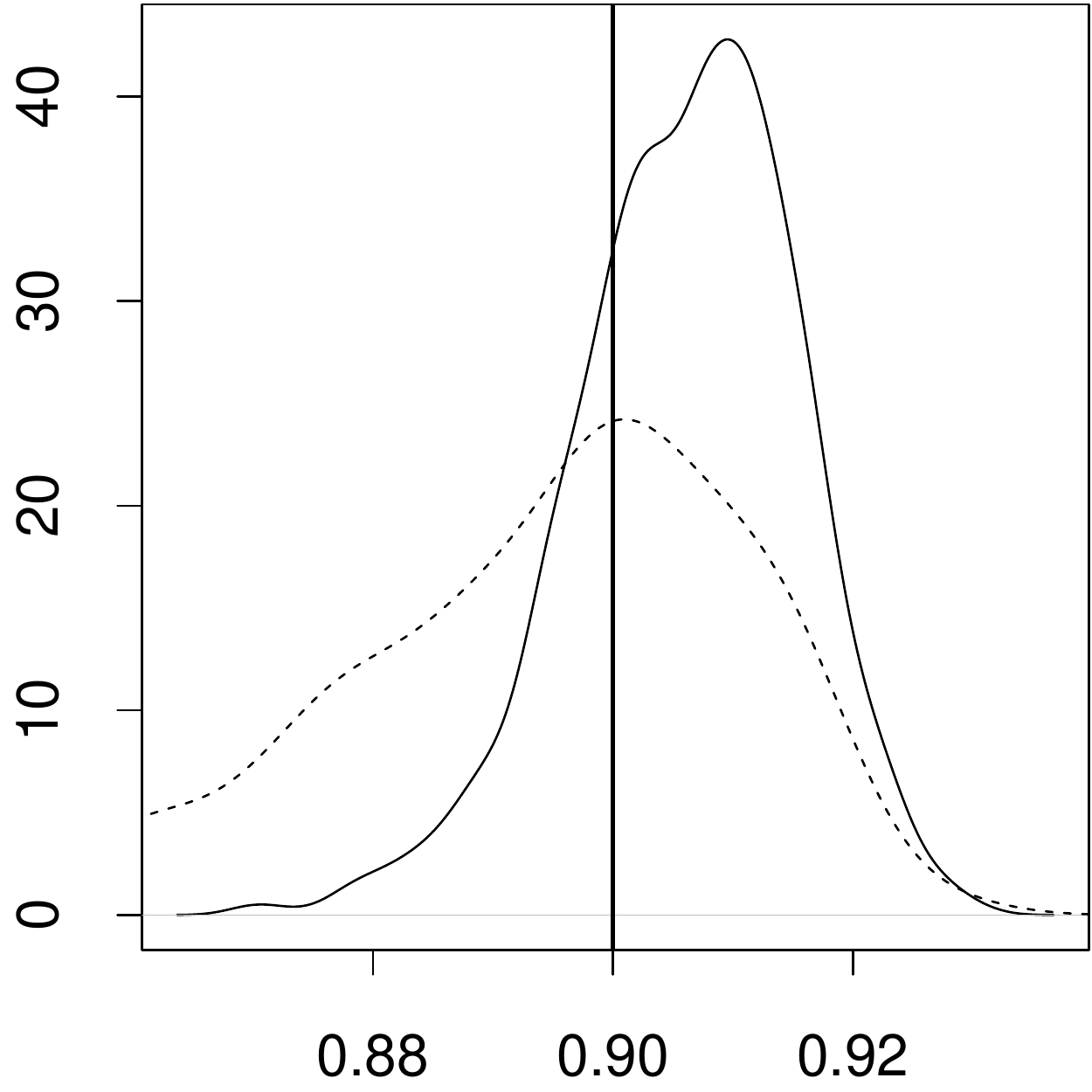}
  \caption{$\alpha_0=0.9$; $\rho = 0$}
\end{subfigure}
\begin{subfigure}{.24\textwidth}
  \centering
  \includegraphics[width=\textwidth]{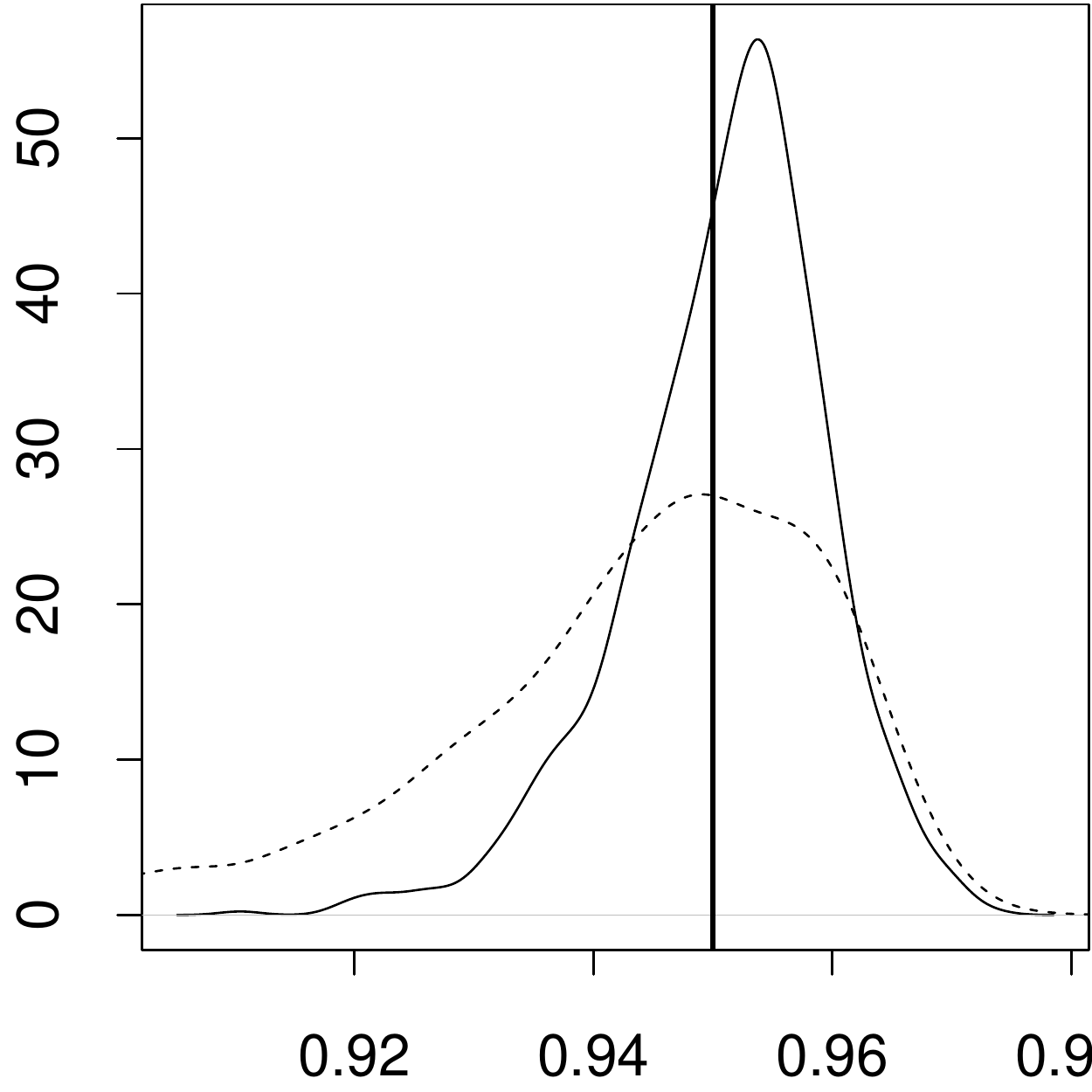}
  \caption{$\alpha_0=0.95$; $\rho = 0$}
\end{subfigure}

\begin{subfigure}{.24\textwidth}
  \centering
  \includegraphics[width=\textwidth]{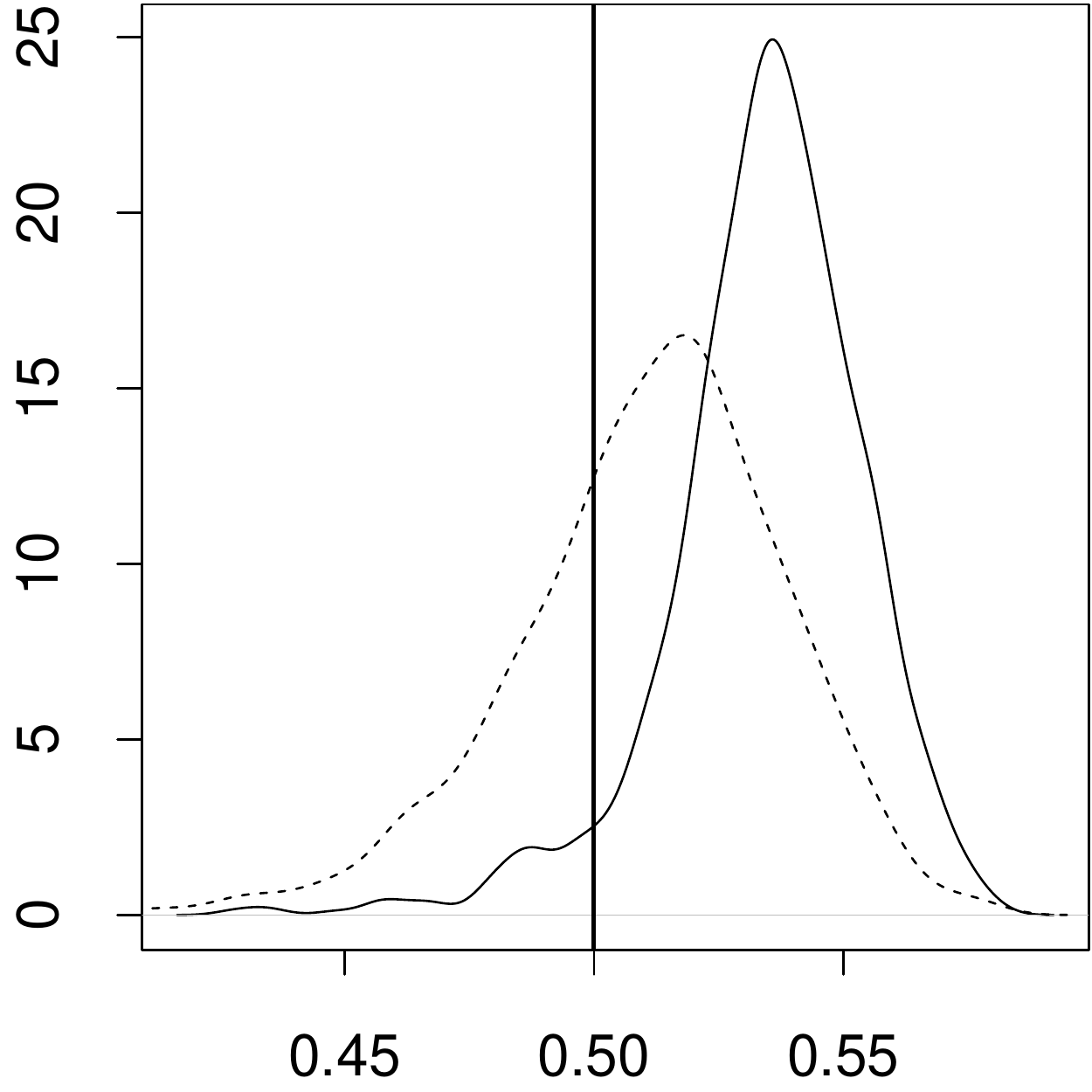}
  \caption{$\alpha_0=0.5$; $\rho = 0.25$}
\end{subfigure}
\begin{subfigure}{.24\textwidth}
  \centering
  \includegraphics[width=\textwidth]{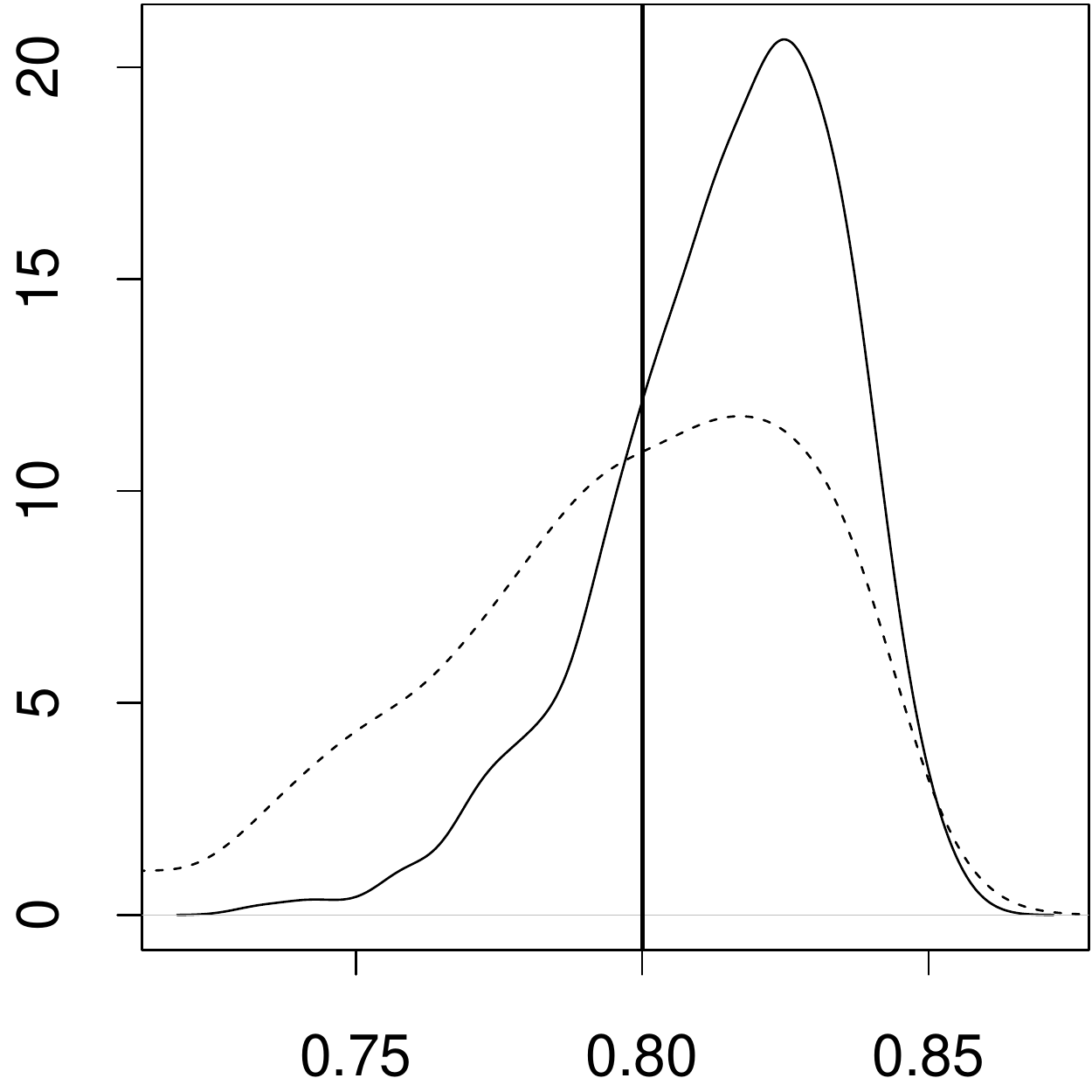}
  \caption{$\alpha_0=0.8$; $\rho = 0.25$}
\end{subfigure}
\begin{subfigure}{.24\textwidth}
  \centering
  \includegraphics[width=\textwidth]{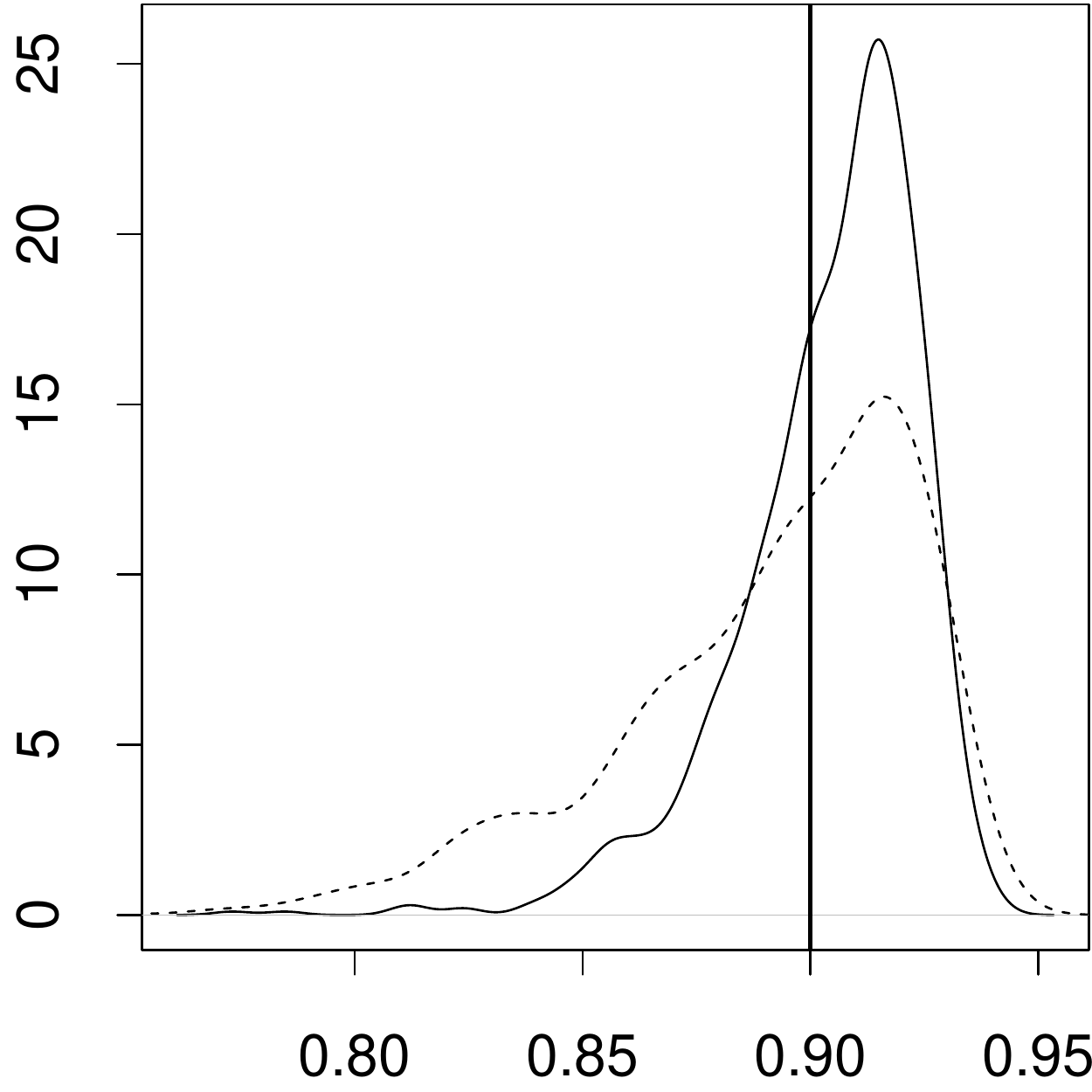}
  \caption{$\alpha_0=0.9$; $\rho = 0.25$}
\end{subfigure}
\begin{subfigure}{.24\textwidth}
  \centering
  \includegraphics[width=\textwidth]{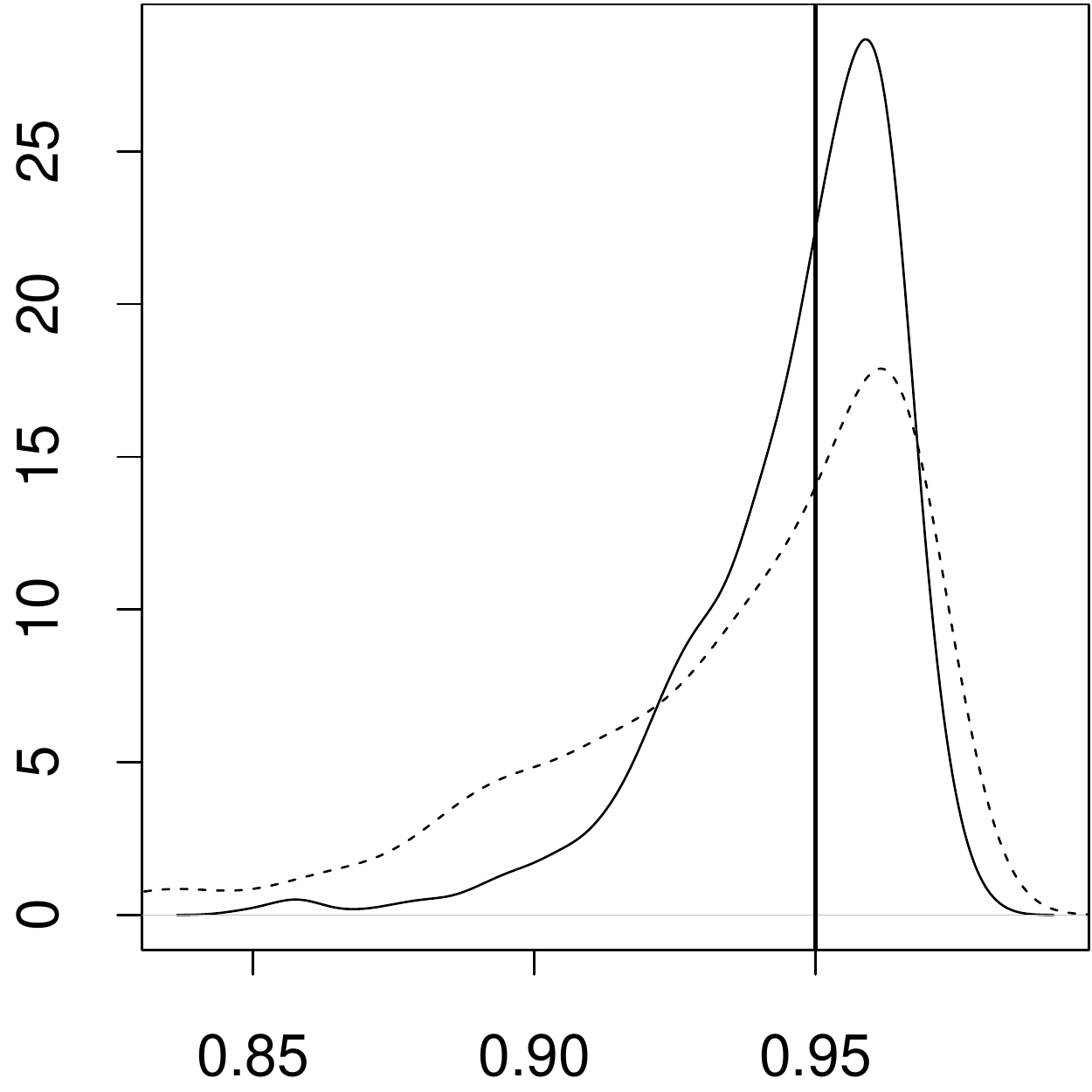}
  \caption{$\alpha_0=0.95$; $\rho = 0.25$}
\end{subfigure}

\begin{subfigure}{.24\textwidth}
  \centering
  \includegraphics[width=\textwidth]{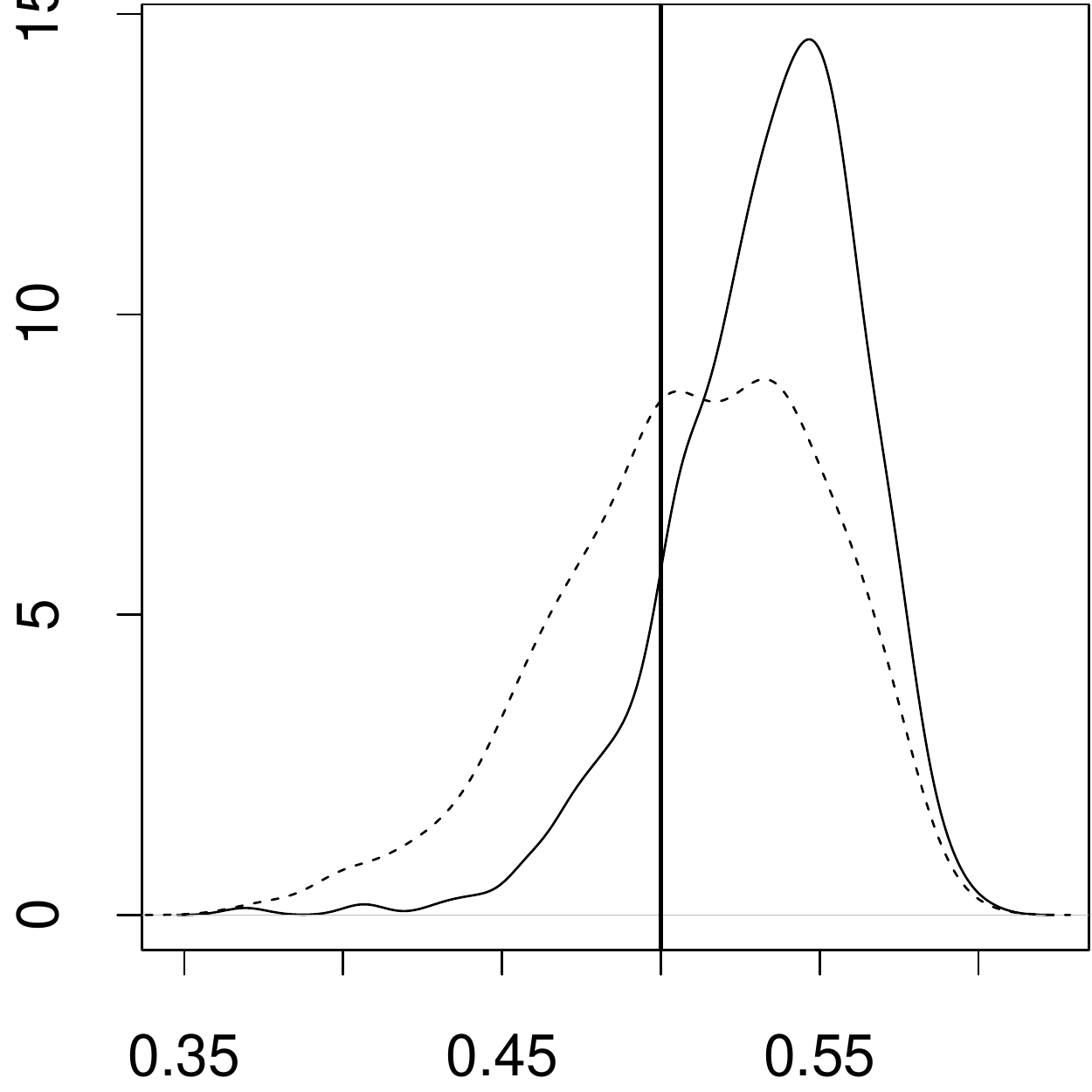}
  \caption{$\alpha_0=0.5$; $\rho = 0.5$}
\end{subfigure}
\begin{subfigure}{.24\textwidth}
  \centering
  \includegraphics[width=\textwidth]{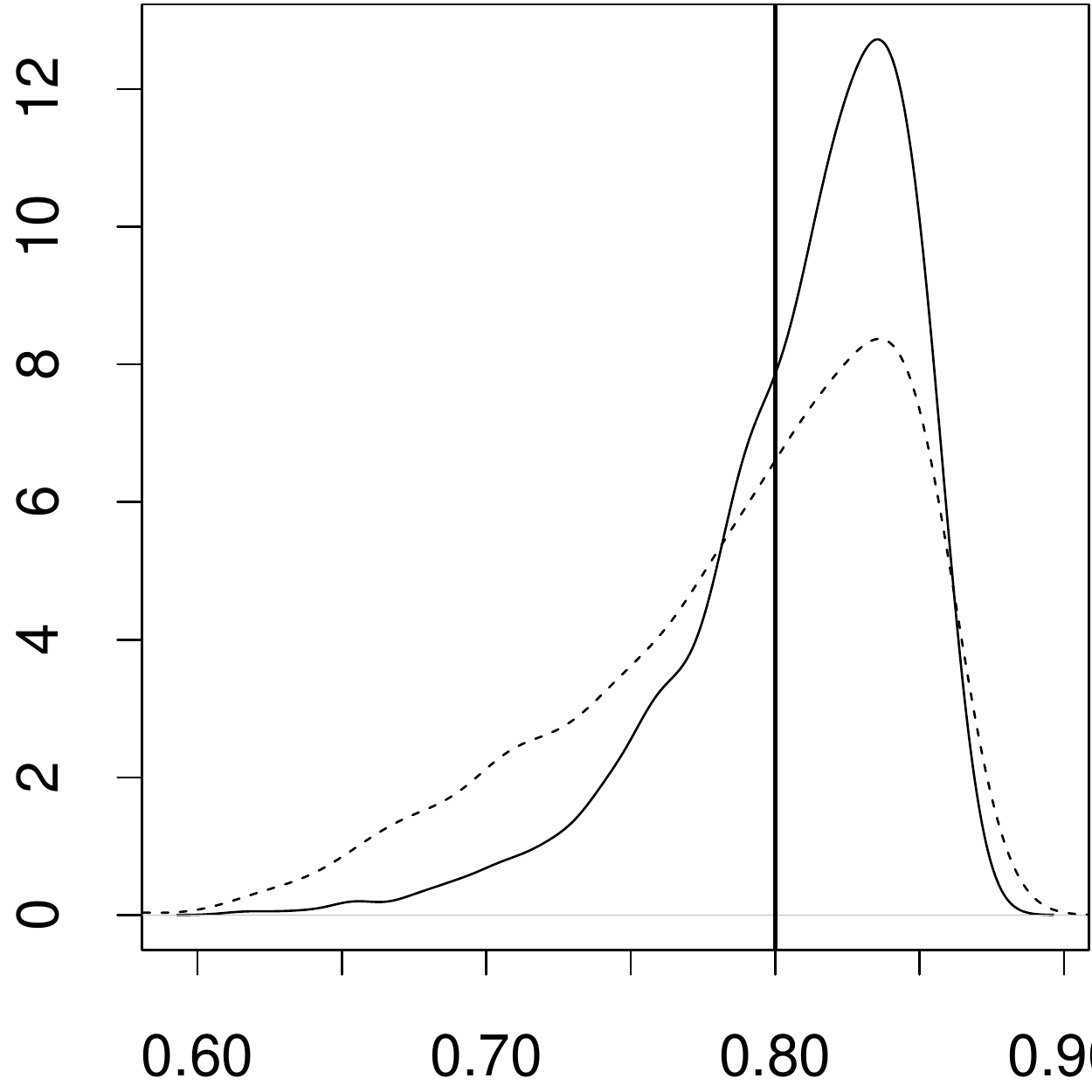}
  \caption{$\alpha_0=0.8$; $\rho = 0.5$}
\end{subfigure}
\begin{subfigure}{.24\textwidth}
  \centering
  \includegraphics[width=\textwidth]{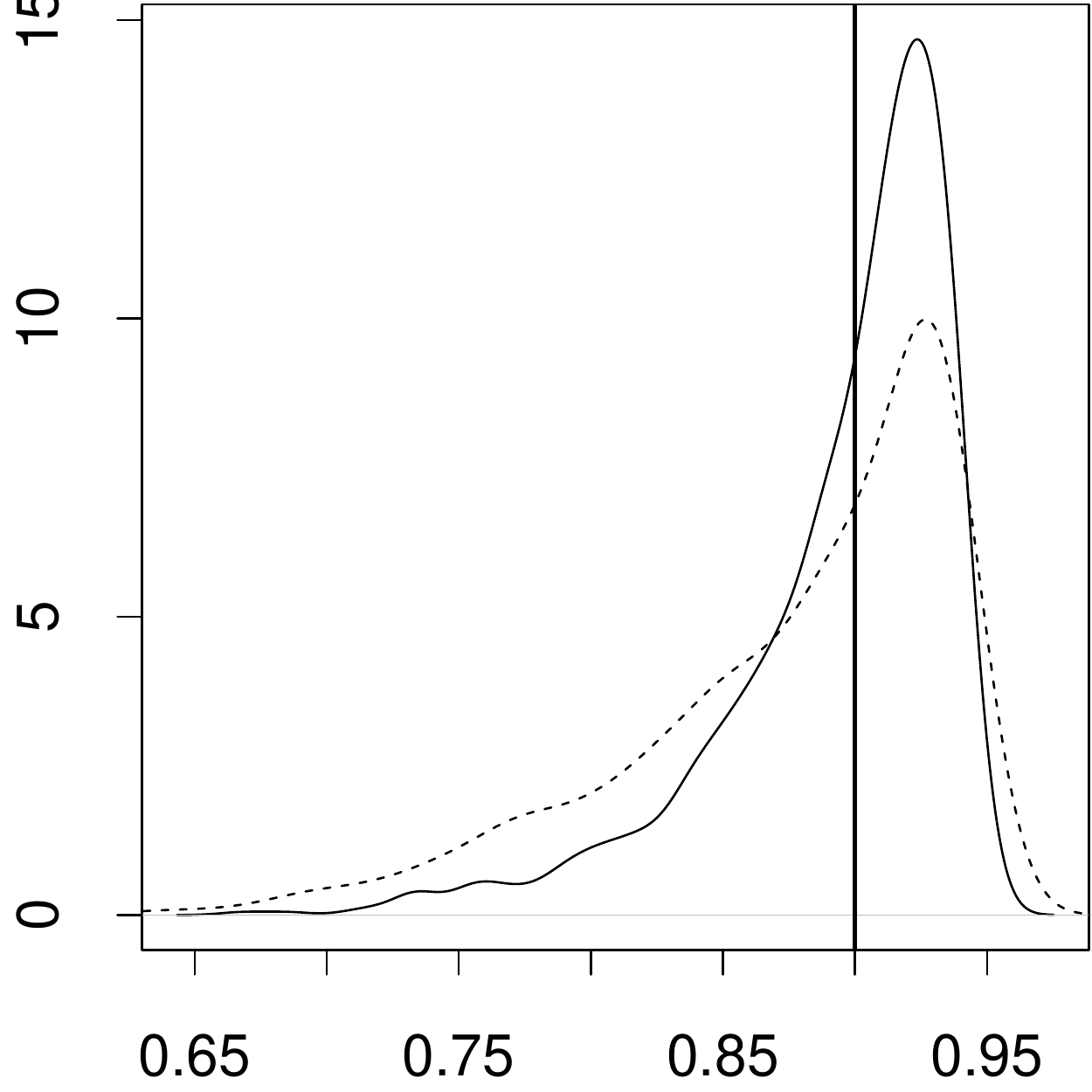}
  \caption{$\alpha_0=0.9$; $\rho = 0.5$}
\end{subfigure}
\begin{subfigure}{.24\textwidth}
  \centering
  \includegraphics[width=\textwidth]{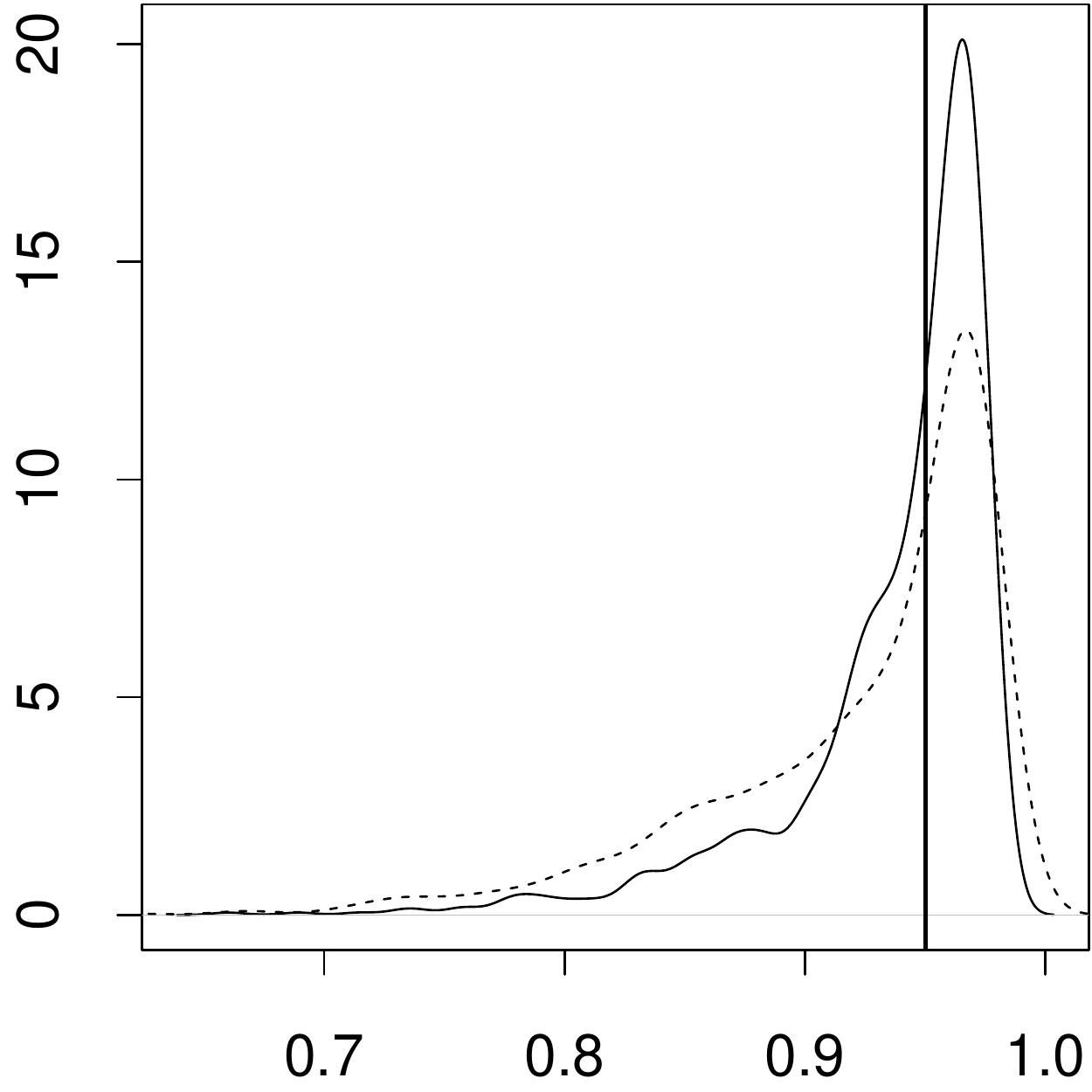}
  \caption{$\alpha_0=0.95$; $\rho = 0.5$}
\end{subfigure}

\begin{subfigure}{.24\textwidth}
  \centering
  \includegraphics[width=\textwidth]{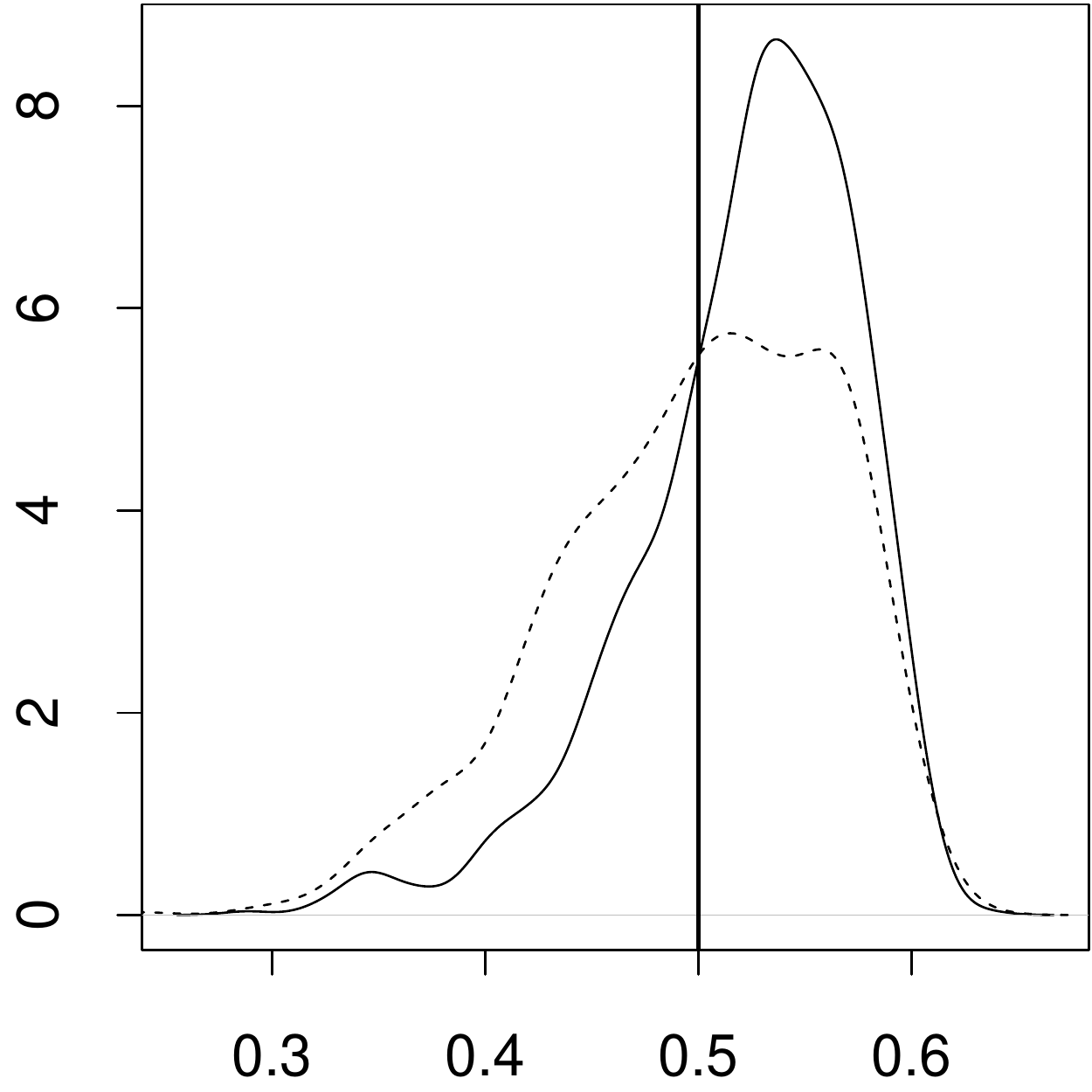}
  \caption{$\alpha_0=0.5$; $\rho = 0.75$}
\end{subfigure}
\begin{subfigure}{.24\textwidth}
  \centering
  \includegraphics[width=\textwidth]{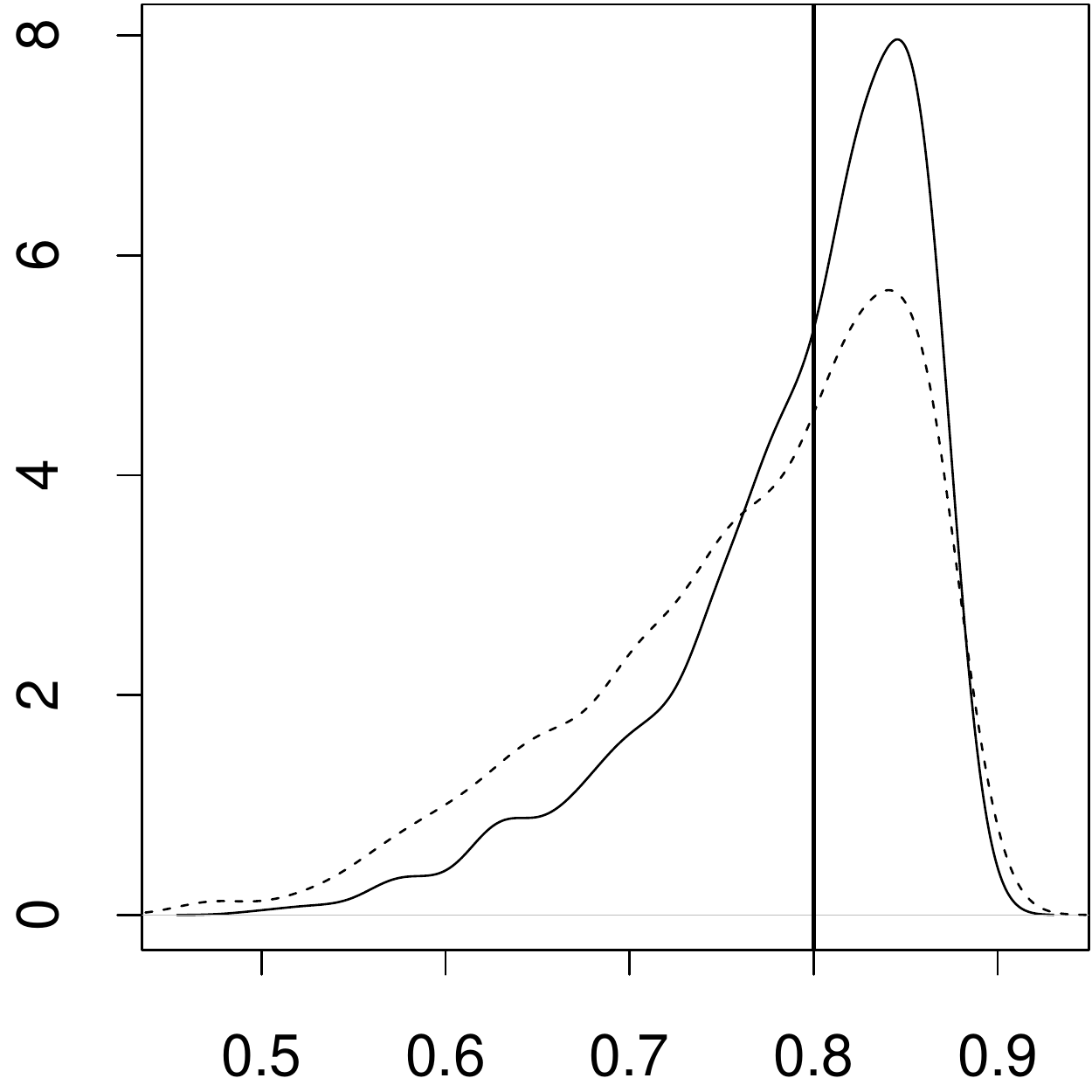}
  \caption{$\alpha_0=0.8$; $\rho = 0.75$}
\end{subfigure}
\begin{subfigure}{.24\textwidth}
  \centering
  \includegraphics[width=\textwidth]{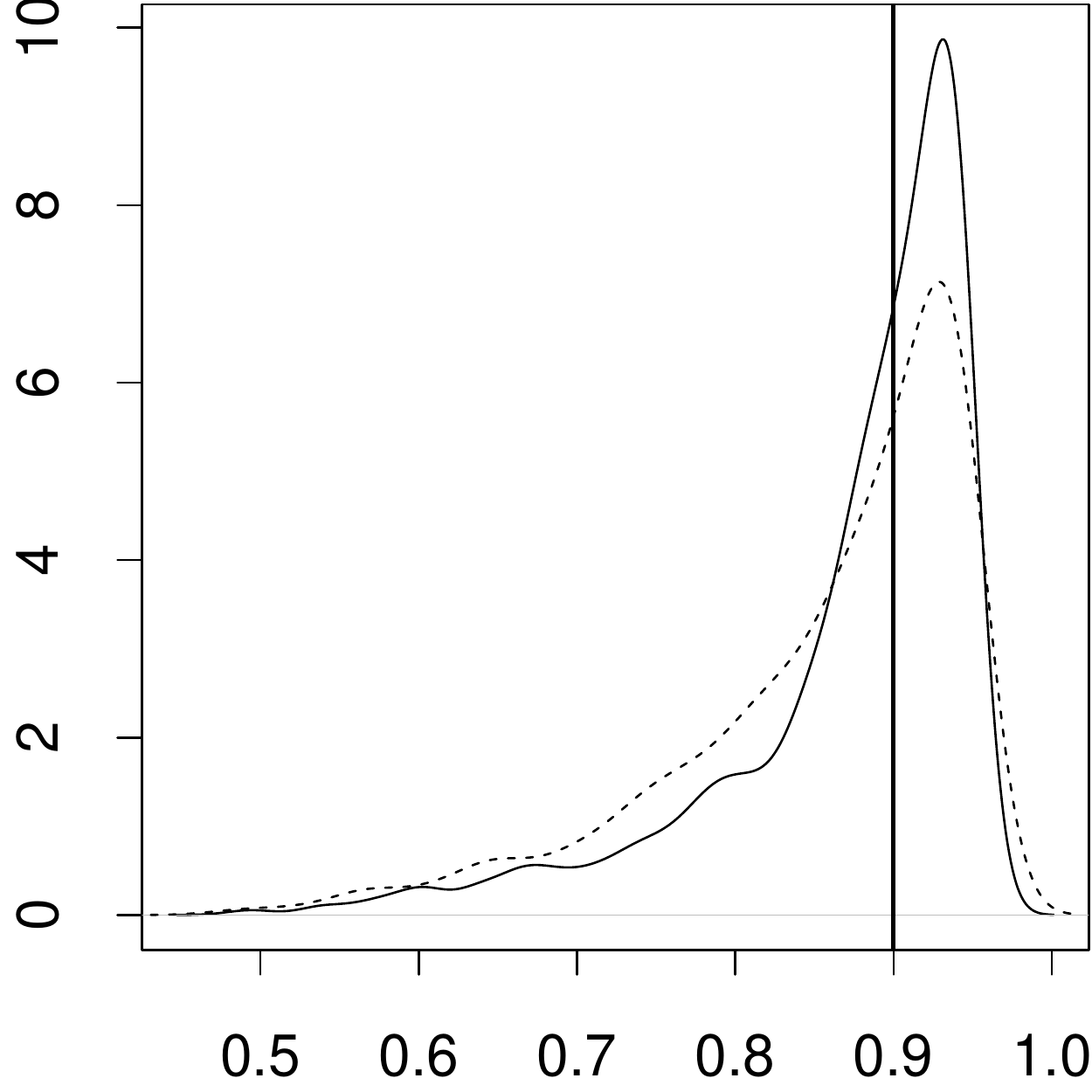}
  \caption{$\alpha_0=0.9$; $\rho = 0.75$}
\end{subfigure}
\begin{subfigure}{.24\textwidth}
  \centering
  \includegraphics[width=\textwidth]{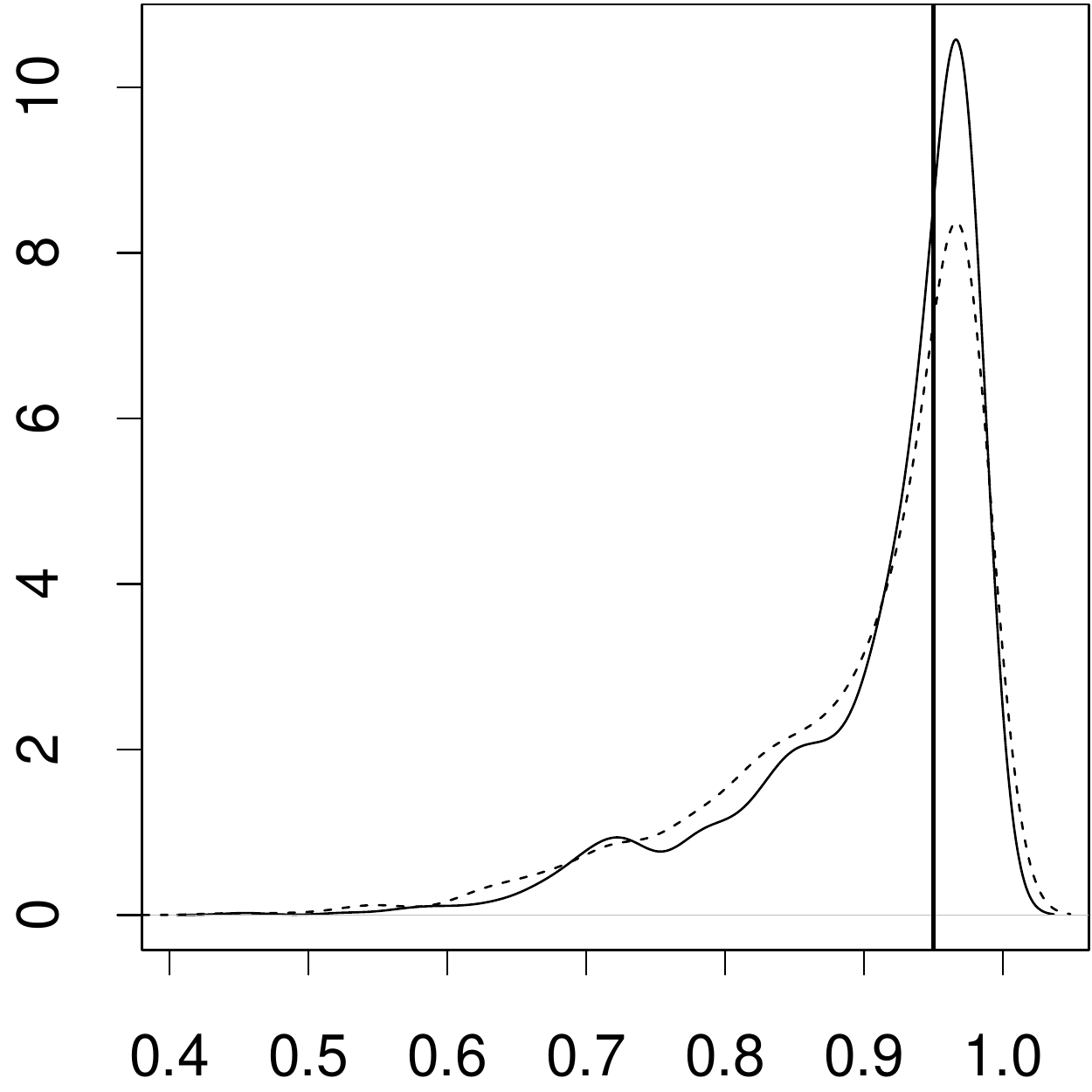}
  \caption{$\alpha_0=0.95$; $\rho = 0.75$}
\end{subfigure}
\caption{Density plots of the estimated $\alpha_0$ for $p$-values from two-sided t-tests ($G = 100$). The vertical line indicates the true $\alpha_0$, which is also marked below each figure, along with the within-group correlation coefficient. The solid lines are the densities of the posterior mean of $\alpha_0$, and the dashed lines are the densities of estimated $\alpha_0$ by fitting a convex decreasing density.}
\label{fig:g100}
\end{figure}

\begin{figure}[ht]
\captionsetup[subfigure]{labelformat=empty}
\centering
\begin{subfigure}{.24\textwidth}
  \centering
  \includegraphics[width=\textwidth]{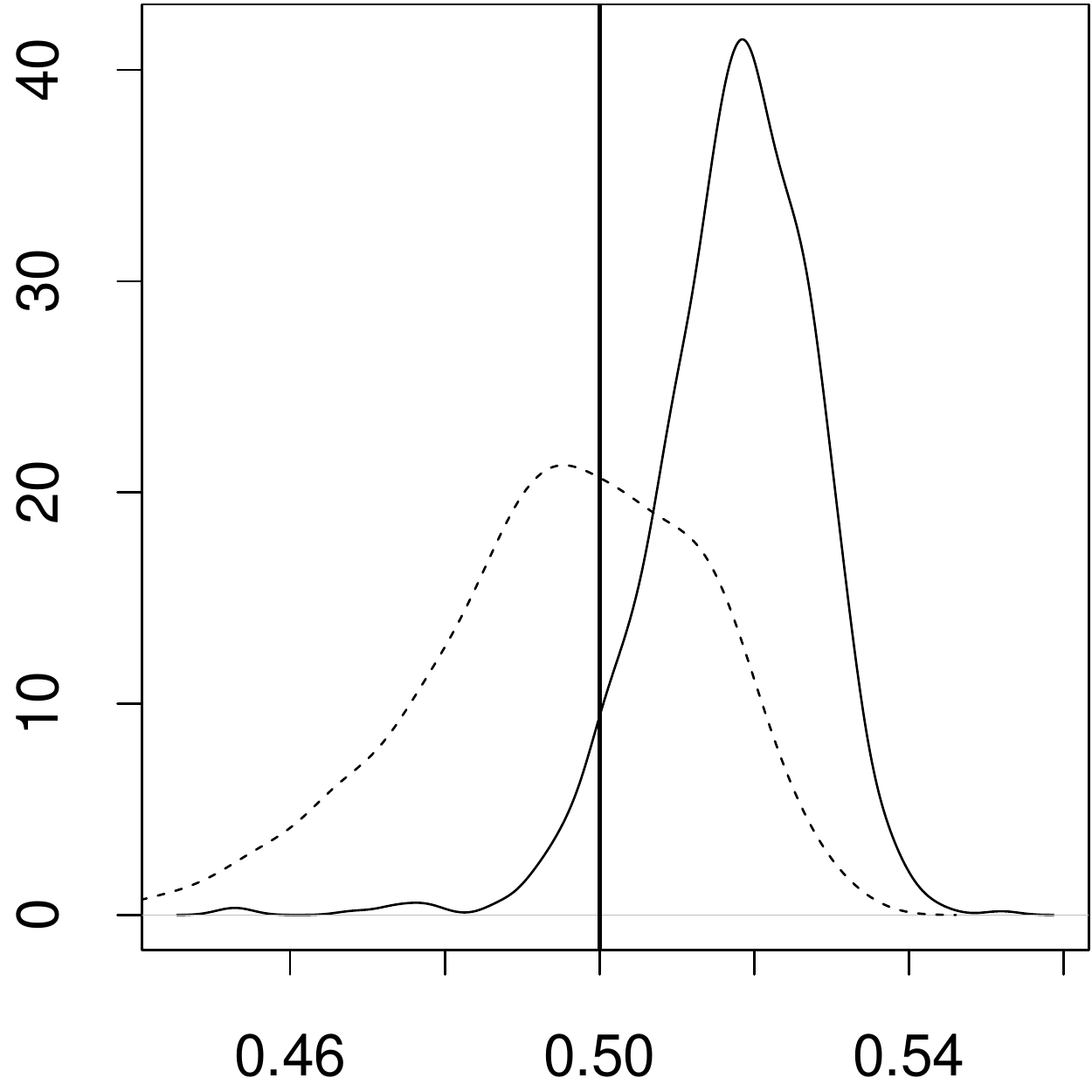}
  \caption{$\alpha_0=0.5$; $\rho = 0$}
\end{subfigure}
\begin{subfigure}{.24\textwidth}
  \centering
  \includegraphics[width=\textwidth]{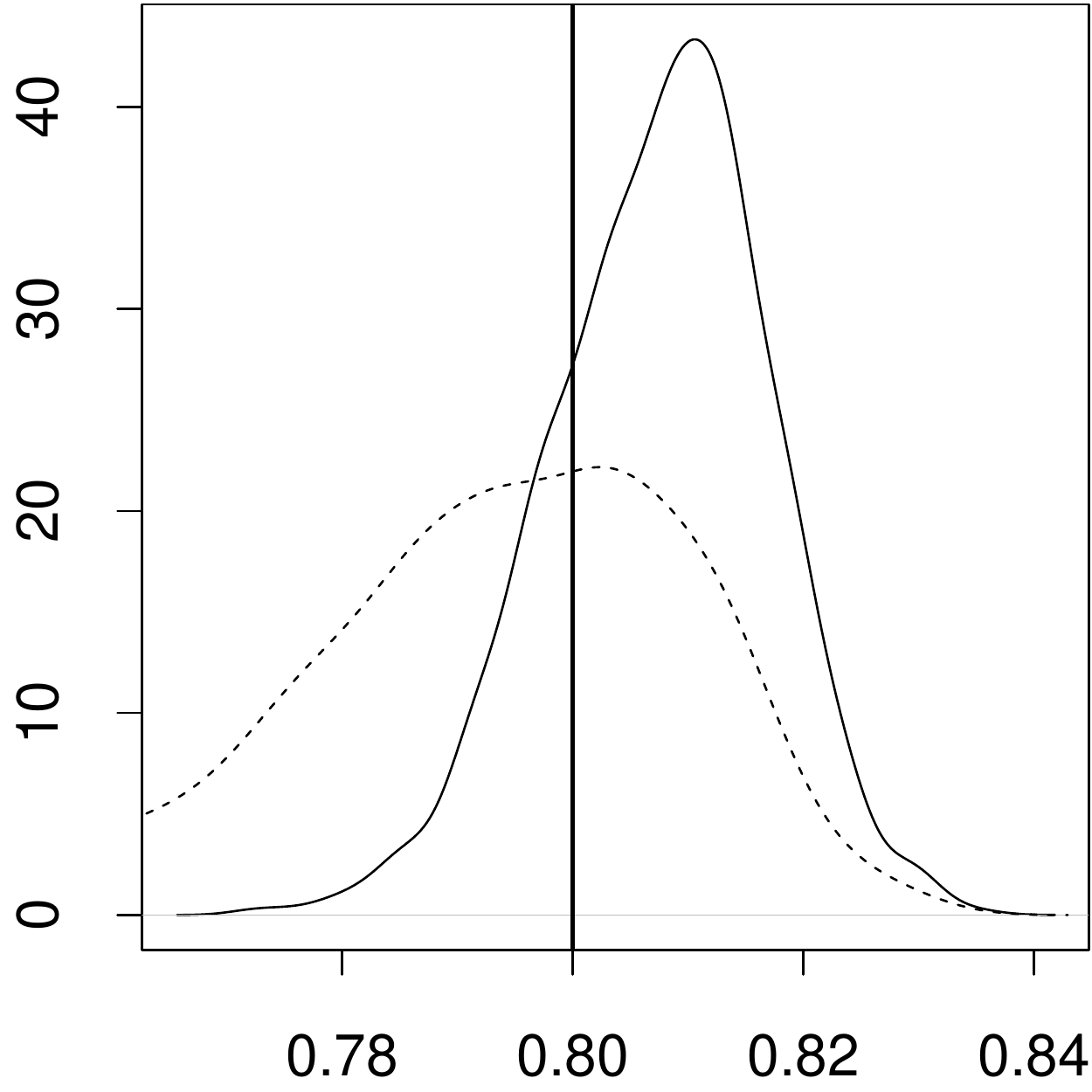}
  \caption{$\alpha_0=0.8$; $\rho = 0$}
\end{subfigure}
\begin{subfigure}{.24\textwidth}
  \centering
  \includegraphics[width=\textwidth]{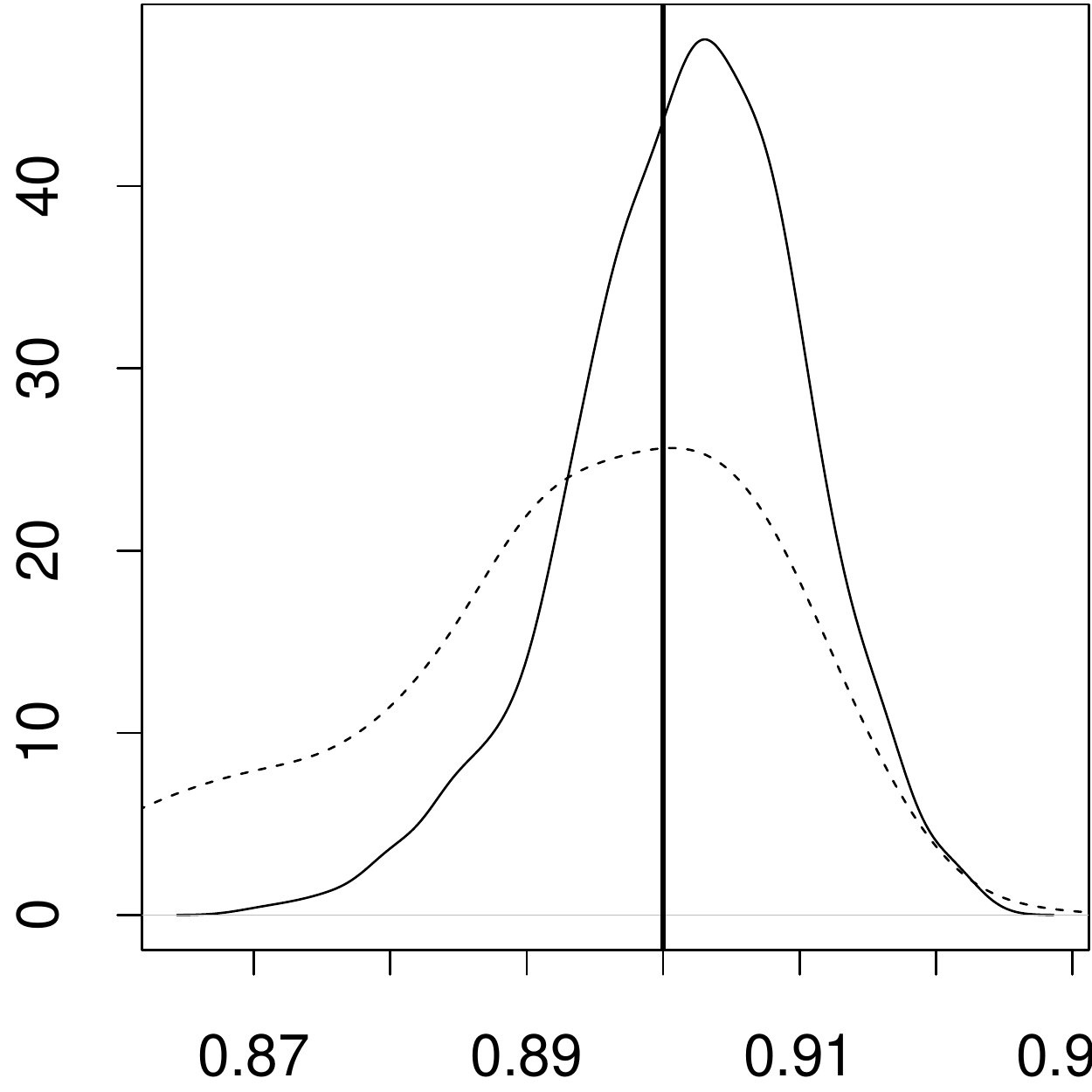}
  \caption{$\alpha_0=0.9$; $\rho = 0$}
\end{subfigure}
\begin{subfigure}{.24\textwidth}
  \centering
  \includegraphics[width=\textwidth]{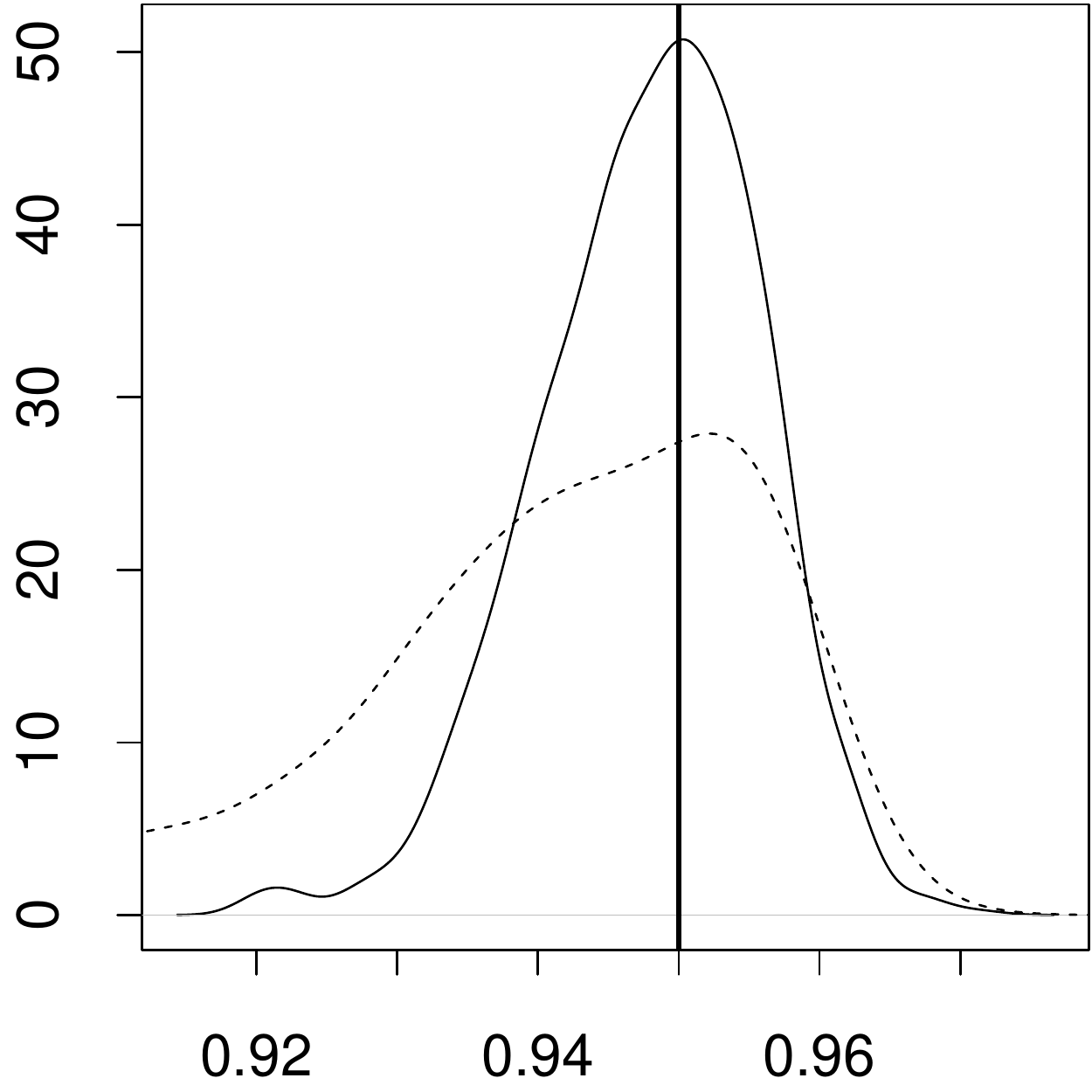}
  \caption{$\alpha_0=0.95$; $\rho = 0$}
\end{subfigure}

\begin{subfigure}{.24\textwidth}
  \centering
  \includegraphics[width=\textwidth]{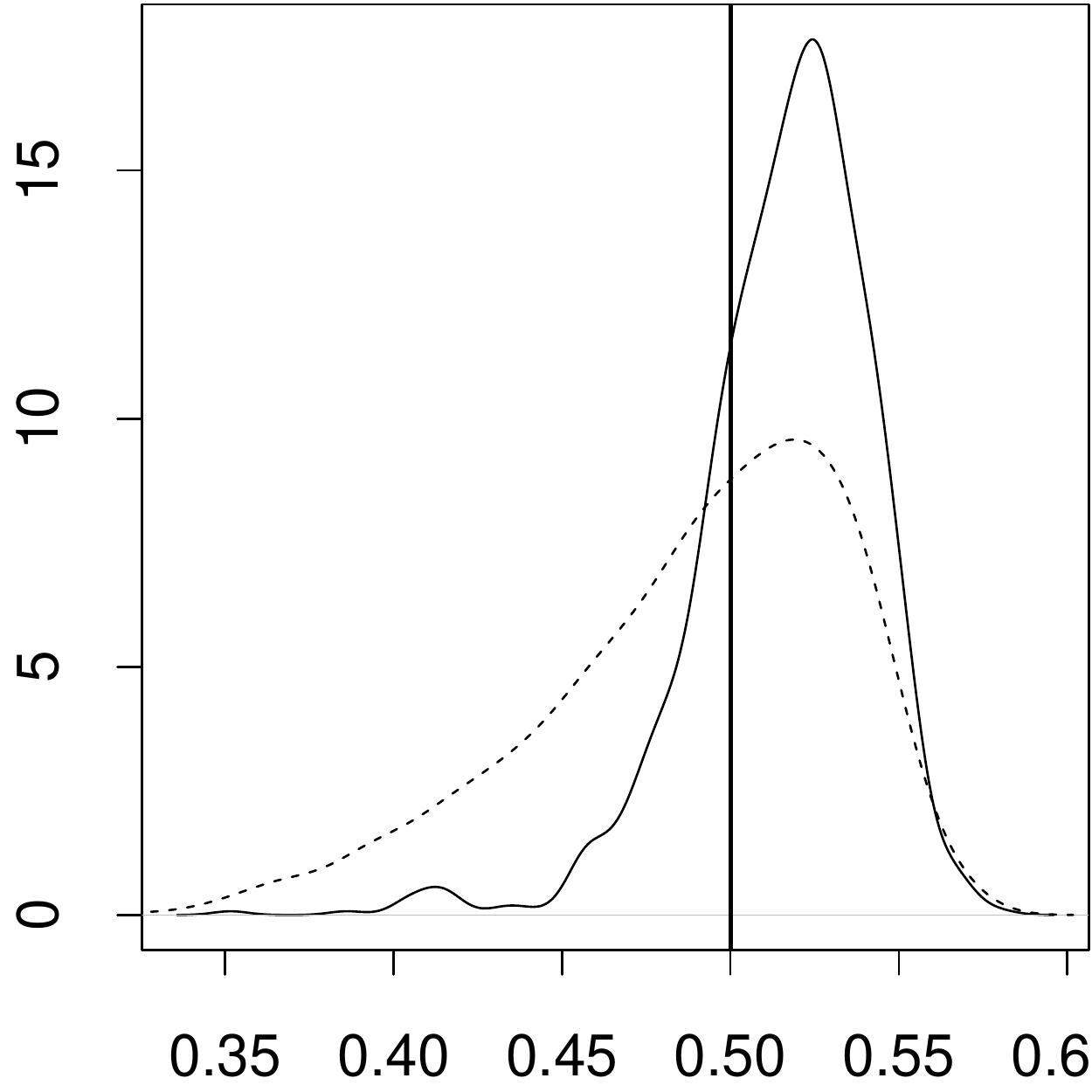}
  \caption{$\alpha_0=0.5$; $\rho = 0.25$}
\end{subfigure}
\begin{subfigure}{.24\textwidth}
  \centering
  \includegraphics[width=\textwidth]{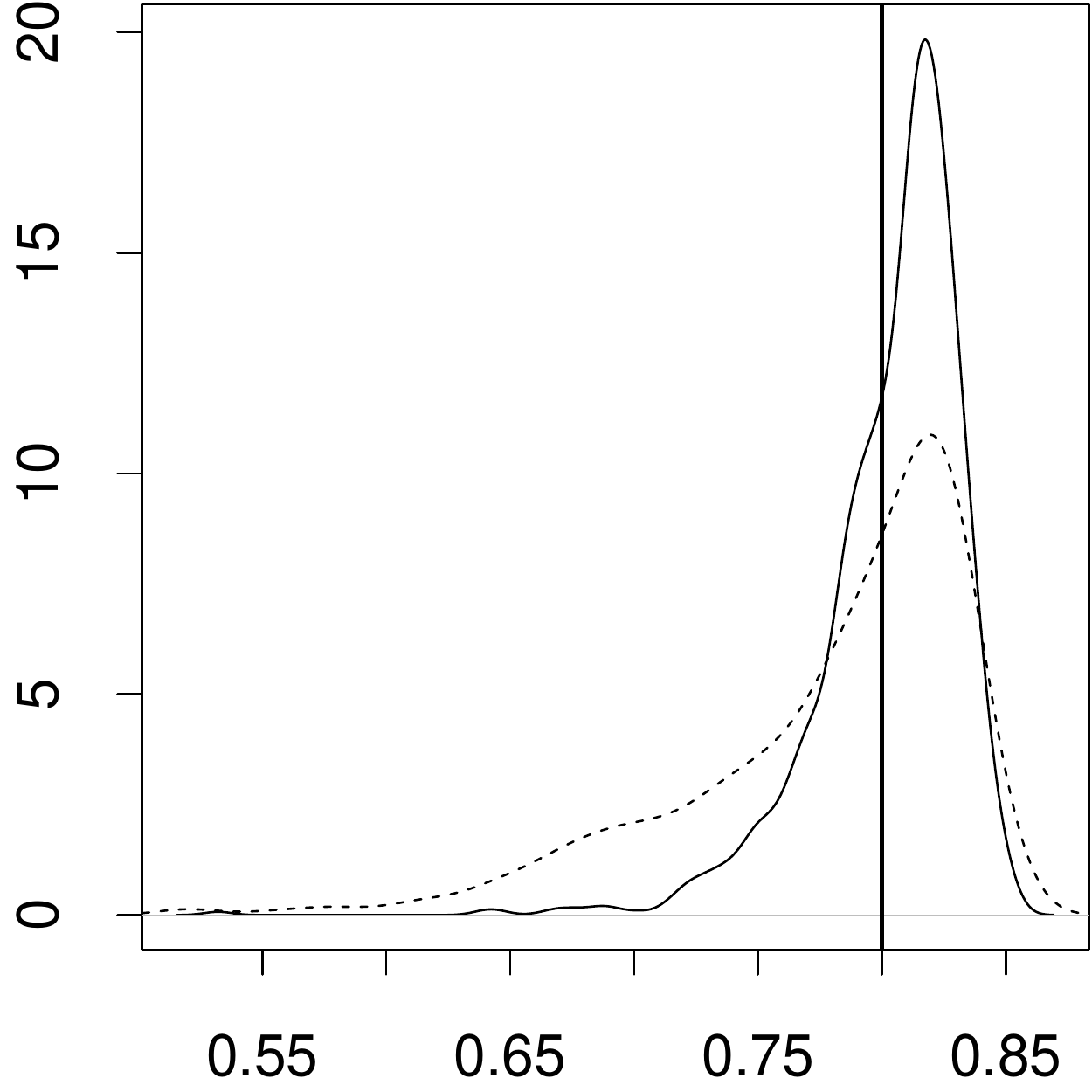}
  \caption{$\alpha_0=0.8$; $\rho = 0.25$}
\end{subfigure}
\begin{subfigure}{.24\textwidth}
  \centering
  \includegraphics[width=\textwidth]{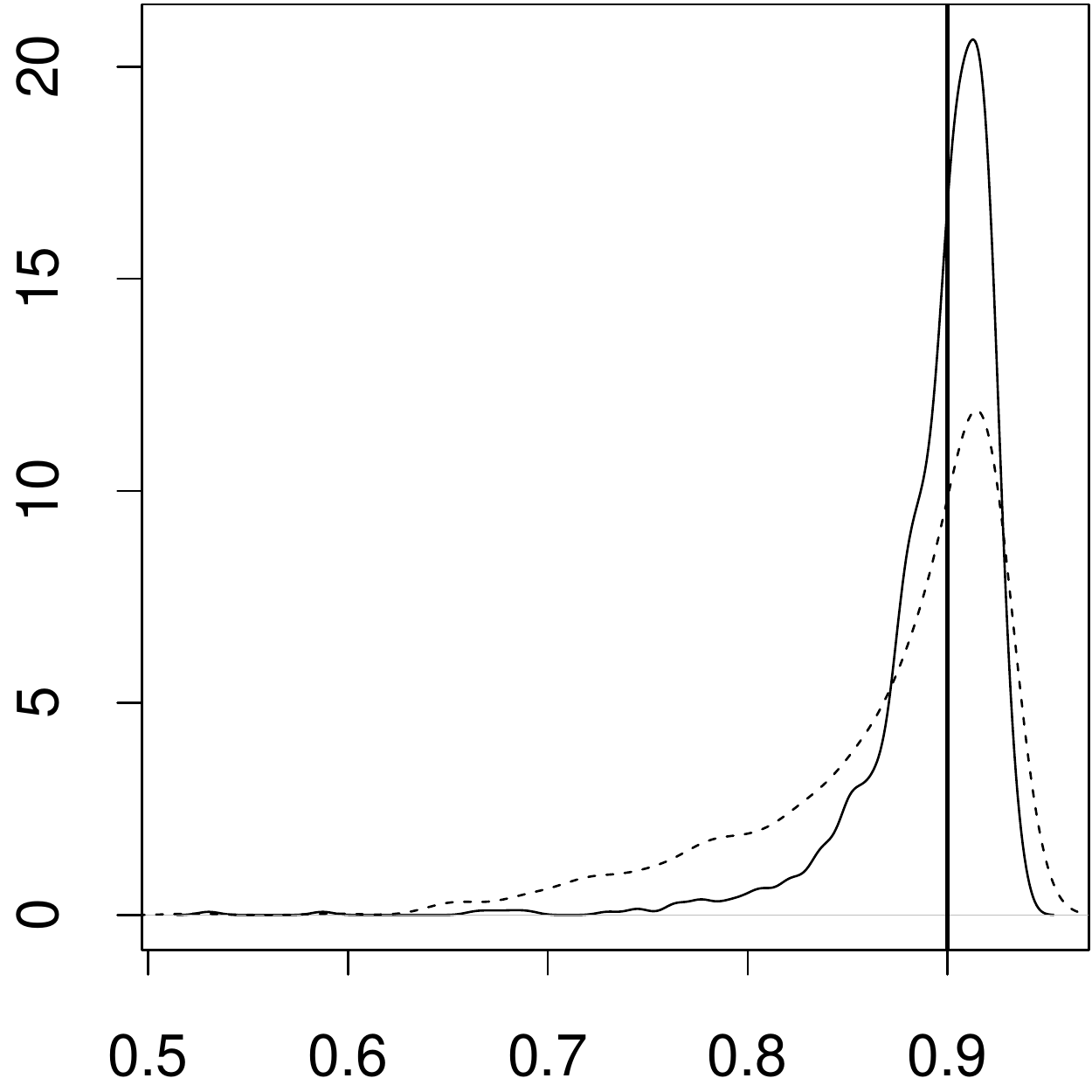}
  \caption{$\alpha_0=0.9$; $\rho = 0.25$}
\end{subfigure}
\begin{subfigure}{.24\textwidth}
  \centering
  \includegraphics[width=\textwidth]{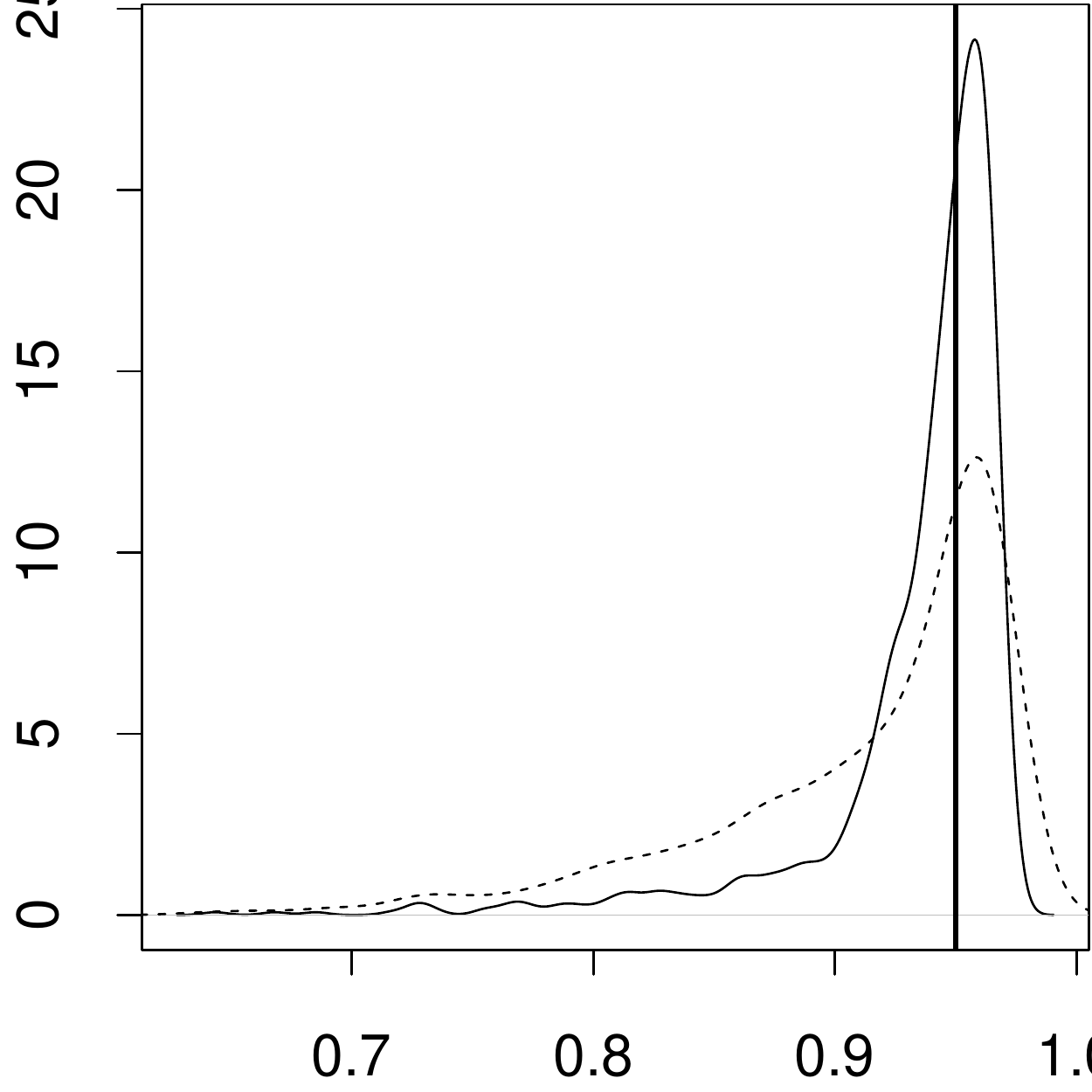}
  \caption{$\alpha_0=0.95$; $\rho = 0.25$}
\end{subfigure}

\begin{subfigure}{.24\textwidth}
  \centering
  \includegraphics[width=\textwidth]{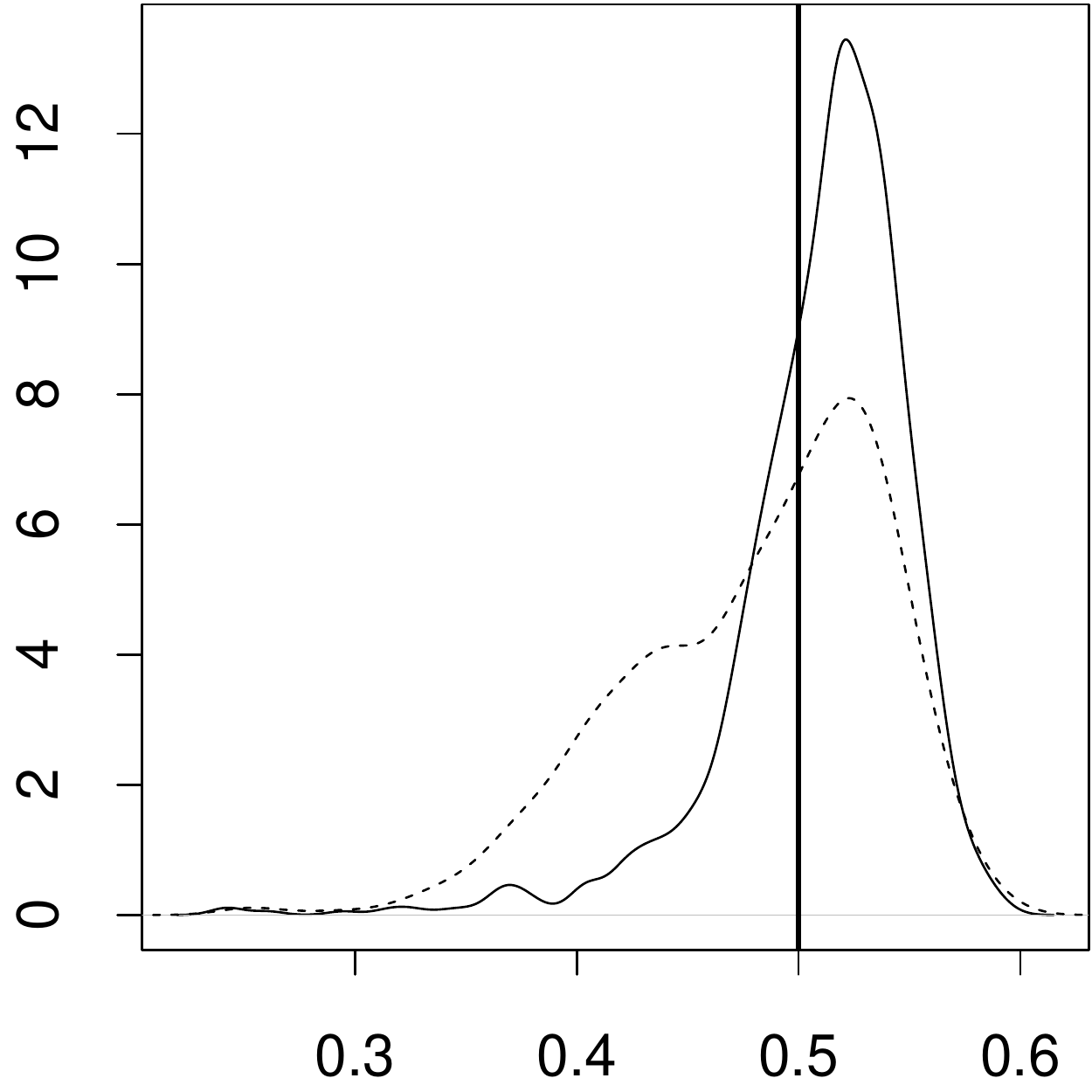}
  \caption{$\alpha_0=0.5$; $\rho = 0.5$}
\end{subfigure}
\begin{subfigure}{.24\textwidth}
  \centering
  \includegraphics[width=\textwidth]{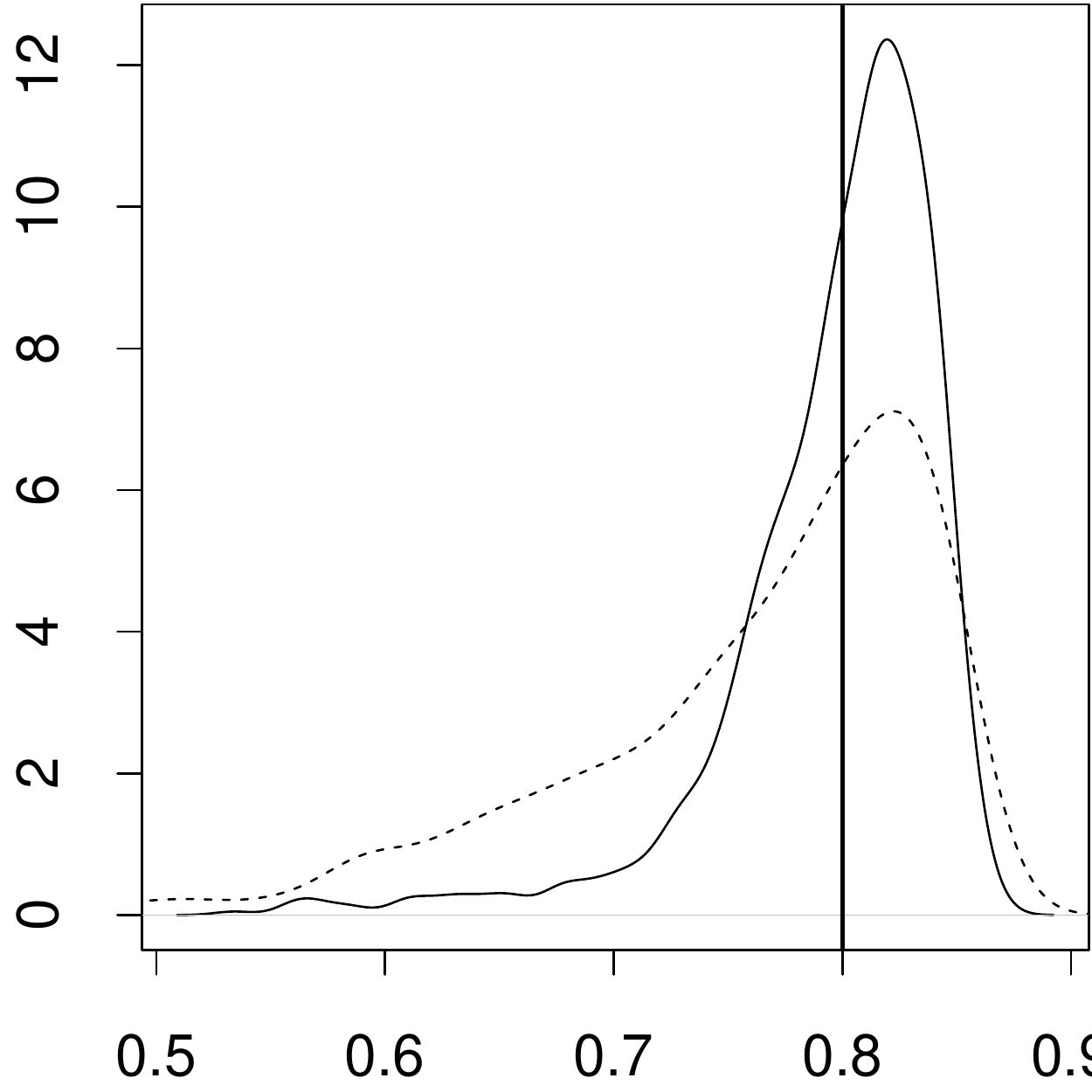}
  \caption{$\alpha_0=0.8$; $\rho = 0.5$}
\end{subfigure}
\begin{subfigure}{.24\textwidth}
  \centering
  \includegraphics[width=\textwidth]{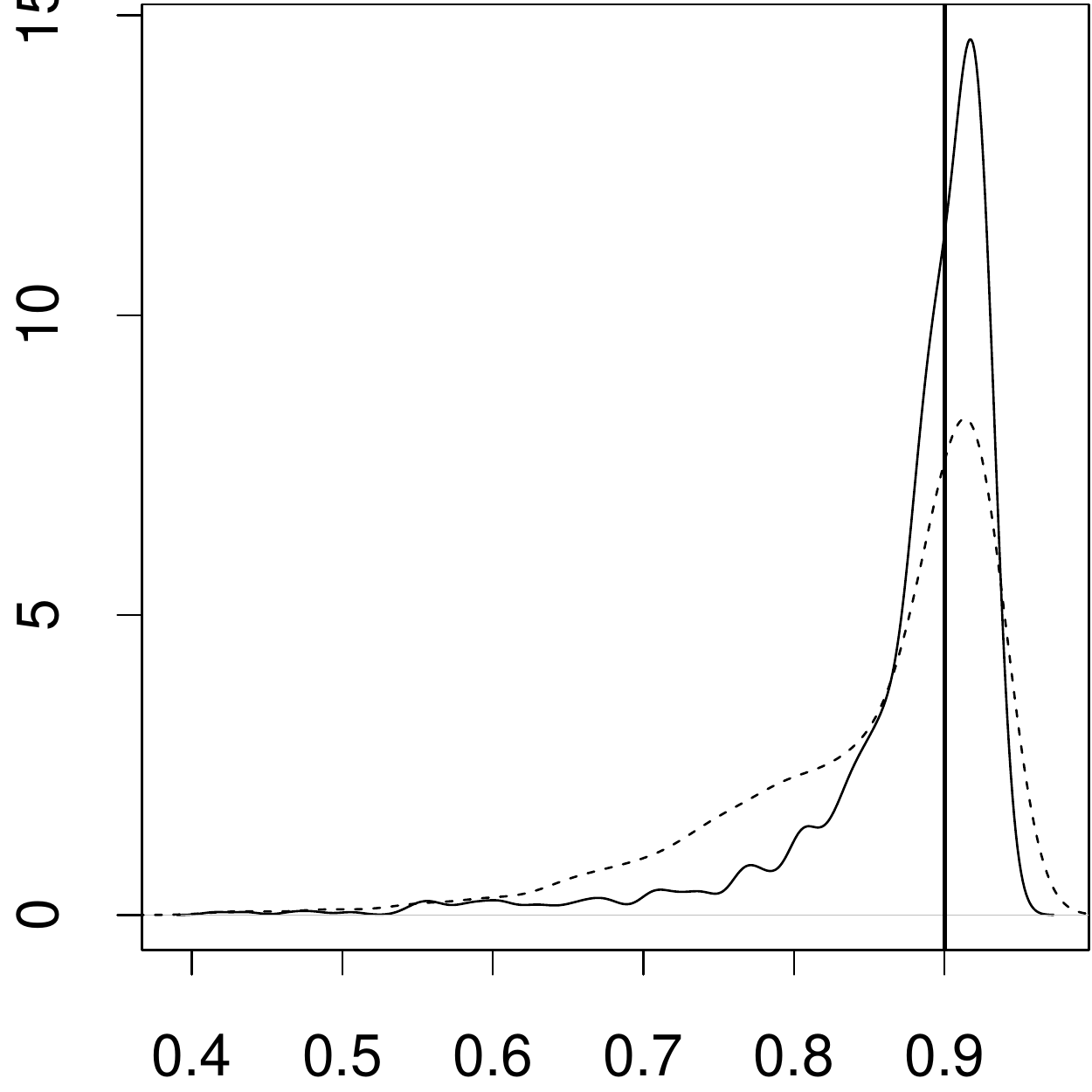}
  \caption{$\alpha_0=0.9$; $\rho = 0.5$}
\end{subfigure}
\begin{subfigure}{.24\textwidth}
  \centering
  \includegraphics[width=\textwidth]{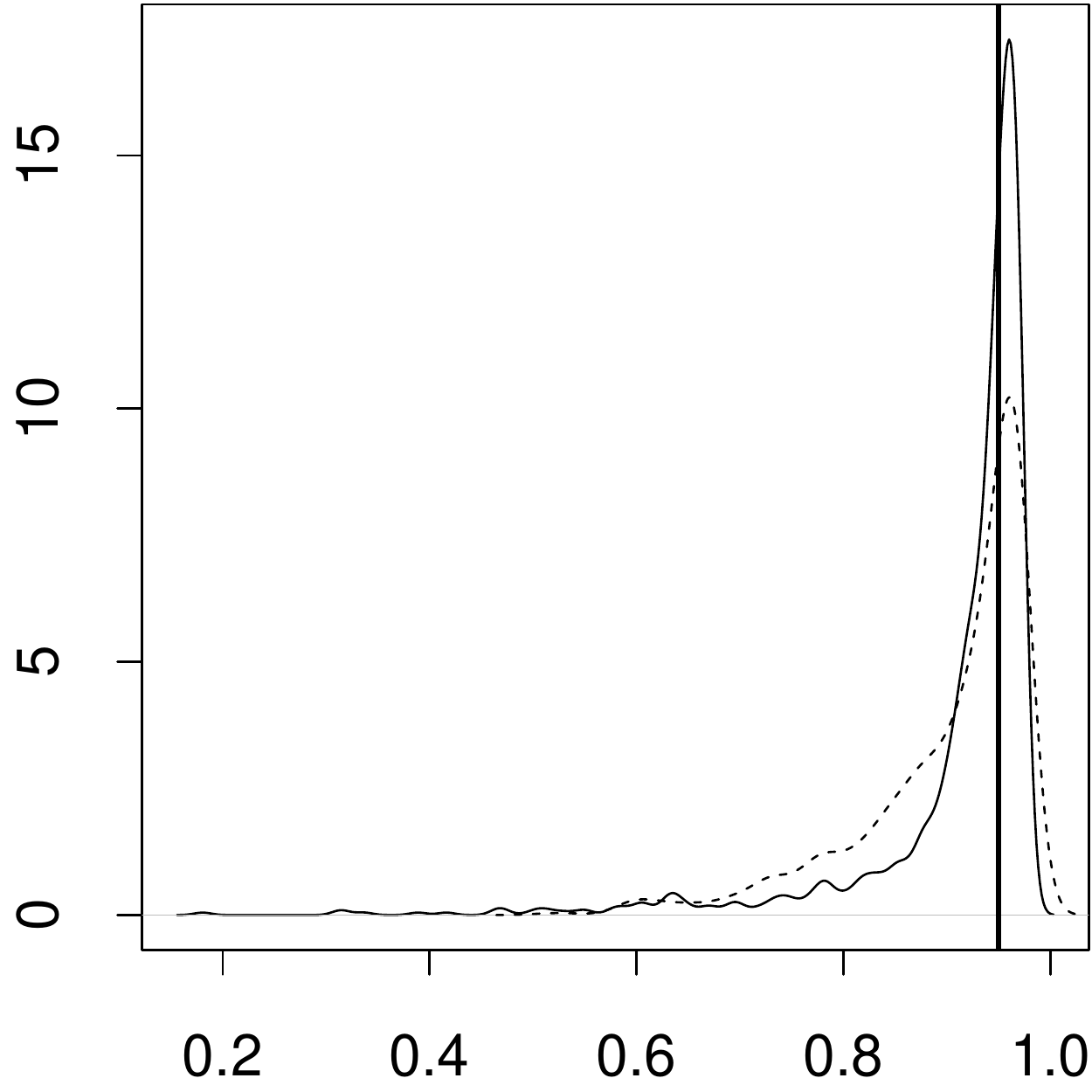}
  \caption{$\alpha_0=0.95$; $\rho = 0.5$}
\end{subfigure}

\begin{subfigure}{.24\textwidth}
  \centering
  \includegraphics[width=\textwidth]{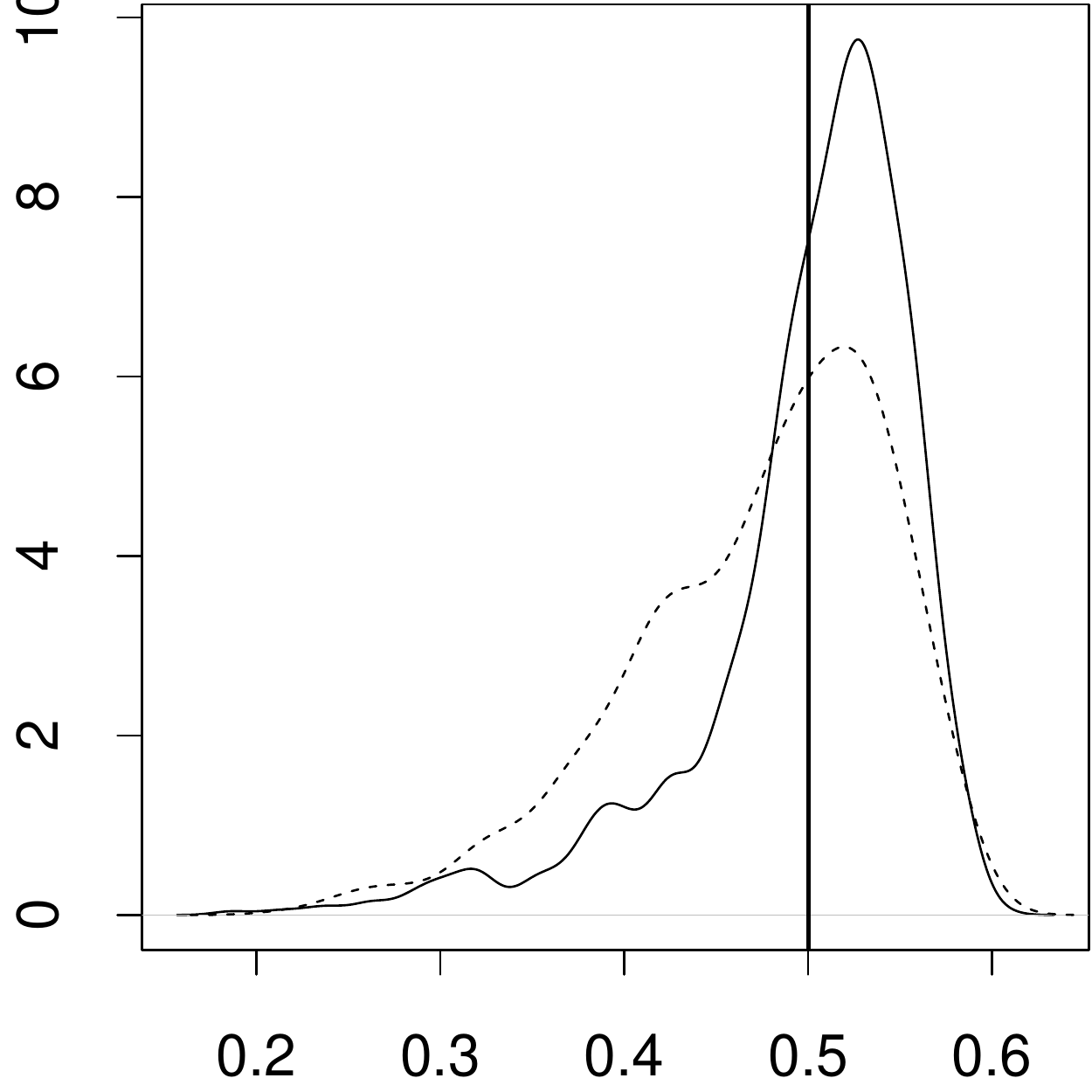}
  \caption{$\alpha_0=0.5$; $\rho = 0.75$}
\end{subfigure}
\begin{subfigure}{.24\textwidth}
  \centering
  \includegraphics[width=\textwidth]{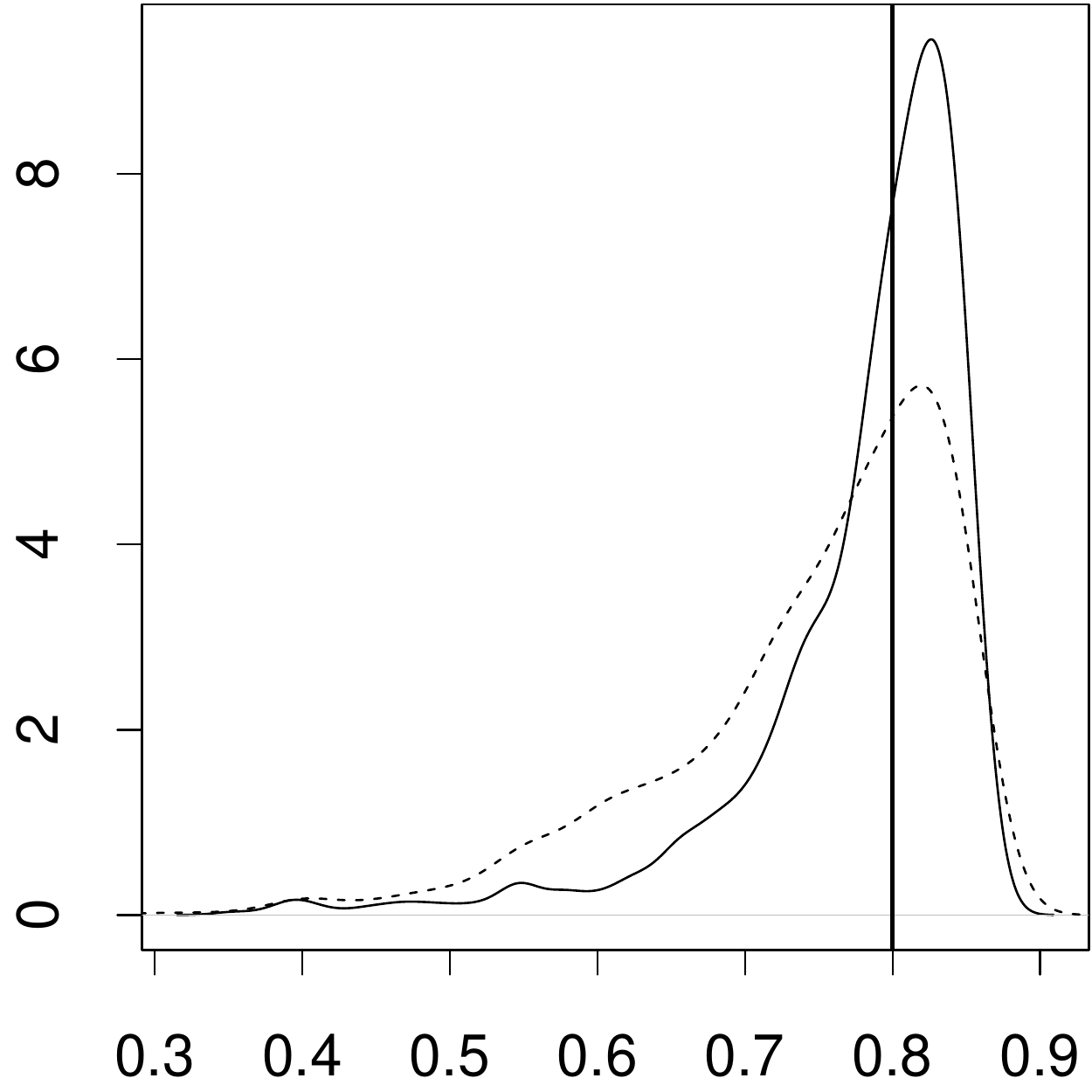}
  \caption{$\alpha_0=0.8$; $\rho = 0.75$}
\end{subfigure}
\begin{subfigure}{.24\textwidth}
  \centering
  \includegraphics[width=\textwidth]{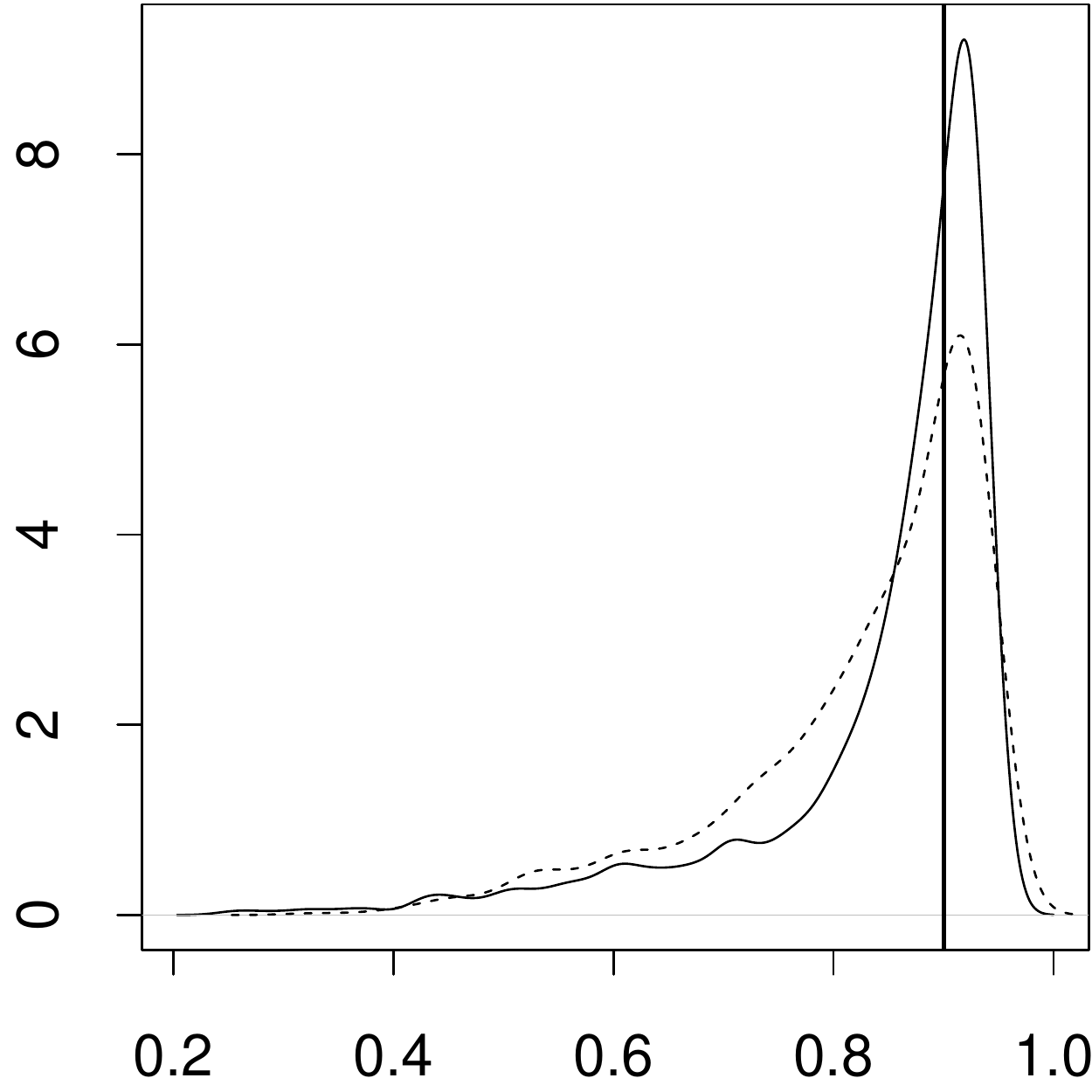}
  \caption{$\alpha_0=0.9$; $\rho = 0.75$}
\end{subfigure}
\begin{subfigure}{.24\textwidth}
  \centering
  \includegraphics[width=\textwidth]{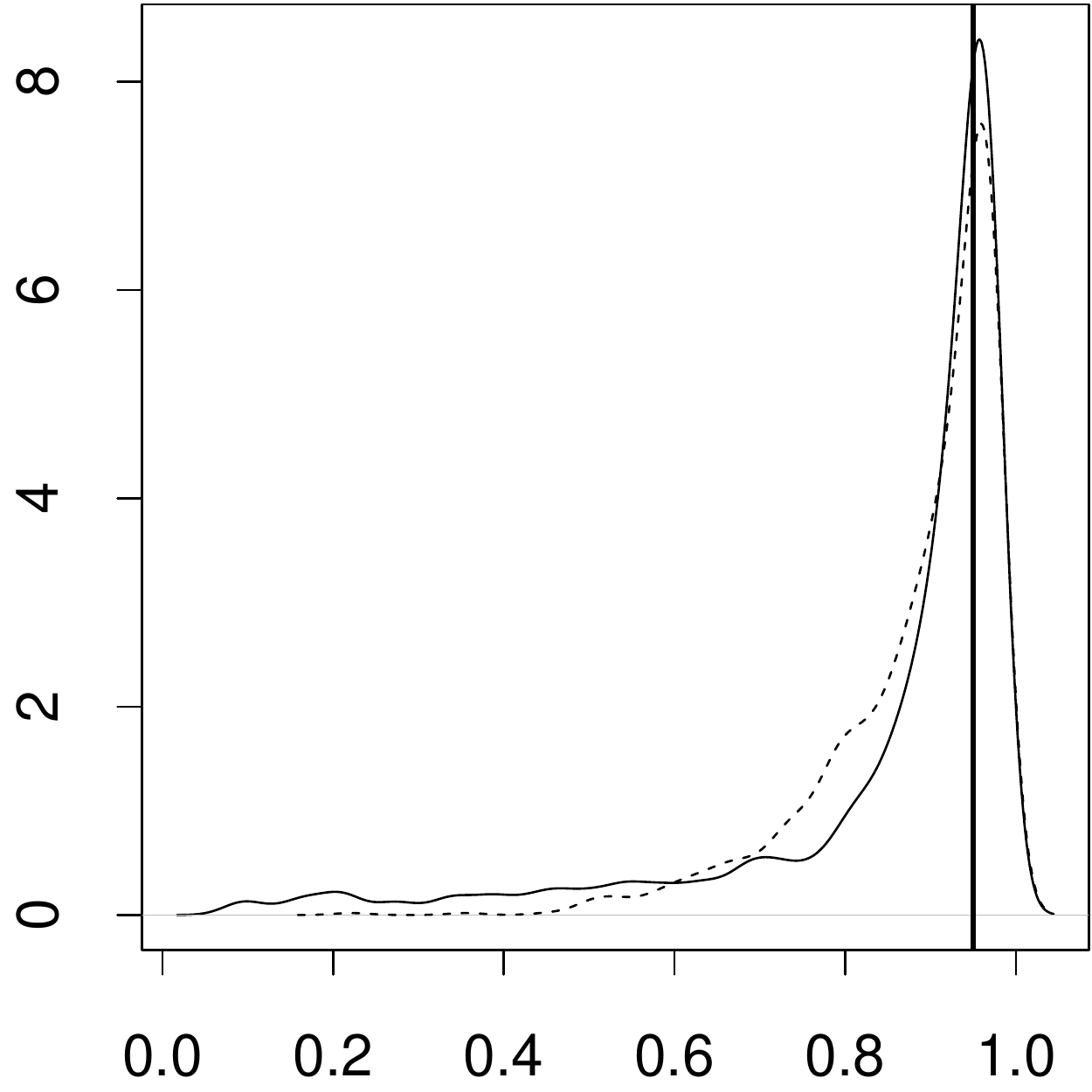}
  \caption{$\alpha_0=0.95$; $\rho = 0.75$}
\end{subfigure}
\caption{Density plots of the estimated $\alpha_0$ for $p$-values from one-sided t-tests ($G = 50$). The vertical line indicates the true $\alpha_0$, which is also marked below each figure, along with the within-group correlation coefficient. The solid lines are the densities of the posterior mean of $\alpha_0$, and the dashed lines are the densities of estimated $\alpha_0$ by fitting a convex decreasing density.}
\label{fig:osg50}
\end{figure}

\begin{figure}[ht]
\captionsetup[subfigure]{labelformat=empty}
\centering
\begin{subfigure}{.24\textwidth}
  \centering
  \includegraphics[width=\textwidth]{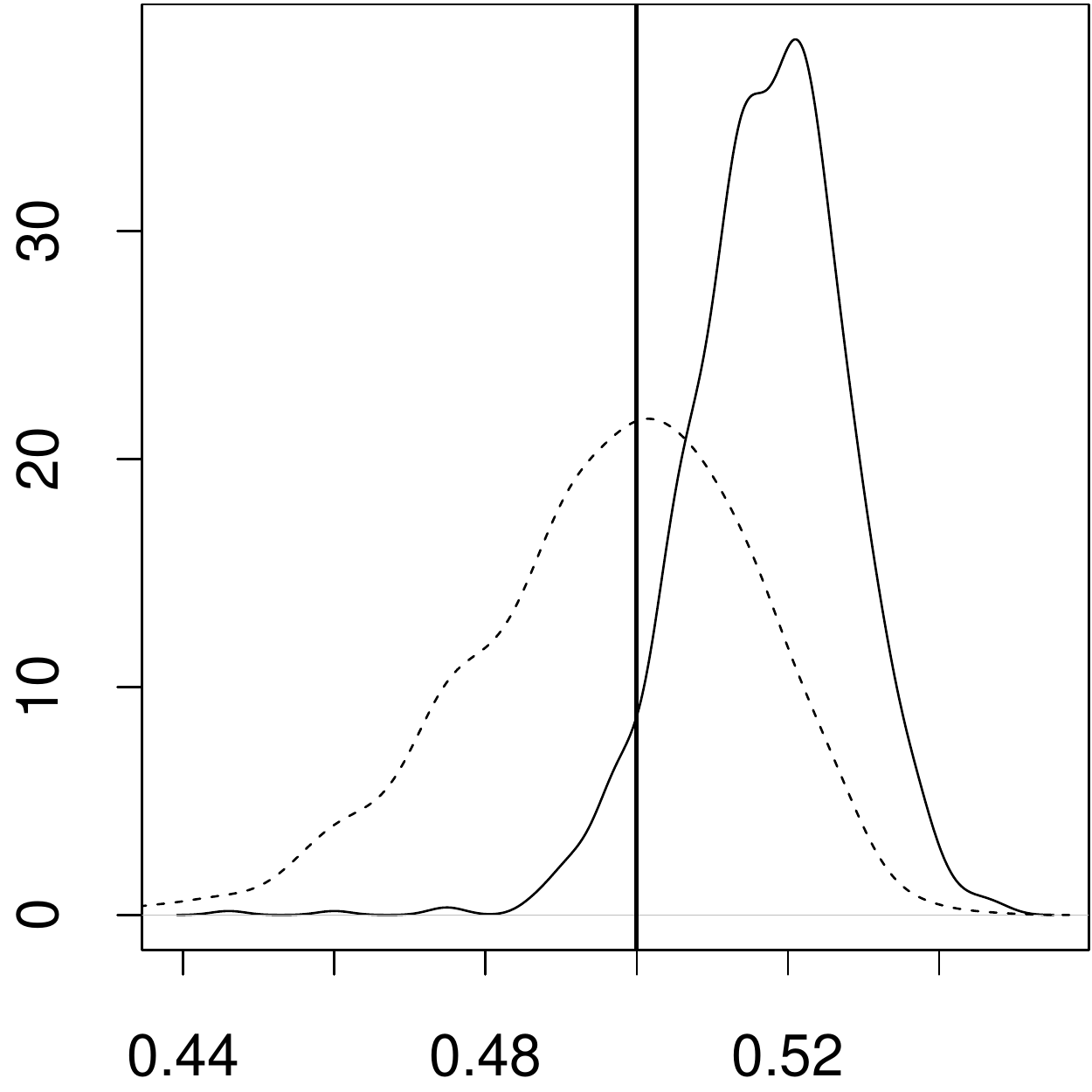}
  \caption{$\alpha_0=0.5$; $\rho = 0$}
\end{subfigure}
\begin{subfigure}{.24\textwidth}
  \centering
  \includegraphics[width=\textwidth]{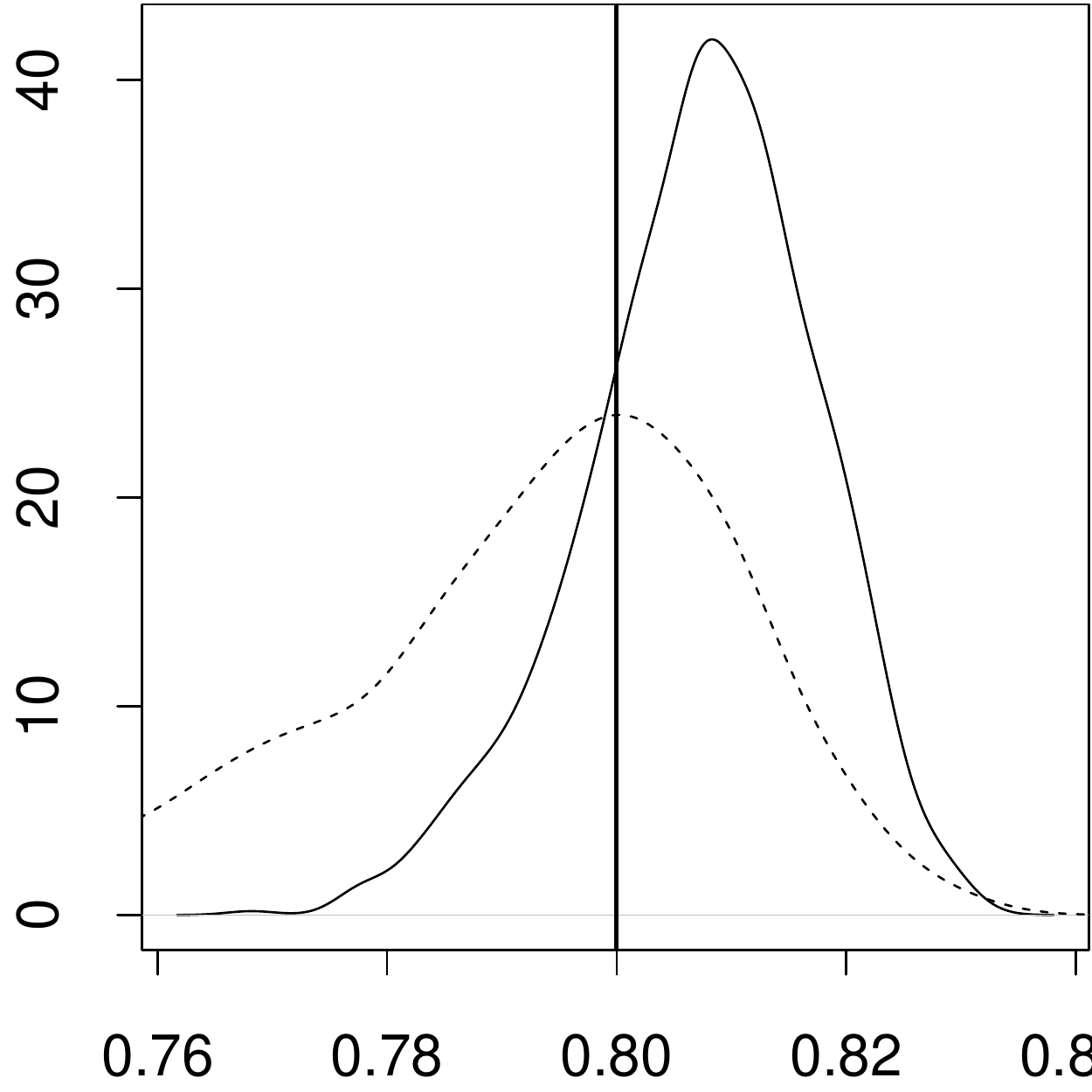}
  \caption{$\alpha_0=0.8$; $\rho = 0$}
\end{subfigure}
\begin{subfigure}{.24\textwidth}
  \centering
  \includegraphics[width=\textwidth]{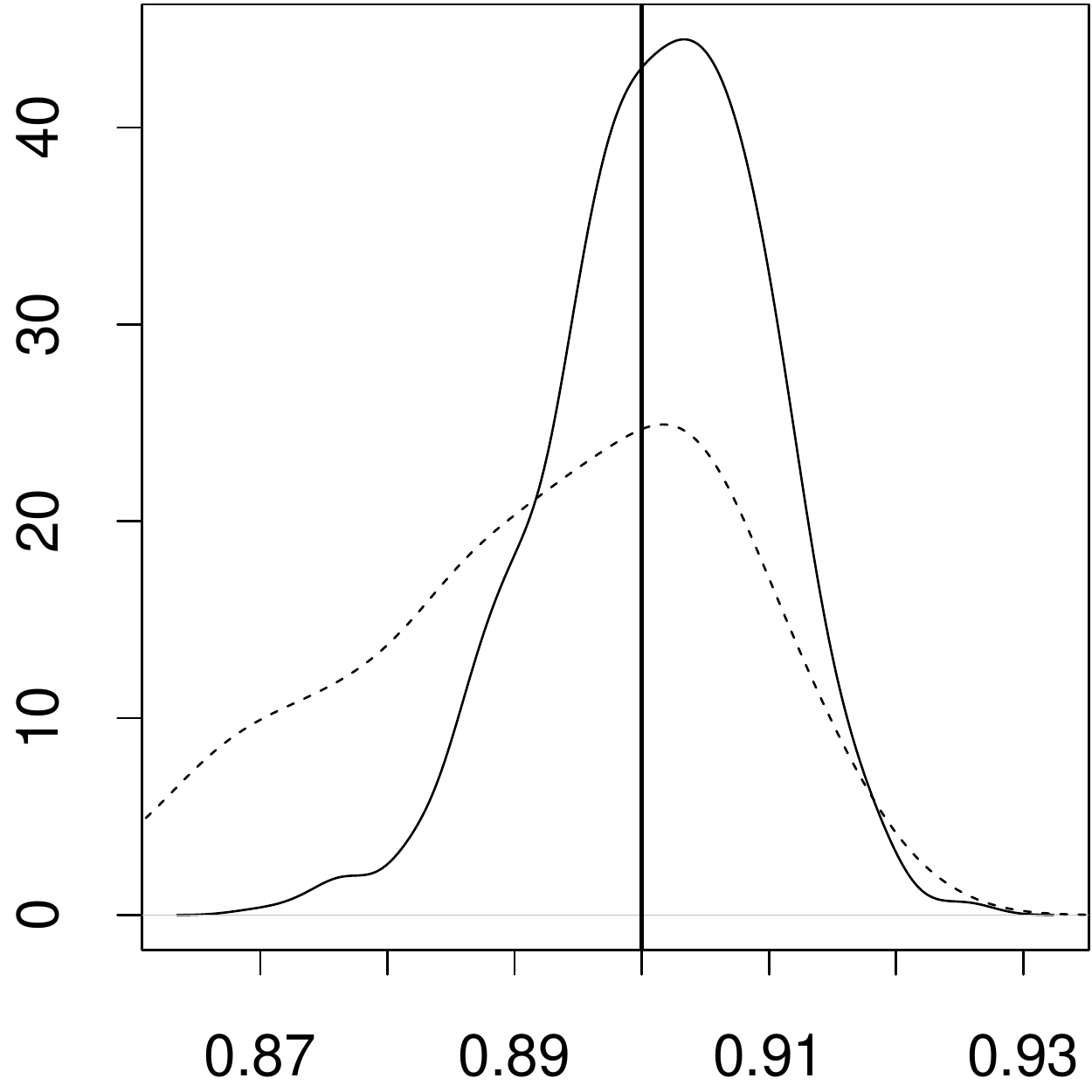}
  \caption{$\alpha_0=0.9$; $\rho = 0$}
\end{subfigure}
\begin{subfigure}{.24\textwidth}
  \centering
  \includegraphics[width=\textwidth]{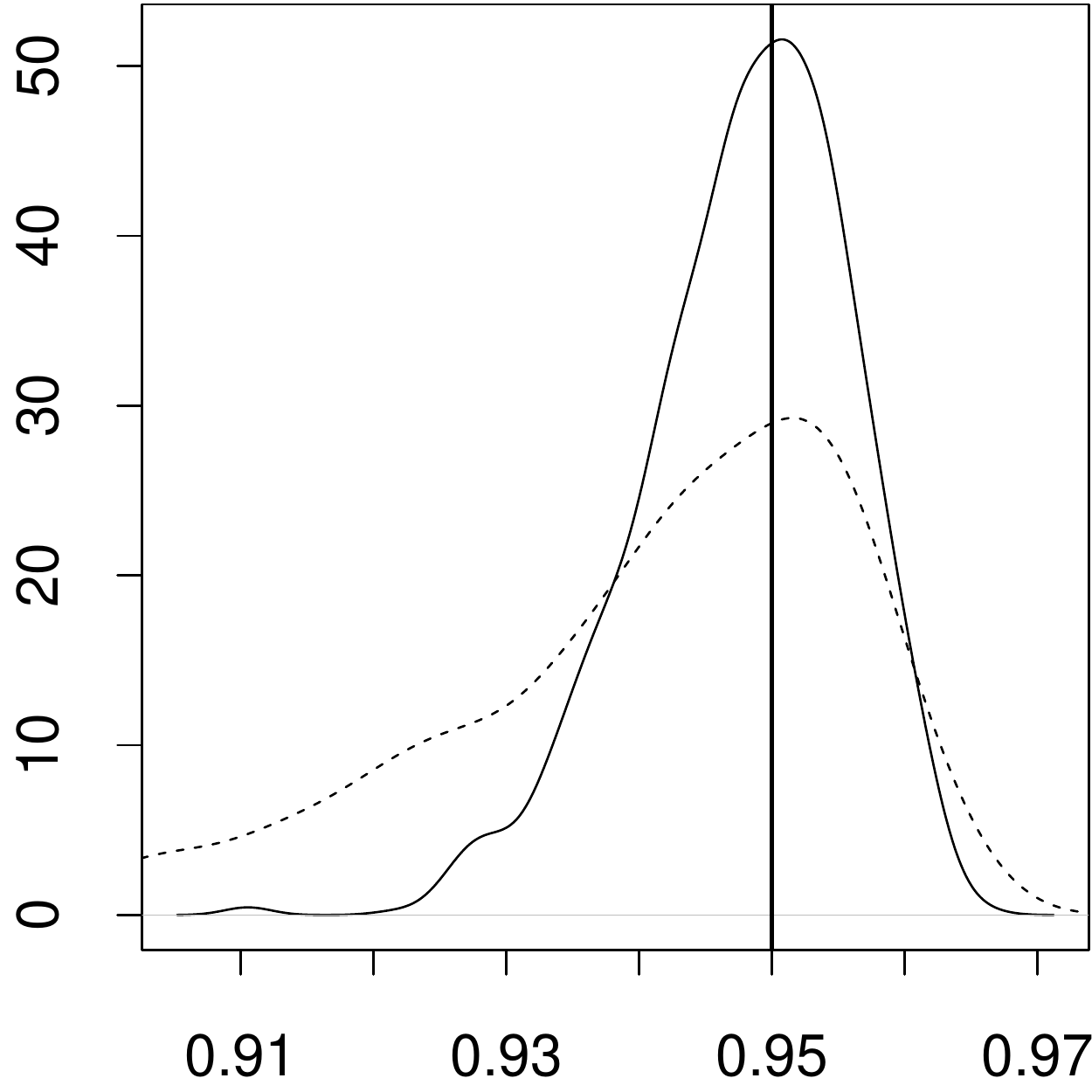}
  \caption{$\alpha_0=0.95$; $\rho = 0$}
\end{subfigure}

\begin{subfigure}{.24\textwidth}
  \centering
  \includegraphics[width=\textwidth]{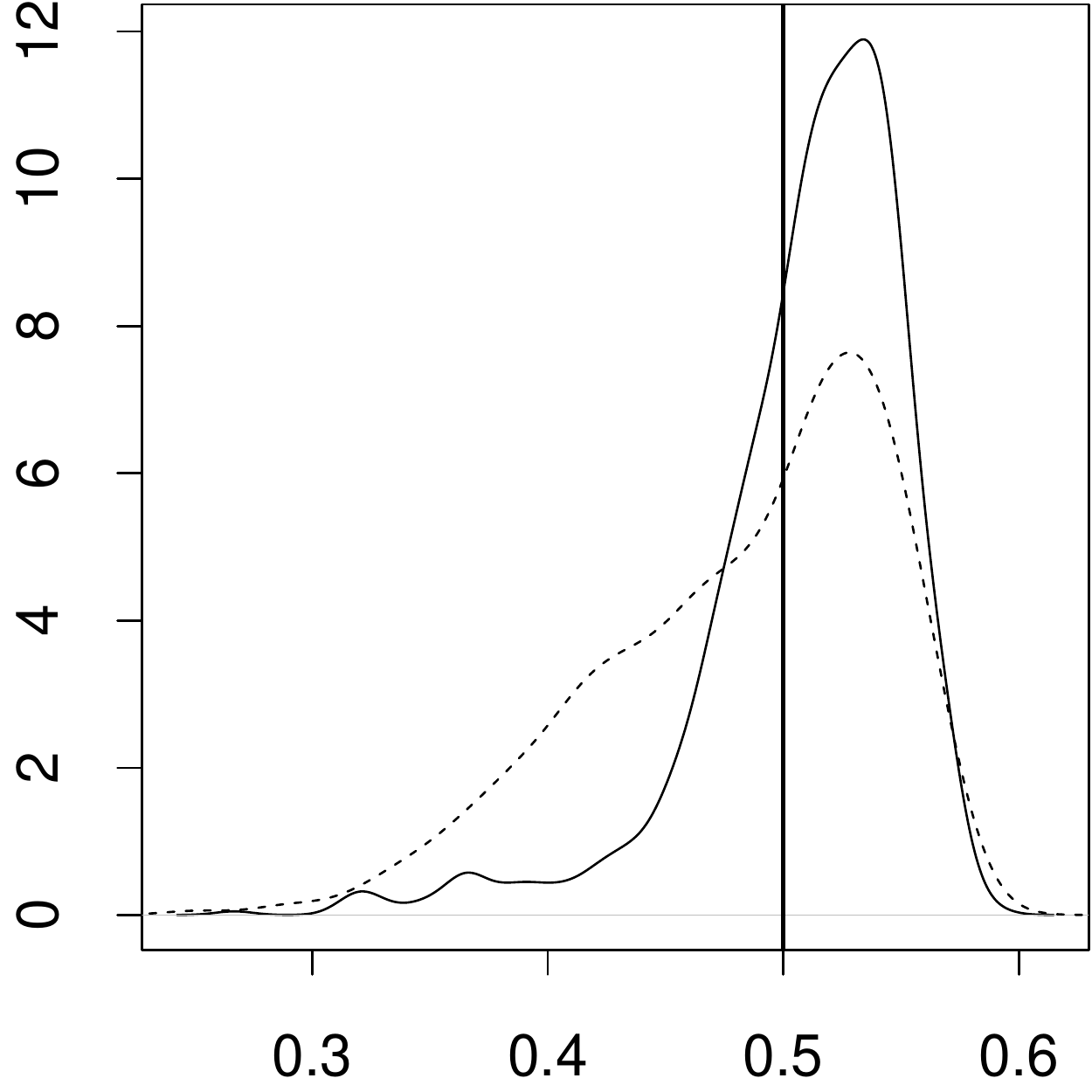}
  \caption{$\alpha_0=0.5$; $\rho = 0.25$}
\end{subfigure}
\begin{subfigure}{.24\textwidth}
  \centering
  \includegraphics[width=\textwidth]{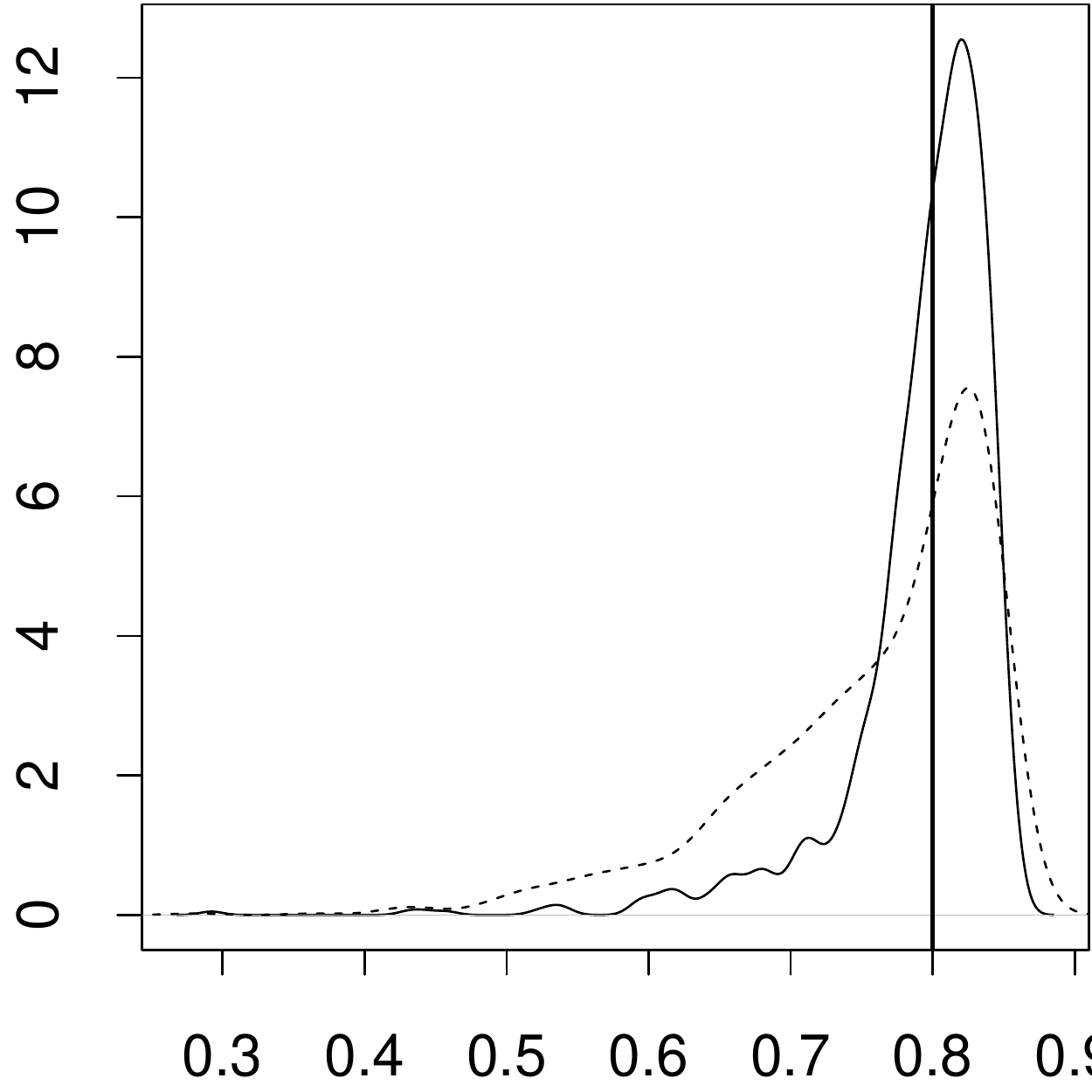}
  \caption{$\alpha_0=0.8$; $\rho = 0.25$}
\end{subfigure}
\begin{subfigure}{.24\textwidth}
  \centering
  \includegraphics[width=\textwidth]{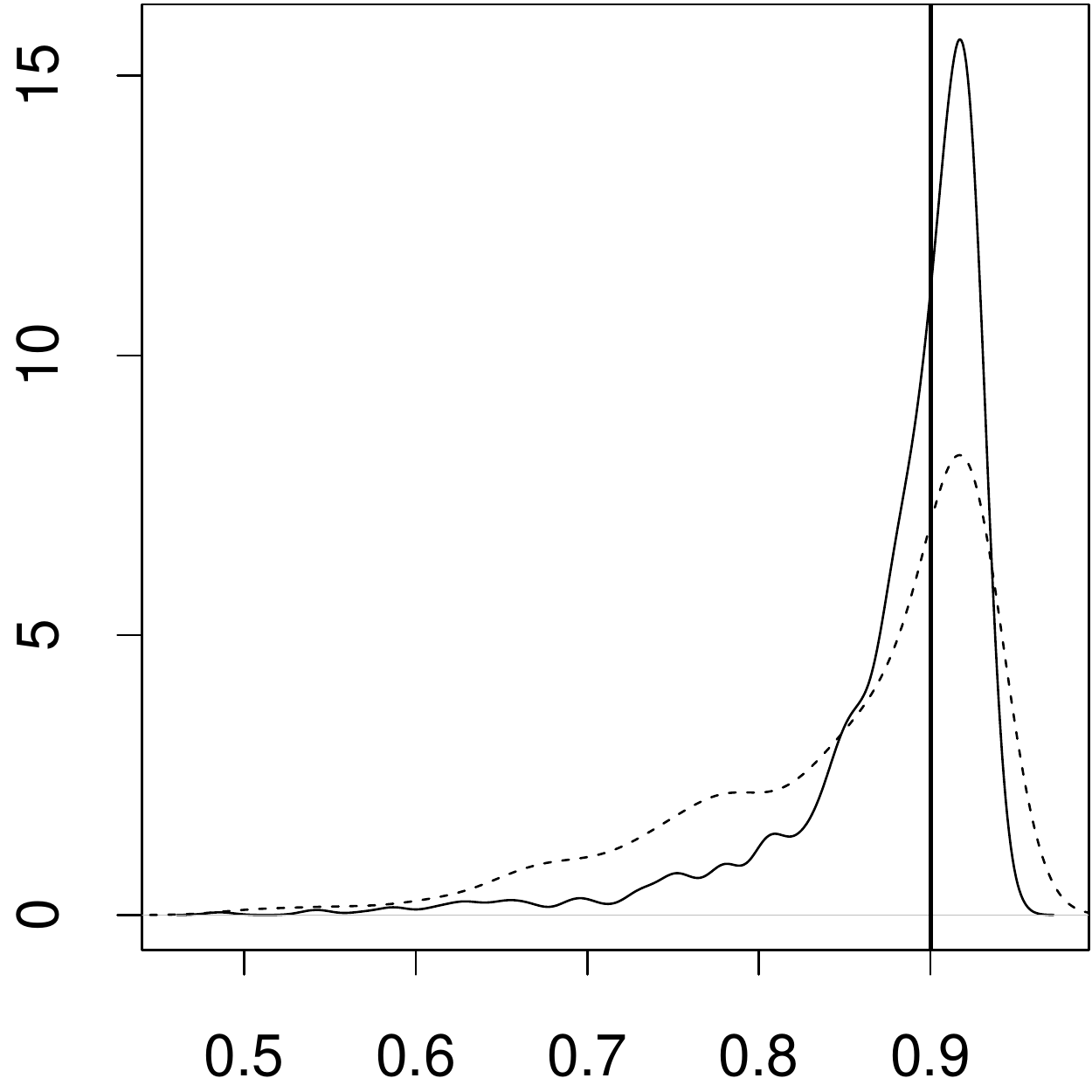}
  \caption{$\alpha_0=0.9$; $\rho = 0.25$}
\end{subfigure}
\begin{subfigure}{.24\textwidth}
  \centering
  \includegraphics[width=\textwidth]{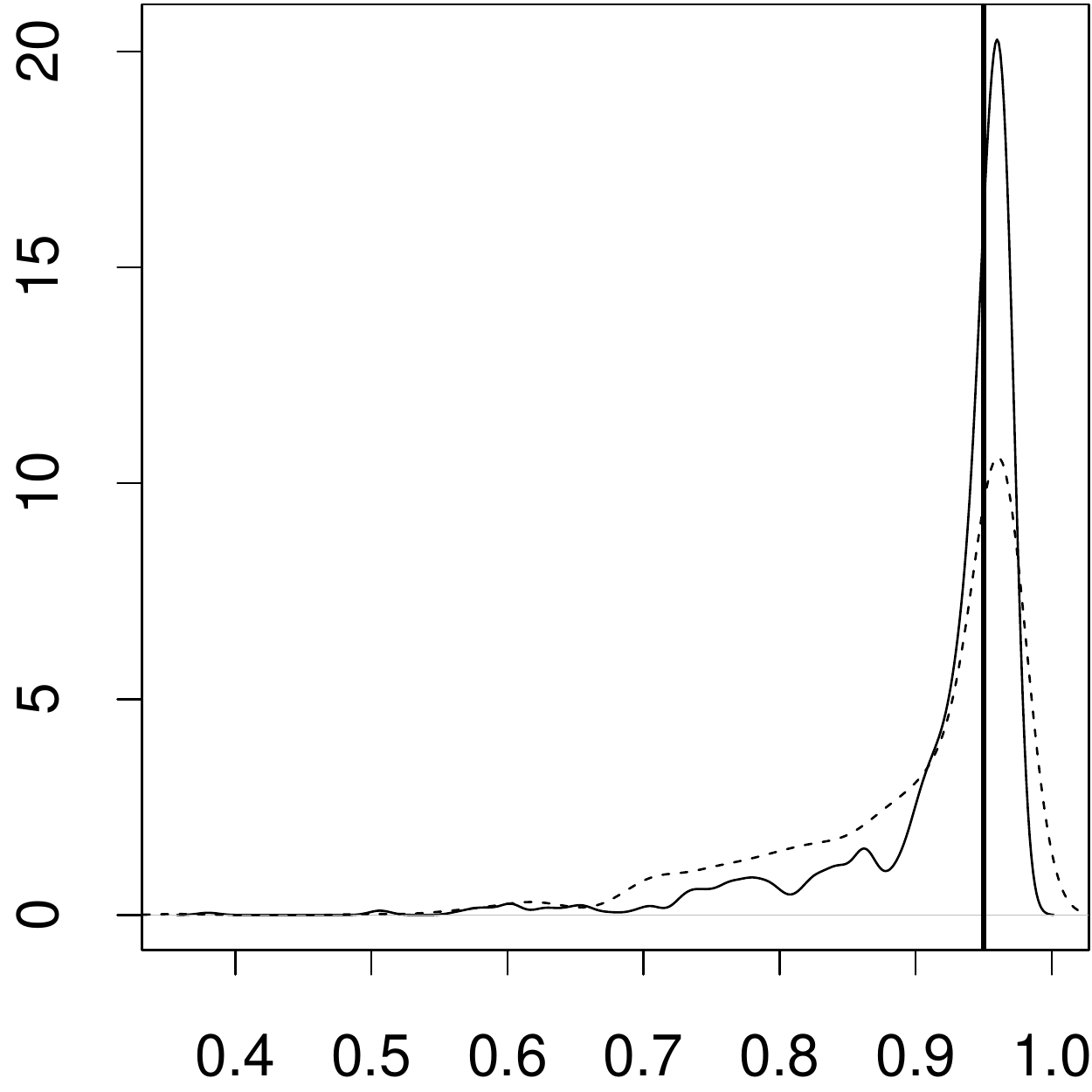}
  \caption{$\alpha_0=0.95$; $\rho = 0.25$}
\end{subfigure}

\begin{subfigure}{.24\textwidth}
  \centering
  \includegraphics[width=\textwidth]{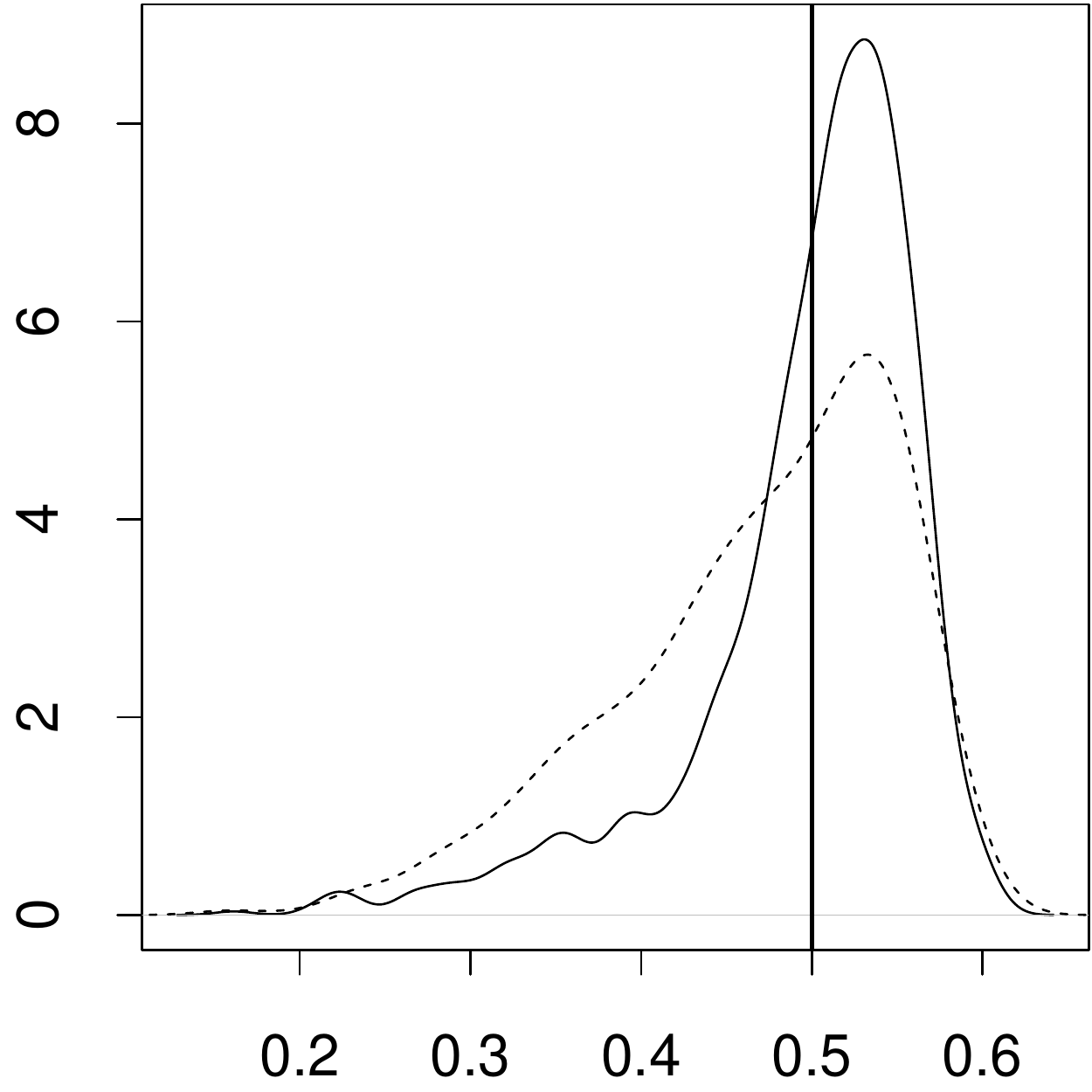}
  \caption{$\alpha_0=0.5$; $\rho = 0.5$}
\end{subfigure}
\begin{subfigure}{.24\textwidth}
  \centering
  \includegraphics[width=\textwidth]{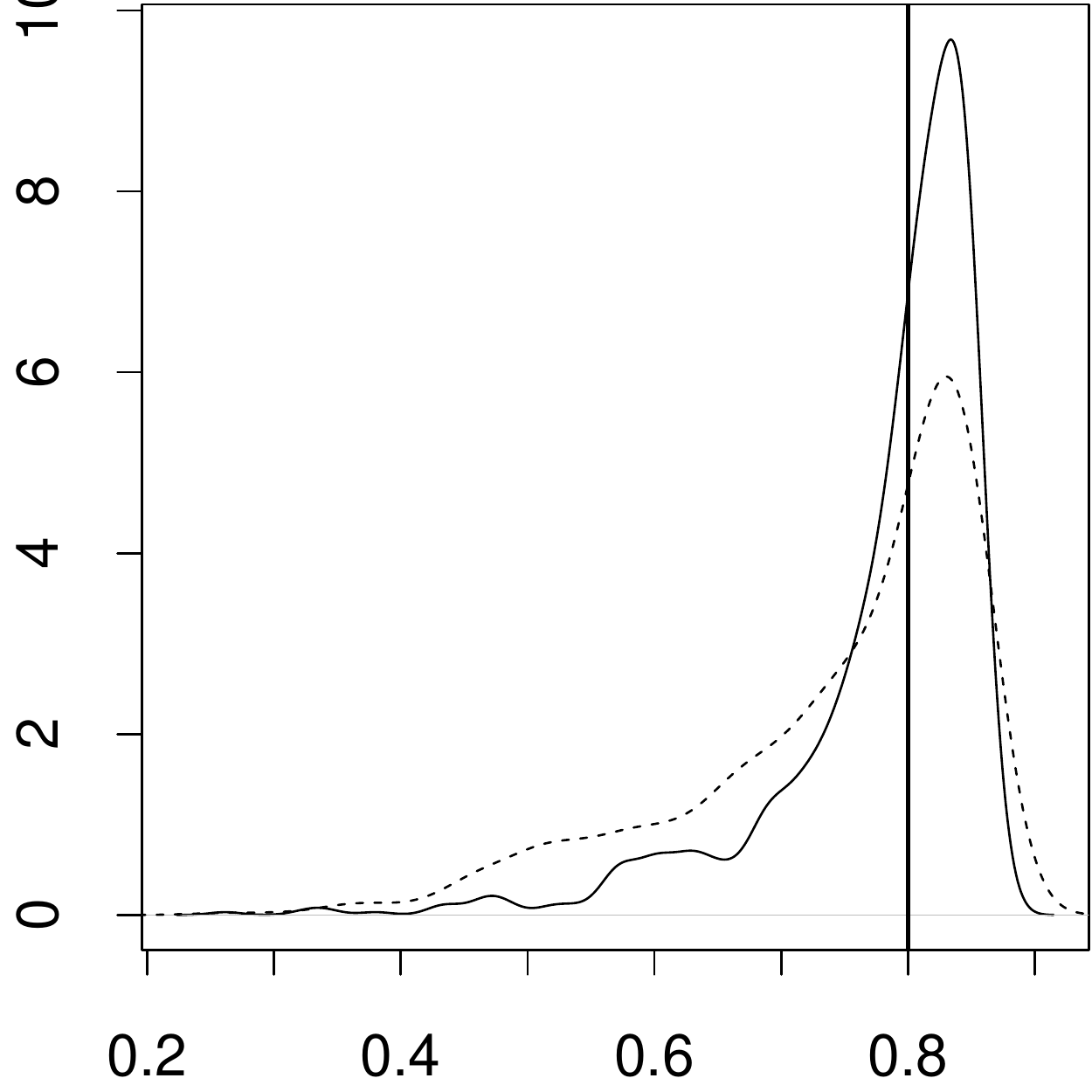}
  \caption{$\alpha_0=0.8$; $\rho = 0.5$}
\end{subfigure}
\begin{subfigure}{.24\textwidth}
  \centering
  \includegraphics[width=\textwidth]{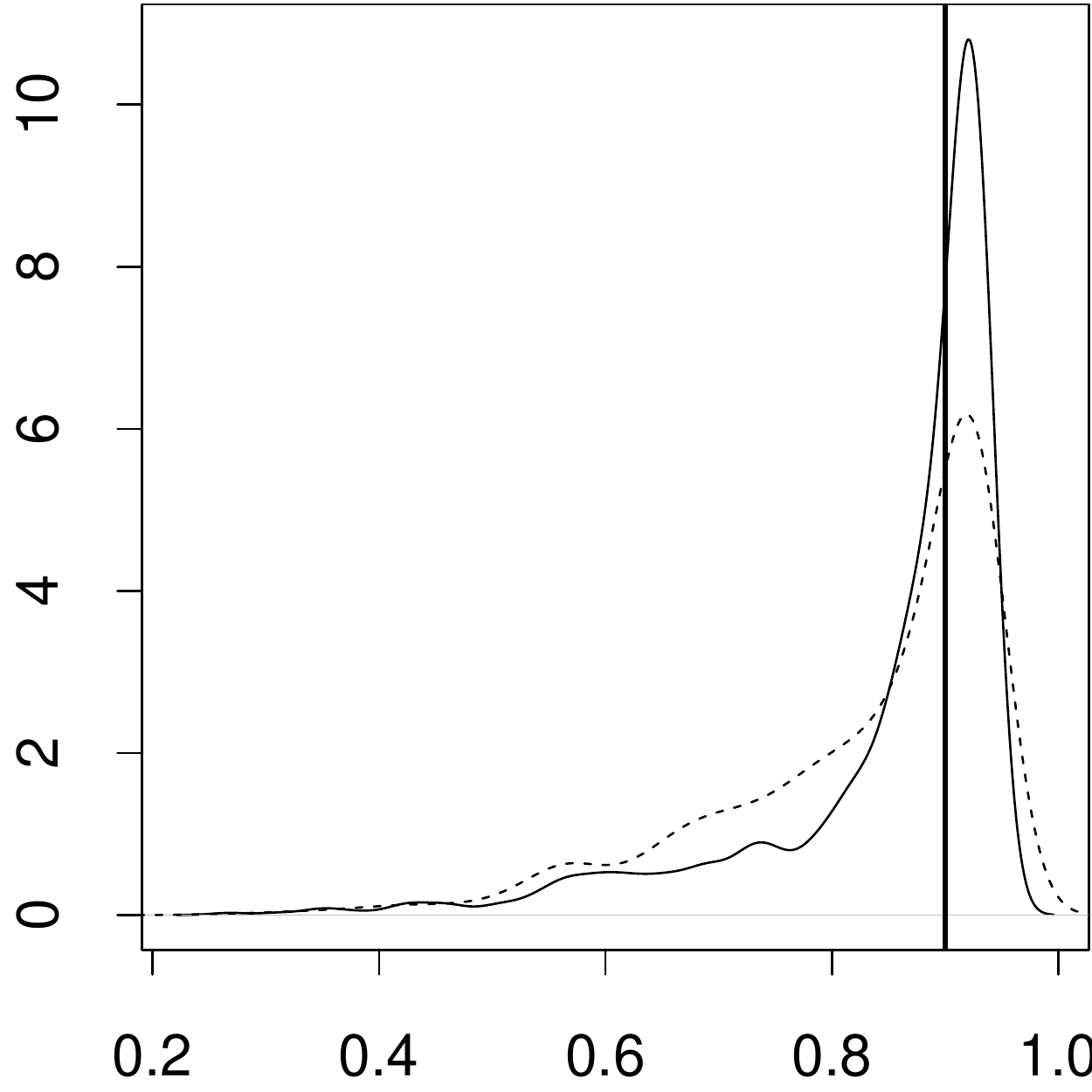}
  \caption{$\alpha_0=0.9$; $\rho = 0.5$}
\end{subfigure}
\begin{subfigure}{.24\textwidth}
  \centering
  \includegraphics[width=\textwidth]{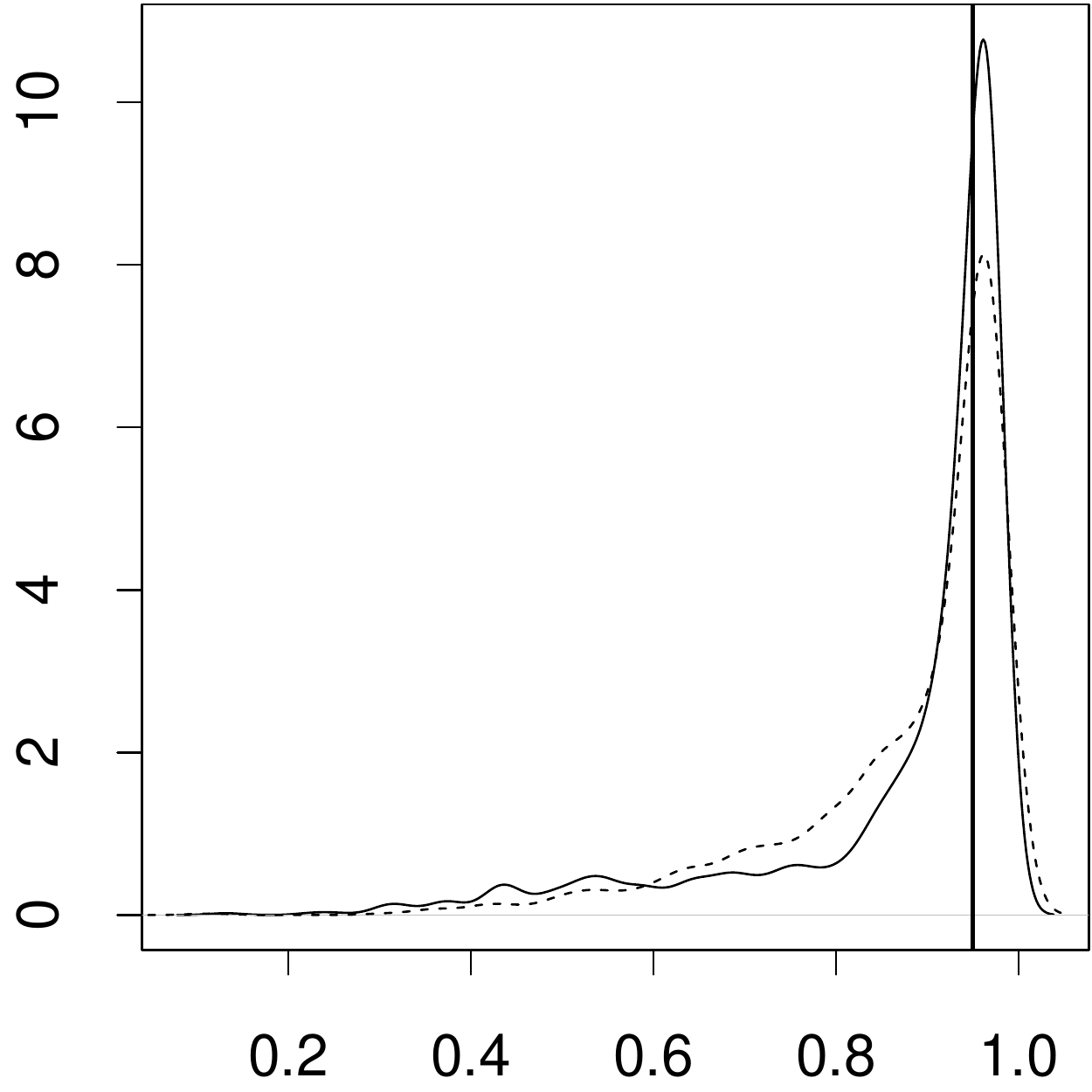}
  \caption{$\alpha_0=0.95$; $\rho = 0.5$}
\end{subfigure}

\begin{subfigure}{.24\textwidth}
  \centering
  \includegraphics[width=\textwidth]{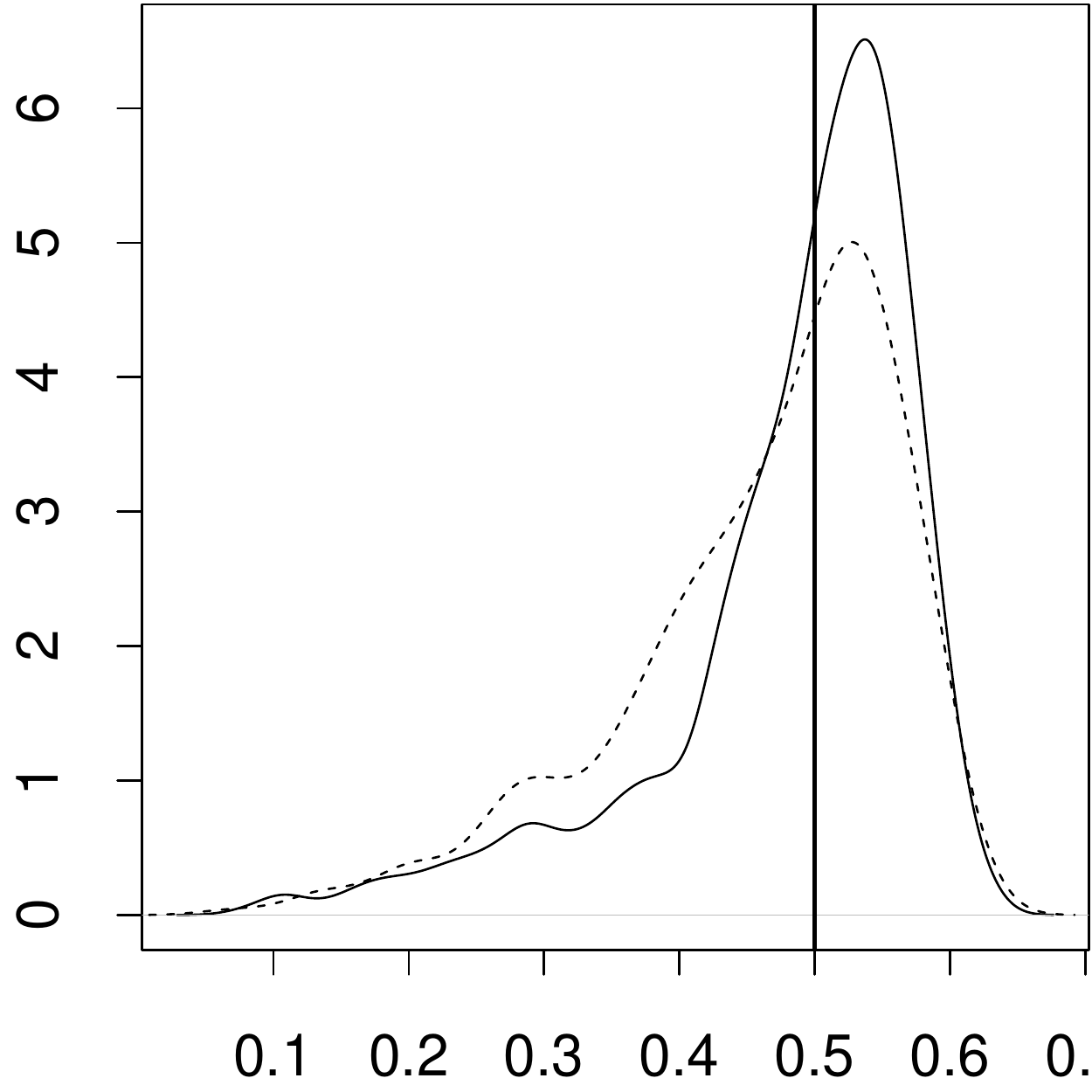}
  \caption{$\alpha_0=0.5$; $\rho = 0.75$}
\end{subfigure}
\begin{subfigure}{.24\textwidth}
  \centering
  \includegraphics[width=\textwidth]{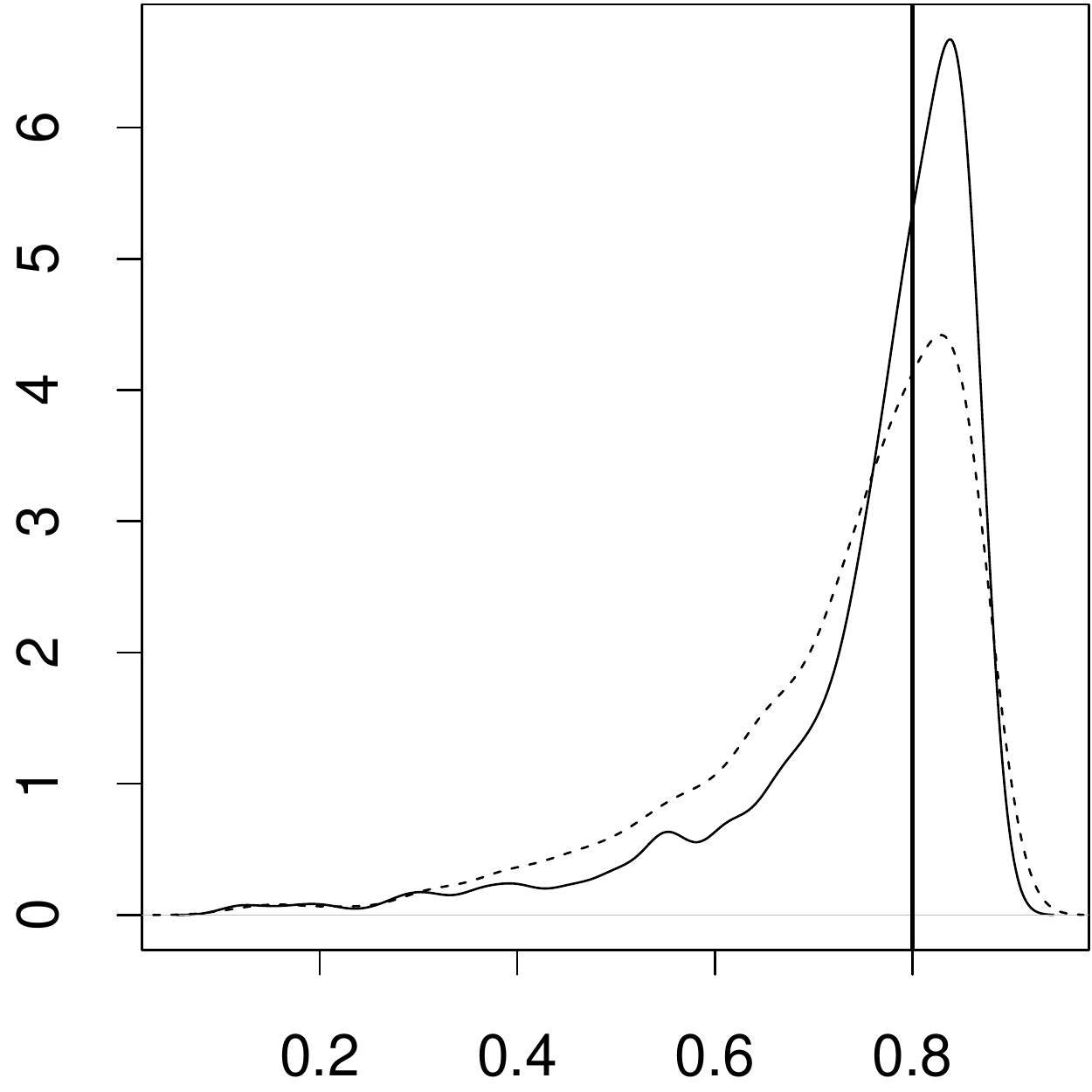}
  \caption{$\alpha_0=0.8$; $\rho = 0.75$}
\end{subfigure}
\begin{subfigure}{.24\textwidth}
  \centering
  \includegraphics[width=\textwidth]{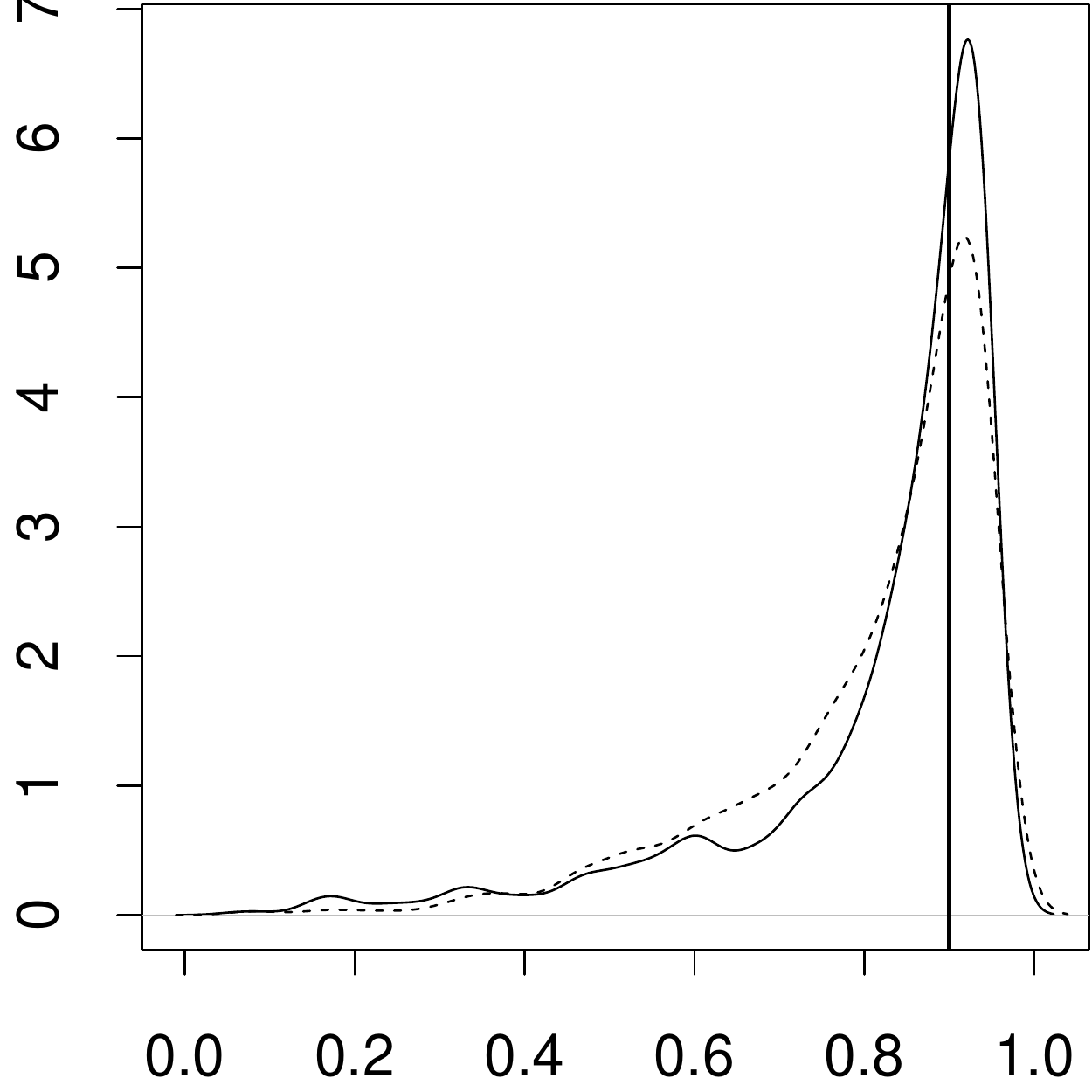}
  \caption{$\alpha_0=0.9$; $\rho = 0.75$}
\end{subfigure}
\begin{subfigure}{.24\textwidth}
  \centering
  \includegraphics[width=\textwidth]{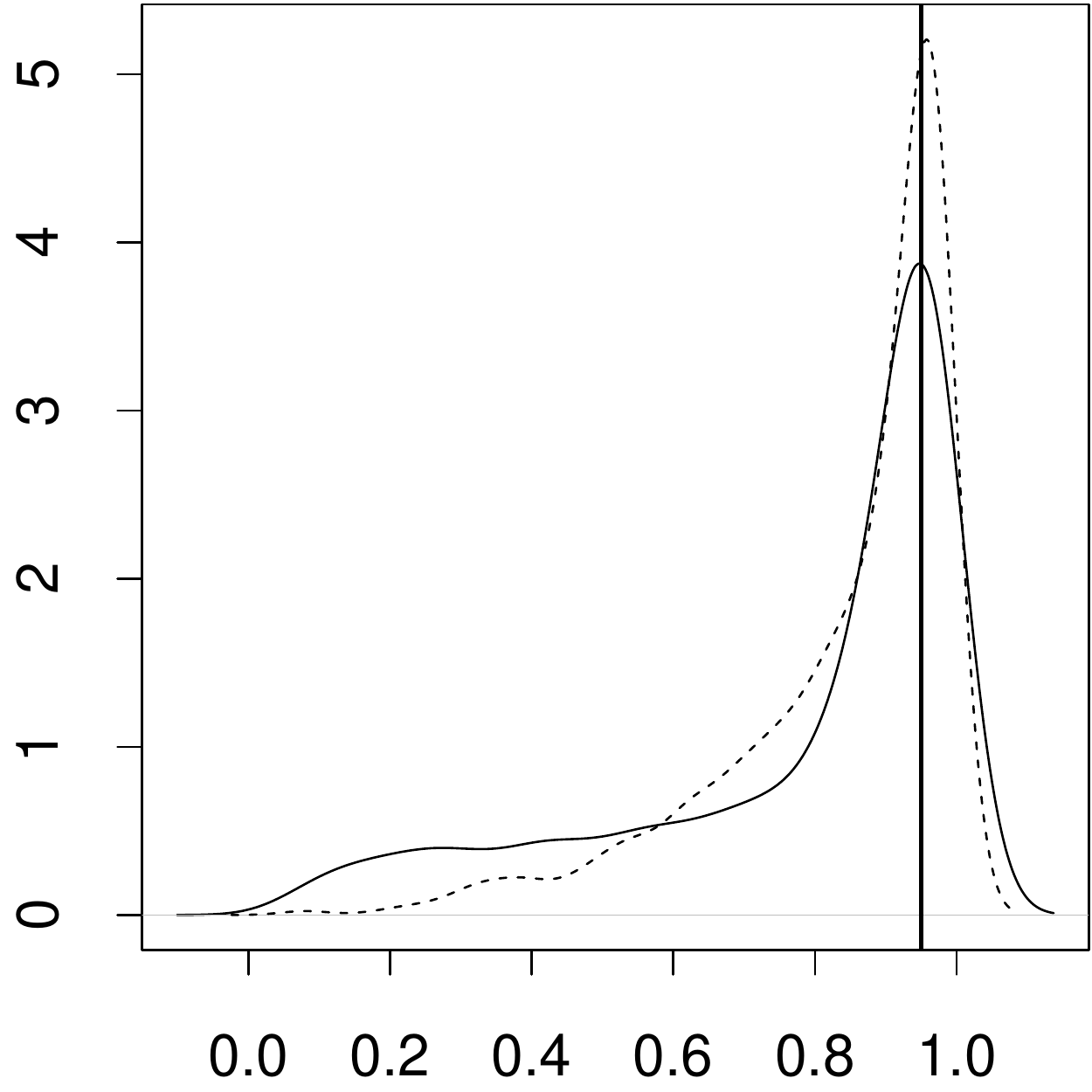}
  \caption{$\alpha_0=0.95$; $\rho = 0.75$}
\end{subfigure}
\caption{Density plots of the estimated $\alpha_0$ for $p$-values from one-sided t-tests ($G = 100$). The vertical line indicates the true $\alpha_0$, which is also marked below each figure, along with the within-group correlation coefficient. The solid lines are the densities of the posterior mean of $\alpha_0$, and the dashed lines are the densities of estimated $\alpha_0$ by fitting a convex decreasing density.}
\label{fig:osg100}
\end{figure}

\section{Summary, conclusions and further directions}
\label{sec:Disc}

We observed that a $k$-monotone density on the unit interval, like a $k$-monotone density on the positive half-line, also admits a mixture representation in terms of scaled beta densities. Such a density can be uniformly approximated by a finite mixture of the same kernels. We considered Bayesian procedures for making an inference on a $k$-monotone density by putting either a Dirichlet process prior or a finite mixture prior on the mixing distribution. We showed that under mild conditions on the true density and the prior distribution, the posterior contracts at the rate $(n/\log n)^{-k/(2k+1)}$, which is the optimal rate up to a polylogarithmic factor. Then we showed that even when $k$ is not known, simply by putting a prior on $k$, the corresponding mixture prior attains the same rate as if $k$ were known. We described an application for estimating the $p$-value distribution in a multiple-testing problem. We argued that it is very appealing to model the $p$-value density under the alternative as a $k$-monotone density, especially if $k$ is unspecified. The posterior contraction results ensure that the posterior distribution for the proportion of null hypotheses is consistent. We conducted a comprehensive simulation to check the comparative performance of the Bayesian method against those of the Grenander estimator and the maximum likelihood estimator for decreasing convex densities in finite samples. We found that the Bayesian procedure is overall the better performer. Further, the performance of the Bayesian procedure with unspecified $k$ is almost as good as the same with the correctly specified $k$, implying that the proposed adaptation scheme works for finite sample sizes. 

Observe that the kernel $\psi_k$ appearing in the characterization of $k$-monotone densities, as given in \eqref{repre_kmofunc},  makes sense for any positive $k$, even when $k$ is not a positive integer. This can be used to define a $k$-monotone density for a fractional $k >0$ through the representations \eqref{repre_kmofunc} and \eqref{kmonodens}, with $k$ replaced by its integer part in the first term of \eqref{repre_kmofunc} and \eqref{kmonodens}.  If Lemma~\ref{lemma:approx} can be generalized for a non-integer $k$, then the posterior contraction rate $(n/\log n)^{-k/(2k+1)}$ can be generalized for a general $k$ as well. Then the contraction rates $(n/\log n)^{-k/(2k+1)}$ prevail for all values of $k$ in the continuum, giving a continuous spectrum of rates similar to the smoothness regime. 

\section{Proofs}
\label{sec:proof}

\subsection{Proof of main theorems}
\begin{proof}[Proof of Theorem \ref{contraction}]
We apply the general theory of posterior contraction for i.i.d. observations as in Section~8.2 of \cite{ghosal2017fundamentals}. We need to obtain a lower bound for prior concentration in a Kullback-Leibler neighborhood of the true density and bound the size of a sieve in terms of the metric entropy so that the remaining part of the parameter space has an exponentially small prior probability. 

We verify the first condition given by (8.4) of \cite{ghosal2017fundamentals} at any $g_0$ with $\epsilon_n$ a constant multiple of $(n/\log n)^{-k/(2k+1)}$: 
\begin{align}
  -\log \Pi ( K(g_0,g)\le \epsilon_n^2,  V(g_0,g)\le \epsilon_n^2)\lesssim n \epsilon_n^2. 
  \label{KLprior}
\end{align}
By Lemma \ref{lemma:approx}, there exists
$g^{\ast}(x) = \sum_{j=0}^{k-1} \beta^{\ast}_j \psi_{j+1}(x,1) + \beta^{\ast}_k\sum_{l = 1}^{J^{\ast}} w_l^{\ast} \psi_k(x, \theta_l^{\ast}) \in \cD^k$ with $(w_l^{\ast}:l=1,\ldots,J^{\ast})\in\Delta_{J^{\ast}}$ and $(\theta_l^{\ast}: l=1,\ldots,J^{\ast})\in (0,1)^{J^{\ast}}$, such that 
$J^{\ast} \lesssim \epsilon_n^{-1/k}$ and
$\norm{g_0 - g^{\ast}}_{\infty} \lesssim \epsilon_n$.

First, we show that we can maintain the same approximation rate by restricting the choice to  $\theta_j^{\ast}\not\in (0,\epsilon^2_n)$ to ensure that $\norm{g_0 - g^{\ast}}_2\lesssim \epsilon_n$. 
Indeed, if there are $\theta_{l}^{\ast} < \epsilon_n^2$, we write $\bar{w} = \sum_{l:\theta_l^{\ast} < \epsilon_n^2} w_l^{\ast}$ and define
\begin{align*}
    g^{\dagger}(x) = \sum_{j=0}^{k-1} \beta^{\ast}_j \psi_{j+1}(x,1) + \beta^{\ast}_k\sum_{l:\theta^{\ast}_l \ge \epsilon_n^2} w_l^{\ast} \psi_k(x, \theta_l^{\ast}) + \beta_{k}^{\ast} \bar{w}\psi_k(x, \epsilon_n^2).
\end{align*}
It follows that $g^{\dagger}\in \cD^k$ and $g^{\dagger}(x) = g^{\ast}(x)$ for all $\epsilon_n^2\le x < 1$. Since, $g_0$ is bounded and $\norm{g^{\ast}-g_0}_{\infty} \lesssim \epsilon_n$, clearly $g^{\ast}$ is bounded. As 
$$g^{\ast}(0+) - g^{\dagger}(0+) = \beta^{\ast}_k \sum_{l:\theta_l < \epsilon_n^2} w^{\ast}_l(\psi_k(0,\theta^{\ast}_l) - \psi_k(0, \epsilon_n^2)) = \beta^{\ast}_k \sum_{l:\theta_l < \epsilon_n^2}w^{\ast}_l (\frac{k}{\theta_l} - \frac{k}{\epsilon_n^2})$$
is nonenagtive, $g^{\star}$ is bounded as well. Now $\norm{g^{\dagger} - g^{\ast}}_2 \lesssim \norm{\Ind_{(0,\epsilon_n^2)}}_2 = \epsilon_n$, and hence $\|g^{\dagger} - g_0\|_2 \lesssim \epsilon_n$. This assures that we can assume without loss of generality that $\theta_l^{\ast} \ge \epsilon_n^2$ for an $\LL_2$-approximation of $g_0$ within the order of $\epsilon_n$ using $J^{\ast} \lesssim \epsilon_n^{-1/k}$ mixture components.

To bound $K(g_0,g)$ and $V(g_0,g)$, we first bound the Hellinger distance. As $d_H(g_0,g) \le d_H(g_0,g^{\ast}) + d_H(g, g^{\ast})$, and $d_H(g_0,g^{\ast})\le \|1/g_0\|_\infty^{1/2} \|g-g_0\|_2 \lesssim \epsilon_n$, it suffices to bound $d_H(g, g^{\ast})$. 
Let $I_j=(\theta_l^{\ast}- \epsilon_n^4/2, \theta_l^{\ast} + \epsilon_n^4/2)$, $l=1,\ldots,J^{\ast}$.
We can assume, without loss of generality, that all spacings between $\theta_1^*,\ldots,\theta_{J^{\ast}}^*$ are bigger than $\epsilon_n^4$. 
Indeed, as $\min \theta_l^{\ast} \ge \epsilon_n^2$,
Lemma~\ref{lem:Psi} eliminates the need for placing multiple support points within $\epsilon_n^4$-neighborhood to control the $\LL_1$-distance within a constant multiple of $\epsilon_n^2$. This implies that $I_l$, $l=1,\ldots,J^{\ast}$, can be assumed to be pairwise disjoint.
Let $I_0=(0,1) \setminus (\cup_{l=1}^{J^{\ast}} I_l) $. 
Then 
\begin{align}
    \|g-g^{\ast}\|_1 &\le \sum_{j=0}^{k-1} |\beta_j - \beta_{j}^{\ast}|
    + \sum_{l=1}^{J^{\ast}} \int_{I_l} \|\psi_k(\cdot, \theta)-\psi_k(\cdot,\theta_l^*)\|_1 d Q(\theta) 
    \nonumber\\
    & \quad + \sum_{l=1}^{J^{\ast}} |Q(I_l)-w_l^*|+ Q(I_0).
     \label{bound3}
\end{align}
The fourth term in \eqref{bound3} is bounded by the third term because $Q(I_0)=1-\sum_{l=1}^{J^{\ast}} Q(I_l)$ and $\sum_{l=1}^{J^{\ast}} w_l^*=1$. 
Since $\|\psi_k(\cdot,\theta)-\psi_k(\cdot,\theta_j^*)\|_1 \le 2 |\theta-\theta_l^*|/\min \theta_l^{\ast} \le 2 \epsilon_n^2$ for any $\theta\in I_l$, and  $\sum_{l=1}^{J^{\ast}} Q(I_l)\le 1$, the second term is bounded by $\sum_{l=1}^{J^{\ast}} 2\epsilon_n^2 Q(I_l)\le 2\epsilon_n^2$. 
Therefore 
$\sum_{j=0}^{k-1} |\beta_j^{\ast} - \beta_j| \le \epsilon_n^2$ and $\sum_{l=1}^{J^{\ast}} |Q(I_l)-w_l^*|\le \epsilon_n^2$ together ensure that $\|g-g^{\ast}\|_1\lesssim \epsilon_n^2$, and hence $d_H(g,g^{\ast})\lesssim \epsilon_n$. 
Therefore, by the last part of Lemma~B.2 of \cite{ghosal2017fundamentals}, it follows that $\max (K(g_0,g),  V(g_0,g))\lesssim \epsilon_n^2$, so it suffices to lower-bound $\Pi\{ \sum_{j=0}^{k-1} |\beta_j - \beta_{0,j}| \le \epsilon_n^2, \sum_{l=1}^{J^{\ast}} |Q(I_l)-w_l^*|\le \epsilon_n^2\}$.
By Lemma~G.13 of \cite{ghosal2017fundamentals}, 
under the assumed conditions on the center measure $H$, for some constant $C, C'>0$, we have
$ \Pi(\sum_{j=1}^{k-1} |\beta_j-\beta_{0,j}|\le \epsilon_n^2) \gtrsim e^{-C k \log(1/\epsilon_n)}$ and  
$\Pi(\sum_{l=1}^{J^{\ast}} |Q(I_l)-w_l^*|\le \epsilon_n^2) \gtrsim e^{-C' J^{\ast} \log(1/\epsilon_n)}$.
As $\bm$ and $Q$ are independent, it follows that
\begin{align}
   -\log \Pi\big(\sum_{j=0}^{k-1} |\beta_j - \beta_{0,j}| \le \epsilon_n^2, \sum_{l=1}^{J^{\ast}} |Q(I_l)-w_l^*|\le \epsilon_n^2\big) \lesssim J^{\ast}\log(1/\epsilon_n).
   \label{prior estimate}
\end{align}
Equating $J^{\ast}\log (1/\epsilon_n) $ with $n\epsilon_n^2$, it is now immediate that \eqref{KLprior} holds for $\epsilon_n$ a constant multiple of $(n/\log n)^{-k/(2k+1)}$.  

Define a sieve $\cD^k_n = \{g\in\cD^k: g(0+) \le M_n\}$, where $M_n = \exp\{C n^{1/(2k+1)}(\log n)^{2k/(2k+1)}\}$, for a large positive constant $C>0$ to be determined later, denotes the upper cut-off. It suffices to verify the local entropy condition given by (8.10) of \cite{ghosal2017fundamentals}: For all $\epsilon\ge \epsilon_n$, 
\begin{align}
    \log \cN(\epsilon/2, \{g\in \cD^k_{n}: d_H (g,g_0)\le 2\epsilon\}, d_H)\lesssim n\epsilon_n^2 .
    \label{entropycond}
\end{align}
If $g\in \cD^k_{n}$, then $\norm{g-g_0}_1 \le 2d_H(g_0, g) \le 4\epsilon_n$. This implies that 
$$
\epsilon_n g(\epsilon_n)\le \int_0^{\epsilon_n} g(x)dx \le \int_0^{\epsilon_n} g_0(x)dx + 4\epsilon_n\le (g_0(0+)+4)\epsilon_n,
$$
giving $g(\epsilon_n) \le g_0(0+) + 4$.

Define $\cD^k_{n,1} = \{g\Ind_{(0,\epsilon_n)}:g\in\cD_{n}^k\}$
and $\cD^k_{n,2} = \{g\Ind_{[\epsilon_n,1]}:g\in\cD_{n,j}^k\}$.
By Lemma \ref{entropy}, we have
\begin{align*}
    &\log \cN(\epsilon_n/4, \cD^k_{n,1}, d_H)\lesssim \log(\epsilon_n M_n)^{{1}/({2k})}[(4+g_0(0+))\epsilon_n]^{{1}/{k}}(\epsilon_n/4)^{-{1}/{k}} 
\end{align*}
which can be bounded by a constant multiple of $n^{1/(2k+1)} (\log n)^{2k/(2k+1)} = n\epsilon_n^2$. 
By Lemma \ref{entropy} again, we have
\begin{align*}
\log \cN(\epsilon_n/4, \cD^k_{n,2}, d_H)
    \lesssim  \log g_0(0+) + 8^{{1}/({2k})}(\epsilon_n/4)^{-{1}/{k}},
\end{align*}    
which can be bounded by a constant multiple of  $\epsilon_n^{-{1}/{k}} \lesssim n\epsilon_n^2$.

Since $\cD^k_{n} \subset \cD^k_{n,1} + \cD^k_{n,2}$,  
\begin{align*}
    \log \cN(\epsilon_n/2, \cD^k_{n}, d_H) 
    \le \log \cN(\epsilon_n/4, \cD^k_{n,1}, d_H) + \log \cN(\epsilon_n/4, \cD^k_{n,2}, d_H) 
\end{align*}
is bounded by a multiple of $n\epsilon^2_n$,
verifying \eqref{entropycond}.

Next, we control the residual prior probability $\Pi(\cM^k\setminus\cM^k_n)$. Using the fact that $\psi_k(x,\theta)\le k/\theta$, we obtain the estimate  
\begin{align*}
 \Pi(g(0+) > M_n) 
& \le \Pi(k\int \theta^{-1} dQ(\theta)  >M_n )\\
& =  \Pi\big(\int_0^{2k/M_n} \theta^{-1} dQ(\theta)  + \int_{2k/M_n}^1 \theta^{-1} dQ(\theta) >M_n/k\big).
\end{align*}
Because always $\int_{2k/M_n}^1 \theta^{-1} d Q (\theta) \le M_n/(2k)$, the residual probability is at most 
\begin{align*}
\Pi\big(\int_0^{2k/M_n} \theta^{-1} d Q (\theta)> \frac{M_n}{2k}\big) & \le \frac{2k}{M_n} \E\int_0^{2k/M_n} \theta^{-1}d Q(\theta) 
 \lesssim \int_0^{2k/M_n} \theta^{-1+t_1}d\theta 
\end{align*}
respectively using Markov's inequality and the assumption on the base measure. As the last expression is bounded by a multiple of $M_n^{-t_1}$, it follows that the residual probability is at most $ \exp\{-C n\epsilon_n^2\}$, where $C>0$ can be chosen as large we please for $\epsilon_n=(n/\log n)^{-k/(2k+1)}$ by our choice of $M_n$.  This verifies all required conditions for the applicability of the general theory of posterior contraction rate with $\epsilon_n=(n/\log n)^{-k/(2k+1)}$. 
\end{proof}

\begin{proof}[Proof of Theorem \ref{thm:fmix}]
The proof is largely similar to that of Theorem \ref{contraction} by verifying the three conditions of Theorem~8.9 of \cite{ghosal2017fundamentals} for $\epsilon_n=(n/\log n)^{-k/(2k+1)}$. We highlight the differences in the following.

To estimate the prior concentration in the Kullback-Leibler neighborhood, we need to condition on the event $\{J = J^{\ast}\}$ in \eqref{KLprior}, where $J^{\ast}$ is such that the uniform approximation using a mixture $\psi_k$ with $J^*$ support points is within $n^{-k/(2k+1)}$. By  Lemma~\ref{lemma:approx}, we can assume that $J^{\ast}\le n^{1/(2k+1)}$. To bound the prior probability of $\{\theta_j\in I_j\}$, given $J = J^{\ast}$, observe that the condition assumed on on $H$, we have that 
$$\Pi(\theta_l\in I_l, 1\le l \le J|J = J^{\ast}) \gtrsim \epsilon_n^{4t_2 J^{\ast}} \ge \exp\{ C_1 J^{\ast} \log n\} \ge \exp\{ -C_2 n\epsilon_n^2\}, $$
where $C_1$ and $C_2$ are two positive constants. Along with the estimate 
$\Pi(J = J^{\ast}) =  \exp\{ - c J^{\ast}\log n + \log(n^c - 1)\} \ge \exp\{ - C_3n\epsilon_n^2\}$ 
 for some $C_3>0$, the required prior concentration rate is verified.

A sieve is chosen to be $\{g$ given by \eqref{kmonodens}: $g(0+)\le M_n\}$, where $M_n = \exp\{C n^{1/(2k+1)}(\log n)^{2k/(2k+1)}\}$ for a large positive constant $C$. 
The residual prior probability of the complement of the sieve is bounded by $\sum_{j=1}^{J_n} \Pi(J=j) \Pi (g(0+)\le M_n|J=j)$. Each term in the sum can be estimated as in the proof of Theorem~\ref{contraction}. It suffices to note that, as in the last theorem, $\E Q=H$ because the support points and their weights are independently distributed, these weights sum to one, and the support points are i.i.d. draws from $H$. 

The metric entropy bound for the sieve is obtained by applying Lemma \ref{entropy}, as before. 
\end{proof}

\begin{proof}[Proof of Theorem \ref{thm:finite}]
The proof is again obtained by applying the general theory on the posterior contraction rate in \cite{ghosal2017fundamentals} and the verification of the prior concentration condition proceeds in the same way as in the proof of Theorem \ref{contraction}. However, to derive the stated nearly parametric rate, the estimates of the metric entropy and the residual prior probability will have to be refined.

Suppose that $Q_0$ is supported on $J_0$ fixed points $\theta^0_1, \cdots, \theta^0_{J_0}$ in $(0,1)$ with weights $w_1^0,\ldots,w_{J_0}^0$. Let $c_0$ be the minimum of $\{\theta^0_l\}$.
Let $I_l = (\theta^0_l - c_0\epsilon_n^2/2, \theta^0_l + c_0\epsilon_n^2/2)$, for $l=1,\ldots,J_0$, and $I_0=(0,1) \setminus (\cup_{l=1}^{J_0} I_l)$. Clearly, 
$I_l$, $1\le l \le J_0$, are pairwise disjoint when $n$ large enough. Then $\|g-g_0\|_1 $ is bounded by the expression on the right side of \eqref{bound3} with $J_0$ replacing $J^{\ast}$. Therefore, following the same chain of arguments, the estimate of the prior probability of the Kullback-Leibler neighborhood reduces to 
$$\Pi \big(\sum_{j=0}^{k-1}|\beta_j - \beta_{0,j}| \le \epsilon_n^2\big) \times \Pi(J=J_0) \times \Pi \big(\sum_{l=1}^{J_0} |Q(I_l)-w_l^0|\le \epsilon_n^2| J = J_0\big).$$ 
As before, $-\log \Pi (\sum_{j=0}^{k-1}|\beta_j - \beta_{0,j}| \le \epsilon_n^2)\lesssim  k \log (1/\epsilon_n)$ and 
$$-\log \Pi\big(\sum_{l=1}^{J_0} |Q(I_l)-w_l^0|\le \epsilon_n^2|J=J_0\big)\lesssim  J_0 \log (1/\epsilon_n)\lesssim J_0\log n,$$ while the prior for $J$ satisfies $-\log \Pi(J=J_0) =cJ_0\log n - \log(n^c - 1)\lesssim J_0 \log n$. Hence the prior concentration condition \eqref{KLprior} holds for $\epsilon_n=\sqrt{(\max(J_0,k) \log n)/n}$.

Take $\bar{J}=L J_0$ for some $L>1$ to be determined later. 
We consider a sieve $\cL_n^k =\{g$ given by \eqref{kmonodens}: $Q =\sum_{l=1}^J w_l \psi_k(\cdot, \theta_l), (w_l: l\le J)\in\Delta_J, (\theta_l)\in (n^{-2},1)^J, l=1,\ldots,J, J\le \bar{J} \}=\cup_{J=1}^{\bar{J}}\cL_{n,\bar{J}}^k$, say. Then the residual prior is bounded by
\begin{align}
    \Pi(J > \bar{J}) + \Pi(\theta_l < n^{-2},1\le l\le J, J \le \bar{J}). \label{resisum}
\end{align}
The first term in the last display is bounded by 
\begin{align*}
\exp \{- cL\max\{k, J_0\} \log n + \log(n^c - 1)\} \le \exp\{ -C \max(k,J_0)\log n\},
\end{align*}
for some $C>0$. 
We also observe that
$$ \Pi(\bigcup_{J \le \bar{J}}\bigcup_{l \le J}\{ \theta_l < n^{-2} \})
     = \sum_{j=1}^{\bar{J}}  \sum_{l=1}^J \Pi(\theta_l < n^{-2}| J=j )
     \le {\bar{J}}^2   H((0, n^{-2})). 
$$
Using the inequality 
\begin{align}
    \int_0^{a} e^{-t/x} dx = \frac{a^2}{t} e^{-t/a} - \int_0^a \frac{2}{t} xe^{-t/x}dx \le \frac{a^2}{t} e^{-t/a}, \quad t>0,
    \label{residualprob}
\end{align}
the estimate above reduced to a constant multiple of $\bar{J}^2 n^{-4} e^{-t_3 n^2}$. Thus the residual prior probability is bounded by $\exp\{ -C \max(k, J_0) \log n\}$, where $C$ can be chosen as large as we please by making $L$ large enough. 

Next, we estimate the metric entropy of the sieve. For any two arbitrary elements $g_1,g_2$ of $\cL_{n,\bar{J}}^k$, 
$g_r(x)=\sum_{j=0}^{k-1} \beta_{r,j}\psi_{j+1}(x, 1) + \beta_{r,k}p_r(x)$, where $p_r(x) = \sum_{l=1}^J w_{r,j} \psi_k(x, \theta_l)$, $r=1,2$, 
observe that 
$\norm{g_1 - g_2}_1 \le \sum_{j=0}^{k-1} |\beta_{1,j} - \beta_{2,j}| + \norm{p_1-p_2}_1$. 
Using the estimate in Lemma~\ref{lem:Psi},  $\norm{p_1-p_2}_1$ can be bounded by  
\begin{eqnarray} 
   \lefteqn{ \sum_{l=1}^{J} w_{1,l} \|\psi_k (\cdot, {\theta_{1,l}})-\psi_k (\cdot, {\theta_{2,l}})\|_1+ \sum_{l=1}^J |w_{1,l}-w_{2,l}|} \nonumber \\ 
   &&\le 2n^2 \sum_{l\le J} |\theta_{1,l}-\theta_{2,l}|+\sum_{l=1}^J |w_{1,l}-w_{2,l}|.
    \label{l1bound2}
\end{eqnarray}
Thus if $\sum_{j=0}^{k-1} |\beta_{1,l}-\beta_{2,l}|\le \epsilon_n^2/2$, $\sum_{l=1}^J |w_{1,l}-w_{2,l}|\le \epsilon_n^2/4$ and $ \max\{|\theta_{1,l}-\theta_{2,l}|: l\le J\} \le \epsilon_n^2/(8n^2J)$, then $d_H(g_1,g_2)\le \|g_1-g_2\|_1^{1/2} \le \epsilon_n$. 
The $\epsilon_n^2/(4n^2J)$ covering number of  $[0,1]$ is bounded by $4n^2J\le 4 n^2 \bar{J}$.  
The $\epsilon_n^2/2$-covering number of $\Delta_{k}$ in the $\ell_1$ metric and the $\epsilon_n^2/4$-covering number of $\Delta_{J}$ in the $\ell_1$ metric are respectively bounded by $(10/\epsilon^2_n)^{k-1}$ and $(20/\epsilon^2_n)^{J-1}\le (20/\epsilon^2_n)^{\bar J}$ by Proposition~C.1 of \cite{ghosal2017fundamentals}. Then the $\epsilon$-Hellinger metric entropy of $\cL_n^k$ is bounded by 
$$\log (\bar{J}\times 4 n^2 \bar{J} \times (10/\epsilon^2_n)^{k-1} \times (20/\epsilon^2_n)^{\bar J}) \lesssim \bar{J}(\log n +\log (1/\epsilon)). $$
Thus for $\epsilon_n=\sqrt{(\max(J_0,k) \log n)/n}$, the entropy condition (8.5) of Theorem 8.9 of \cite{ghosal2017fundamentals} holds.
\end{proof}

\begin{proof}[Proof of Theorem \ref{thm:adp}]
    The first condition follows from the proof of Theorems \ref{contraction} and \ref{thm:fmix} upon conditioning on ${k=k_0}$ and using the fact that 
     $-\log \Pi(k=k_0) \lesssim k_0\log k_0 \le n\epsilon_n^2$ under (K1)  
    and $-\log \Pi(k=k_0) \lesssim k_0\log n \lesssim n\epsilon_n^2$ under (K2), where $\epsilon_n=(n/\log n)^{-k_0/(2k_0+1)}$.
    Then the first condition holds for both the Dirichlet process mixture prior (see the proof of Theorem \ref{contraction}) and the finite mixture prior (see the proof of Theorem \ref{thm:fmix}).

To verify the remaining two conditions for posterior contraction rate, we first address the finite mixture prior. For the metric entropy condition, consider the sieve 
        $\cL_n = \cup_{k=1}^{k_n} \cup_{j=1}^{J_n}\cL_{j,k}$,  
    where 
    $\cL_{J,k} = \big\{\sum_{j=0}^{k-1}\beta_j \psi_{j+1}(x,1) + \beta_k\sum_{l=1}^{J} w_l \psi_k(\cdot,\theta_l) \in \cD^k: (w_l: l\le J)\in \Delta_{J}, n^{-2} \le \theta_l\le 1, l\le J\big\}.$ 
    Following the corresponding part in the proof of Theorem \ref{thm:finite}, we know that the Hellinger metric entropy of $\cL_{J_n, k}$ is bounded by up to a constant multiple of $\max(k, J_n) \log n$. 
    Hence, the Hellinger entropy of $\cL_n$ can be bounded as follows,
    \begin{align*}
        \log \sum_{k=1}^{k_n}\cN(\epsilon_n, \cL_{J_n, k}, d_H)  \le \log k_n + \log \cN(\epsilon_n, \cL_{J_n, k_n}, d_H) \lesssim \max(k_n,J_n)\log{n}.
    \end{align*}
By choosing $J_n$ the integer part of $ L_1 (n/\log n)^{1/(2k_0 + 1)}$ for some $L_1>0$, we get $J_n\log n \asymp n\epsilon_n^2$ while maintaining $J_n > J^{\ast} \asymp \epsilon_n^{-k_0}$, where $J^{\ast}$ is the number of terms used to approximate to derive the estimate in the prior concentration condition. We choose $k_n$ the integer part of $L_2 (n/\log n)^{1/(2k_0 + 1)}$ for some $L_2>0$ to fulfil the entropy condition.

Now it remains to bound the residual prior probability of the sieve. Under both (K1) and (K2), the tail estimate of  $\Pi(k>k_n)\le e^{-L n\epsilon_n^2}$ is obtained with $L$ as large as we please by choosing $L_2$ sufficiently large. A similar argument applies for the tail $\Pi(J>J_n)$. Now  
$$\Pi(\theta_l < n^{-2}, 1\le l \le J, J \le J_n)
     \le \sum_{m=1}^{J_n} \sum_{l=1}^J \Pi(\theta_l\in (0, n^{-2})) \le J_n^2 n^{-4} e^{-t_3 n^2},$$
where the last inequality is due to \eqref{residualprob}.
This expression is also bounded by $e^{-L n\epsilon_n^2}$ where we can make $L>0$ as large as we wish. 
the proof for the finite mixture prior case now follows by an application of the general theory of posterior contraction.

For the Dirichlet process mixture prior, the sieve construction and the residual prior bounding need some modifications. We will elaborate on the differences in the following. Consider the sieve, $\cE_n = \cup_{k=1}^{k_n} \cE_{k,n}$ where 
\begin{multline*}
    \cE_{k,n} = \big\{\sum_{j=0}^{k-1} \beta_j \psi_{j+1}(x, 1) + \beta_k\sum_{l=1}^{\infty} w_l\psi_k(x, \theta_l) \in \cD^k: \\
    (w_j: j=1,2,\ldots)\in \Delta_{\infty}, \sum_{j>J_n} w_j < \epsilon_n^2, \theta_1,\ldots, \theta_{J_n} \in (n^{-2},1)\big\}.
\end{multline*}
The residual prior probability is bounded in the following,
\begin{align*}
    \Pi(\cE_n^c)\le \Pi(k>k_n) + \Pi\big(\sum_{j>J_n} w_j \ge \epsilon_n^2\big) + J_n H((0,n^{-2})).
\end{align*}
The first and the third terms can be bounded in a similar way as in the previous part. For the second term, by stick-breaking weight representation, $\sum_{l > J_n} w_l =  \prod_{l=1}^{J_n} (1-V_l)$, where $V_l \overset{i.i.d.}{\sim}\text{Beta}(1,a)$. Since $-\sum_{l=1}^{J_n}\log(1-V_l)$ is Gamma distributed with shape parameter $J_n$ and rate parameter $A$.
Then it follows that $ \Pi(\sum_{l>J_n} w_l \ge \epsilon_n^2) $ is given by 
\begin{align*}
    \P\big(-\sum_{l=1}^{J_n}\log(1-V_l) \le  2\log \epsilon_n^{-1}\big)\le \frac{(2a\log \epsilon_n^{-1})^{J_n}}{(J_n-1)!} \le \sqrt{\frac{J_n}{2\pi}}(2eAJ_n^{-1}\log \epsilon_n^{-1} )^{J_n}
\end{align*}
by Stirling's inequality for factorials. Choosing $J_n$ to be the integer part of $L_1 (n/\log n)^{1/(2k_0 + 1)}$ for some $L_1$, .
We can bound the expression by $e^{- L n \epsilon_n^2}$, where $L$ can be made as large as we like by choosing $L_1$ large enough. 

Following the same argument of the proof of Theorem \ref{thm:finite}, we obtain the bound in  \eqref{l1bound2} plus $\|\sum_{l\ge J_n} w_l \psi_k(\cdot, \theta_l) - \sum_{l\ge J_n} w'_l \psi_k(\cdot, \theta'_l)\|_1\le 2\epsilon_n^2$. 
Hence, the Hellinger metric entropy of the sieve $\cE_{k,n}$ can be bounded by a constant multiple of $J_n\log n$. The proof is concluded by following the same argument used for the finite mixture case.
\end{proof}

\begin{proof}[Proof of Theorem~\ref{cor:alp}]
    For two density functions $g_1,g_2$ from model \eqref{kmonodens}, we represent them as $g_1(u) = \alpha_1 + (1-\alpha_1)h_1(u)$ and $g_2(u) = \alpha_2 + (1-\alpha_2)h_2(u)$, separating out the constant component. We shall bound the $|\alpha_1-\alpha_2|$ by a constant multiple of the square root of the Hellinger distance between $g_1$ and $g_2$. This will lead to the conclusion in view of Theorems~\ref{contraction}, \ref{thm:fmix}, and \ref{thm:adp}. 
    
    For $\alpha_1 > \alpha_2$ and $\alpha_1 \le g_2(0+)$, the solution $s_0$ to the equation $g_2(u) = \alpha_1$ exists and is unique due to the strict convexity of $g_2$. Then 
    \begin{align}
        g_{2}(u) \ge g'_2(s_0)(u-s_0) + \alpha_1  \text{ for every } u\in(0,1);
        \label{bound g}
    \end{align}
    here $g'_2$ can be considered as either the right or the left derivative, both of which are well-defined for a convex function.
    As $g_2\ge \max\{g'_2(s_0)(u-s_0) + \alpha_1, 0\}$ for every $u\in(0,1)$, and          $|g'_2(s_0)|\ge (\alpha_1 - \alpha_2)/(1- s_0)$, the absolute slope of the line passing through two points on the graph of $g_2$,  $(s_0, \alpha_1)$ and $(1, \alpha_2)$, due to the convexity of $g_2$,  upon integrating \eqref{bound g}, it follows that 
    \begin{align*}
        1  \ge \int_0^{s_0} [g_2'(s_0)(u-s_0)+\alpha_1]du 
        \ge \frac{(\alpha_1 - \alpha_2)s_0^2}{2(1-s_0)} + \alpha_1 s_0
         \ge \frac{(\alpha_1 - \alpha_2)s_0}{2(1-s_0)}.
    \end{align*}
    This implies that $s_0 \le (1 + (\alpha_1 - \alpha_2)/2)^{-1}$, or equivalently, the bound $ 1-s_0 \ge (1 + 2/(\alpha_1 - \alpha_2))^{-1}\ge (\alpha_1-\alpha_2)/3$. Using these estimates 
    \begin{align*}
        \norm{g_1 - g_2}_1 & \ge \int_{s_0}^1(g_1(u) - g_2(u))d u 
         \ge \int_{s_0}^1(\alpha_1 - \frac{\alpha_1 -\alpha_2}{1-s_0}( s_0 - u) - \alpha_1)d u
         \end{align*} 
      is seen to be bounded below by    
       $(\alpha_1 - \alpha_2)(1-s_0)/2 \ge (\alpha_1 - \alpha_2)^2/6$. 
Thus    
\begin{align}
|\alpha_1 - \alpha_2| \le \sqrt{6\norm{g_1-g_2}_1} \le \sqrt{6 d_H(g_1,g_2)} .
\end{align}
Let $\bm{U}_n = \{U_1,\ldots, U_n\}$. Hence, for any $M_n\to \infty$, $\Pi(|\alpha-\alpha_0|>M_n (n/\log n)^{-k/(2(2k+1))}|\bm{U}_n)\le \Pi (d_H(g,g_0) > (M_n^2/6) (n/\log n)^{-k/(2k+1)}|\bm{U}_n)\to 0$ in probability under the true distribution. 
\end{proof}

\section*{Appendix: Proofs of the auxiliary results}
\label{appendix}

The following lemma is adapted from Theorem 3 of \cite{Gao2009}, which gives an upper bound of the Hellinger metric entropy of $k$-monotone functions.

\begin{lemma}\label{entropy}
Let $\mathcal{F}$ be the set of nonnegative $k$-monotone functions on an interval $[p,p+A]$ such that $f(p) \le B$ and $\int f \le M$ for any $f\in \mathcal{F}$, then
\begin{align*}
   \log \cN (2\epsilon,  \mathcal{F}, d_H) \le \log \cN_{[\,]}(\epsilon, \mathcal{F}, d_H) \lesssim |\log AB|^{{1}/({2k})}M^{{1}/{k}} \epsilon^{-{1}/{k}}.
\end{align*}
\end{lemma}

The following lemma gives a property of the kernel function $\psi_k(\cdot;\theta)$ that will be used in our analysis. 
\begin{lemma}\label{lem:Psi}
For $\psi_k(x, \theta)$ as defined in \eqref{psi_k}, we have
\begin{align} 
	\label{L1}
\|\psi_k(\cdot,\theta)-\psi_k(\cdot,\theta')\|_1 \le 2 (1-\min\{\theta,\theta'\}/\max\{\theta,\theta'\}).
\end{align}
\end{lemma}
\begin{proof}
Without loss of generality, assume that $0<\theta<\theta'$. 
Let $\delta_k(x) = \psi_k(x,\theta)-\psi_k(x, \theta')$. It is easy to see that \eqref{L1} holds for $k=1$. 
In fact, the equality in \eqref{L1} holds for $k=1$ by direct calculation using the fact that $\psi_k(\cdot, \theta)$ and $\psi_k(\cdot, \theta')$ are two densities of uniform distributions on $(0,\theta)$ and $(0,\theta')$ respectively. 
It is clear that $\delta_k(x)\equiv 0$ on $[\theta', 1)$.
If $k\ge 2$, we first claim that 
there exists a unique solution $x_0$ to the equation $\delta_k(x) = 0$ for $x\in(0,\theta')$.
Since $\delta_k(x) < 0$ for all $x\in[\theta, \theta')$, we restrict $x$ in $(0,\theta)$. 
Noting that $\delta_k(0) = k(\theta^{-1} - \theta'^{-1})>0$ and $\delta_k(\theta) = -k\theta'^{-1} ( 1-\theta/\theta')^{k-1} <0$, by the continuity of $\delta_k$, there exists at least one $x$ such that $\delta_k(x) = 0$. 
Additionally, by \eqref{psi_k}, for $x\in (0,\theta)$, the equation $\delta_k(x) = 0$ is equivalent to
     $\{({\theta' - x})/({\theta - x} )\}^{k-1} = ({\theta'}/{\theta})^k$.
Since the function on the left-hand side of the equation is strictly increasing in $x\in(0,\theta)$, there can be only one solution. 
By continuity again, $\delta_k(x) > 0 $ for $x\in (0, x_0)$ and $\delta_k(x) < 0$ for $x\in(x_0, \theta')$, and hence the $\LL_1$-distance in \eqref{L1} as 
\begin{align*}
     2\int_0^{x_0} \big[ \frac{k}{\theta}\big(1-\frac{x}{\theta}\big)^{k-1} - \frac{k}{\theta'} \big(1-\frac{x}{\theta'}\big)^{k-1} \big]dx 
     = 2\big[ \big(1-\frac{x_0}{\theta'}\big)^k-\big(1-\frac{x_0}{\theta}\big)^k \big].
\end{align*}
Rewriting the equation $\delta_k(x)=0$ as  
  $ ( 1 - {x}/{\theta} )^k = \displaystyle \frac{\theta-x}{\theta'-x}( 1 - {x}/{\theta'})^k,$
the expression for the $\mathbb{L}_1$-distance reduces to 
 $   2 ( 1- {x_0}/{\theta'})^{k-1}( 1-{\theta}/{\theta'})$. 
Since $0< x_0 < \theta < \theta'$, the bound $2(1-\theta/\theta')$ is immediate.
\end{proof}

\begin{proof}[Proof of Lemma \ref{lemma:charac}]
For sufficiency, note that $f$ given in \eqref{repre_kmofunc} is continuously differentiable up to the order $k-2$. The derivatives are given by
\begin{align}
\label{formula:derivatives}
    (-1)^j f^{(j)}(x) &= \sum_{l=j}^{k-1-j}\alpha_l l(l-1)\cdots(l-j+1)(1-x)^{l-j}\nonumber \\
    & \quad + \int_0^1 \frac{k(k-1)\cdots (k-j)}{t^{j+1}} \big(1-\frac{x}{t}\big)_+^{k-1-j} d \gamma(t),
\end{align}
for $j=0,1, \ldots, k-2$. It is obvious that the expression in \eqref{formula:derivatives} are nonnegative and nonincreasing as $\alpha_j\ge 0$. The derivative functions \eqref{formula:derivatives} are also convex as the summation and integral of convex functions are also convex.

To prove the necessity of the characterization, expand $f$ in a Taylor series at $a\in(0,1)$ using the fact that $f^{(k-2)}$ is absolutely continuous on $[a,x]$ or $[x,a]$: 
\begin{align*}
    f(x) = \sum_{j=0}^{k-2} \frac{(x-a)^j}{j!}f^{(j)}(a) + \int_a^x
    \frac{(x-t)^{k-2}}{(k-2)!}f^{(k-1)}(t) d t,
\end{align*}
where $f^{(k-1)}$ can be either the right or the left derivative function of the convex or concave function $f^{(k-2)}$, as they are different only on up to countably many points.
Note that $f^{(k-1)}$ is a monotone piecewise constant function with bounded variation on $[a,x]$ or $[x,a]$. Applying integration by parts on the reminder once and letting $a$ tend to $1$, we obtain
\begin{align*}
    f(x) & = \sum_{j=0}^{k-1} \frac{(x-1)^j}{j!}f^{(j)}(1-) - \int_x^{1-}
    \frac{(x-t)^{k-1}}{(k-1)!} d f^{(k-1)}(t)\\
    &= \sum_{j=0}^{k-1} \frac{(x-1)^j}{j!}f^{(j)}(1-) + \int_{0+}^{1-}
    \frac{k}{t}\left(1-\frac{x}{t}\right)_+^{k-1} d \gamma(t),
\end{align*}
where
$    \gamma(t) = \int_{0+}^t {(-1)^{k}u^{k}} d f^{(k-1)}(u)/k!$ for any $t>0$.
Note that $\gamma$ is nondecreasing as $(-1)^{k}f^{(k-1)}$ is nondecreasing. Then the characterization for $k$-monotone functions follows.

The characterization for $k$-monotone densities follows from proper normalization to a probability density function, which leads to a constraint, $\{\beta_j:0\le j\le k\} \in \Delta_{k+1}$. 
\end{proof}

\subsection{Discrete approximation}

We shall show the following shape-preserving approximation result using free knot splines. Indeed, we consider the shape-preserving free knot spline approximation of the functions in $\check{\cF}^k$, which can be transformed into the approximant of $\cF^k$ since if $f\in \check{\cF}^k\cap \LL_p$ and $s \in \cS_{N,k}$, then 
$f\circ \tau \in \cF^k\cap \LL_p$, $s\circ \tau \in \cS_{N,k}$, and
$\|f-s\|_p = \|f\circ \tau - s\circ \tau\|_p$ and vice versa.

However, Lemma \ref{approx} is not a consequence of Proposition \ref{kconvapprox}.
Since $\check{\cF}^k$ is a proper subset of $\cC^k$, for $f\in \check{\cF}^k \cap \LL_p$, typically $d_p(f, \cS_{N,k}\cap \check{\cF^k})\ge d_p(f, \cS_{N,k}\cap \cC^k )$. Thus we can not directly derive Lemma \ref{approx} from Proposition \ref{kconvapprox}. Fortunately, we can follow the argument of the proof of Proposition \ref{kconvapprox} by modifying some supporting lemmas therein, and the $k$-convex free knot spline approximant, constructed in \cite{Kopotun2003}, of a $k$-convex function is a free knot spline in $\check{\cF}^k$ provided the function to be approximated is not only $k$-convex but in $\check{\cF}^k$ as well.

We introduce some notations used in \cite{Kopotun2003} and also used in the following lemmas. For $(a,b)\subset (0,1)$, set 
\begin{align*}
\cC^k_{\ast}(a,b) &= \{f\in \cC^k: \max\{ |f^{(j)}(a+)|, |f^{(j)}(b-)|: j=0, 1, \ldots, k-1\} < \infty\}\\  
\check{\cF}^k_{\ast}(a,b) &= \{f\in \check{\cF}^k: \max\{ |f^{(j)}(b-)|: j=0, 1, \ldots, k-1\} < \infty\}.
\end{align*}
 Note that, if $f\in \check{\cF}^k$, $f^{(j)}$ is nonnegative and nondecreasing for every $j=0,1,\ldots, k-1$. Then 
 $\check{\cF}^k_{\ast}(a,b) \subset \cC^k_{\ast}(a,b)$ as $f^{(j)}(a)$ are all bounded up to the order $k-1$.  For $f\in \cC^k_{\ast}(a,b)$, let $\cC^k[f](a,b)$ stand for the set 
\begin{align*}
      \left\{ g\in\cC^k: 
    \begin{matrix}
    g^{(j)}(a+)=f^{(j)}(a+), 0\le j\le  k-2, 
    g^{(k-1)}(a+) \ge f^{(k-1)}(a+);\\
    g^{(j)}(b-)=f^{(j)}(b-), 0\le j\le k-2, 
    g^{(k-1)}(b-) \le f^{(k-1)}(b-).
    \end{matrix}
     \right\}.
\end{align*}

In view of the following lemma, we can assume, without loss of generality that $f$ has bounded derivatives up to the order $k-1$.

\begin{lemma}\label{lemma:finitederiv}
Let $f\in\check{\cF}^k\cap \LL_p(0,1)$ for $1\le p \le \infty$. Then for any $\epsilon>0$, there exists $f_{\epsilon}\in \check{\cF}^k_{\ast}(0,1)$ such that $\|f-f_{\epsilon}\|_p < \epsilon$.
\end{lemma}

\begin{proof}[Proof of Lemma~\ref{lemma:finitederiv}]
    This proof follows from some modifications from \cite[Lemma 4.4]{Kopotun2003}. First, we construct $f_{\epsilon}$ in the following. For $u\in (0,1)$, denote the Taylor polynomial of $f$ up to the degree $k-1$ at $u$ as
        $T_{u}(x) = \sum_{l=0}^{k-1} {f^{(l)}(u+)}{l!}(x-u)^l/l!$. 
    For some $\delta \in (0,1)$, define 
        $f_{\epsilon}(x)$ to be $f(x)$ if $x\in [0,1-\delta]$ and 
        $T_{1-\delta}(x)$ if $x\in [1-\delta,1]$. By the proof of \cite[Lemma 4.4]{Kopotun2003}, we know that $\|f-f_{\epsilon}\|_p \to 0$ as $\delta\to 0$. To conclude the proof, it suffices to show that $f_{\epsilon}\in \check{\cF}^k$. By definition, $f^{(j)}((1-\delta)+)$ are all nonnegative for $j=0,1,\ldots, k-1$. Then on $[1-\delta, 1]$, the derivative values $T^{(j)}_{1-\delta}$ are nonnegative and nondecreasing for $j=0,1,\ldots, k-1$. As $T_{1-\delta}$ is the Taylor polynomial of degree $k-1$, obviously, $f_{\epsilon}^{(j)}$ are nonnegative and nondecreasing on $[0,1]$ up to the order $k-1$. Then we can say that $f_{\epsilon}^{(j)}$ is nonnegative, nondecreasing, and convex on $[0,1]$ for every $j=0,1,\ldots, k-1$, that is, $f_{\epsilon} \in \check{\cF}^k$. 
\end{proof}

\begin{lemma}
\label{approx}
For any $f\in \check{\cF}^k \cap \LL_p(0,1)$, 
there exists some $s\in \cS_{C_k N,k}\cap \check{\cF}^k $ such that
$\|f - s\|_p \leq C_{k,p} d_p(f, \cS_{N, k})$ and $s^{(j)}(0+)=f^{(j)}(0+)$ for $j=0,1,\ldots, k-2$.
\end{lemma}

\begin{proof}[Proof of Lemma \ref{approx}]
    In view of Lemma \ref{lemma:finitederiv}, we can assume that $f\in \check{\cF}^k_{\ast}(0,1)\subset \cC^k_{\ast}(0,1)$.  Following the proof of \cite[Theorem 1]{Kopotun2003}, we can construct a spline $s\in \cS_{C_k N,k}$ such that $s\in \cC^k[f](0,1)$ and 
$\|f-s\|_p \leq C_{k, p} d_p(f, \cS_{N,k})$.

    Next, we will show that $s\in \check{\cF}^k$.
    As $s\in \cC^k[f](0,1)$, $s^{(k-1)}$ is of piecewise constant and nondecreasing due to the convexity of $s^{(k-2)}$ and, moreover, $s^{(k-1)}(0+) \geq f^{(k-1)}(0+) \ge 0$. Then $s^{(k-1)}$ is nonnegative.
    As $s^{(k-2)}(0+) = f^{(k-2)}(0+) \ge 0$, $s^{(k-2)}$ is nonnegative, nondecreasing and convex. Noting that $s^{(j)}(0) = f^{(j)}(0) \ge 0$ for $j=0,1,\ldots, k-3$, by induction, it is easy to see that $s^{(j)}$ is nonnegative, nondecreasing, and convex for every $j=0, 1, \ldots, k-3$ since $s^{(j+1)}$ is nonnegative. Hence, we conclude that $s\in\check{\cF}^k$.
\end{proof}

\begin{proof}[Proof of Lemma~\ref{lemma:approx}]
Observe that $g\in \LL_{\infty}(0,1)$ as $|g^{(k-1)}(0+)| < \infty$.
Note that $\cD^k$ is a subclass of $\cF^k$. By Lemma \ref{approx}, for any $g\in D^k$, there exists a $\tilde{g}\in \cS_{N, k}\cap \cF^k$ such that $\|g-\tilde{g}\|_p\le C d_p(g, \cS_{N,k})$ and $\tilde{g}(1-) = g(1-) = 0$ for $j=0, 1, \ldots, k-2$. By Lemma \ref{lemma:charac}, $\tilde{g}(x) = \alpha_{k-1} (1-x)^{(k-1)}/(k-1)! + \int_0^1 kt^{-1} (1-x/t)^{k-1}_+ d \gamma(t)$ for some nonnegative $\alpha_{k-1}$ and some nondecreasing function $\gamma$ on $(0,1)$. The first polynomial term can be incorporated into the integral by defining $\gamma(1) = \gamma(1-) + \alpha_{k-1}/k!$. Since $\tilde{g}$ is a piecewise polynomial of degree $k-1$, we conclude that $\gamma$ is piecewise constant function with jumps at the knots of the spline.
Let $g_N=\tilde{g}/\int \tilde{g}$ satisfying the structure requirement. Now $g_N$ maintains the  desired approximation rate: 
\begin{align}\label{ine:inf}
    \|g-g_N\|_{\infty}\leq \|g-\tilde{g}\|_{\infty} + \big\|\tilde{g} - \frac{\tilde{g}}{\int \tilde{g}}\big\|_{\infty}\le  \frac{1+\|g\|_{\infty}}{1-\|g-\tilde{g}\|_{\infty}}\|g-\tilde{g}\|_{\infty}.
\end{align}
By \cite[Theorem 12.4.5]{DeVore1993}, $\|g-\tilde{g}\|_{\infty} \leq C_{k,g}N^{-k}$, provided $|g^{(k-1)}(0+)|<\infty$. Thus the right-hand side of \eqref{ine:inf} can be further bounded by $C'_{k,g}N^{-k}$.
\end{proof}

\section*{Acknowledgments}
The authors are deeply indebted to Professor Bodhisattva Sen for bringing the present problem to the authors' attention and for pointing out to several key references.

\bibliographystyle{plain}
\bibliography{kMono}

\end{document}